\documentclass[12pt]{article}
\usepackage{amsmath,amsfonts,amssymb,subeqnarray,indent,url,graphicx}
\long\def\beginpgfgraphicnamed#1#2\endpgfgraphicnamed{\includegraphics{#1}}

\newcommand{\ser}{\ensuremath\mathbin{\bowtie}}
\newcommand{\serq}{\ensuremath\mathbin{\bowtie_q}}

\newcommand{\tvser}[1]{\ensuremath\mathbin{\bowtie^#1}}
\newcommand{\tvserq}[1]{\ensuremath\mathbin{\bowtie^#1_q}}

\newcommand{\pll}{\ensuremath\mathbin{\|}}
\newcommand{\pllq}{\ensuremath\mathbin{\|_q}}

\newcommand{\tvpll}[1]{\ensuremath\mathbin{\|^#1}}
\newcommand{\tvpllq}[1]{\ensuremath\mathbin{\|^#1_q}}

\numberwithin{equation}{section}  % This gets equations numbered by section

\begin{document}

\title{\vspace*{-2cm}
      Linear bound in terms of maxmaxflow \\
      for the chromatic roots of series-parallel graphs}

\author{
 \\
 {\small Gordon F.~Royle}                                    \\[-2mm]
 {\small\it School of Mathematics and Statistics}  \\[-2mm]
 {\small\it University of Western Australia} \\[-2mm]
 {\small\it 35 Stirling Highway} \\[-2mm]
 {\small\it Nedlands, WA 6009, AUSTRALIA}                         \\[-2mm]
 {\small\tt GORDON.ROYLE@UWA.EDU.AU}   \\[5mm]
 {\small Alan D.~Sokal\thanks{Also at Department of Mathematics,
        University College London, London WC1E 6BT, United Kingdom.}}  \\[-2mm]
 {\small\it Department of Physics}       \\[-2mm]
 {\small\it New York University}         \\[-2mm]
 {\small\it 4 Washington Place}          \\[-2mm]
 {\small\it New York, NY 10003 USA}      \\[-2mm]
 {\small\tt SOKAL@NYU.EDU}               \\[-2mm]
 {\protect\makebox[5in]{\quad}}  % To force authors' names to be written
				  %   vertically, one above another.
				  % (\author seems to put them side-by-side
				  %   if there is room.)
 \\
}

%\date{\today \ {\bf Remember to put final date!!!!}}
\date{July 5, 2013}  % ADD FINAL DATE HERE!!!

\maketitle
\thispagestyle{empty}   % Suppress page number on front page.

\begin{abstract}
We prove that the (real or complex) chromatic roots of
a series-parallel graph with maxmaxflow $\Lambda$ lie
in the disc $|q-1| < (\Lambda-1)/\log 2$.
More generally, the same bound holds for the (real or complex) roots
of the multivariate Tutte polynomial
when the edge weights lie in the ``real antiferromagnetic regime''
$-1 \le v_e \le 0$.
This result is within a factor $1/\log 2 \approx 1.442695$
of being sharp.
\end{abstract}


\bigskip
\noindent
{\bf Mathematics Subject Classification (MSC 2010) codes:}
05C31 (Primary);
05A20, 05C15, 05C99, 05E99, 30C15, 37F10, 37F45, 82B20 (Secondary).

\bigskip
\noindent
{\bf Abbreviated title:}  Bound on chromatic roots in terms of maxmaxflow

\clearpage

\newtheorem{defin}{Definition}[section]
\newtheorem{definition}[defin]{Definition}
\newtheorem{prop}[defin]{Proposition}
\newtheorem{proposition}[defin]{Proposition}
\newtheorem{lem}[defin]{Lemma}
\newtheorem{lemma}[defin]{Lemma}
\newtheorem{guess}[defin]{Conjecture}
\newtheorem{ques}[defin]{Question}
\newtheorem{question}[defin]{Question}
\newtheorem{prob}[defin]{Problem}
\newtheorem{problem}[defin]{Problem}
\newtheorem{thm}[defin]{Theorem}
\newtheorem{theorem}[defin]{Theorem}
\newtheorem{cor}[defin]{Corollary}
\newtheorem{corollary}[defin]{Corollary}
\newtheorem{conj}[defin]{Conjecture}
\newtheorem{conjecture}[defin]{Conjecture}

\newtheorem{pro}{Problem}
\newtheorem{clm}{Claim}
\newtheorem{con}{Conjecture}

%
% For examples (numbered by sections, separately from theorems etc)
% Note that the \label does not include the section number;
%   that has to be put in by hand in the \ref.
%
\newcounter{example}[section]
\newenvironment{example}%
{\refstepcounter{example}
\bigskip\par\noindent{\bf Example \thesection.\arabic{example}.}\quad
}%
%{\hbox{\kern1pt\vrule height6pt width4pt depth1pt\kern1pt}}
%{\hbox{\hskip 6pt\vrule width6pt height7pt depth1pt \hskip1pt}}
{\quad $\Box$}
\def\bexam{\begin{example}}
\def\eexam{\end{example}}

\renewcommand{\theenumi}{\alph{enumi}}
\renewcommand{\labelenumi}{(\theenumi)}
\def\prf{\par\noindent{\bf Proof.\enspace}\rm}
\def\rmk{\par\medskip\noindent{\bf Remark.\enspace}\rm}

\newcommand{\be}{\begin{equation}}
\newcommand{\ee}{\end{equation}}
\newcommand{\<}{\langle}
\renewcommand{\>}{\rangle}
\newcommand{\widebar}{\overline}
\def\reff#1{(\protect\ref{#1})}
\def\spose#1{\hbox to 0pt{#1\hss}}
\def\ltapprox{\mathrel{\spose{\lower 3pt\hbox{$\mathchar"218$}}
\raise 2.0pt\hbox{$\mathchar"13C$}}}
\def\gtapprox{\mathrel{\spose{\lower 3pt\hbox{$\mathchar"218$}}
\raise 2.0pt\hbox{$\mathchar"13E$}}}
\def\textprime{${}^\prime$}
\def\proof{\par\medskip\noindent{\sc Proof.\ }}
\newcommand{\qed}{\quad $\Box$ \medskip \medskip}
\def\proofof#1{\bigskip\noindent{\sc Proof of #1.\ }}
\def\half{ {1 \over 2} }
\def\third{ {1 \over 3} }
\def\twothird{ {2 \over 3} }
\def\smfrac#1#2{{\textstyle{#1\over #2}}}
\def\smhalf{ \smfrac{1}{2} }
\newcommand{\real}{\mathop{\rm Re}\nolimits}
\renewcommand{\Re}{\mathop{\rm Re}\nolimits}
\newcommand{\imag}{\mathop{\rm Im}\nolimits}
\renewcommand{\Im}{\mathop{\rm Im}\nolimits}
\newcommand{\sgn}{\mathop{\rm sgn}\nolimits}
\def\hboxscript#1{ {\hbox{\scriptsize\em #1}} }
\newcommand{\arctanh}{\mathop{\rm arctanh}\nolimits}
\newcommand{\arccoth}{\mathop{\rm arccoth}\nolimits}

\newcommand{\restrict}{\upharpoonright}
\renewcommand{\emptyset}{\varnothing}

\def\Z{{\mathbb Z}}
\def\ZZ{{\mathbb Z}}
\def\R{{\mathbb R}}
\def\C{{\mathbb C}}
\def\Cbar{{\overline{\mathbb C}}}
\def\Chat{{\widehat{\mathbb C}}}
\def\Ctilde{{\widetilde{\mathbb C}}}
\def\CC{{\mathbb C}}
\def\N{{\mathbb N}}
\def\NN{{\mathbb N}}
\def\Q{{\mathbb Q}}
\def\K{{\mathbb K}}

\newcommand{\scra}{{\mathcal{A}}}
\newcommand{\scrb}{{\mathcal{B}}}
\newcommand{\scrc}{{\mathcal{C}}}
\newcommand{\scrf}{{\mathcal{F}}}
\newcommand{\scrg}{{\mathcal{G}}}
\newcommand{\scrh}{{\mathcal{H}}}
\newcommand{\scrl}{{\mathcal{L}}}
\newcommand{\scro}{{\mathcal{O}}}
\newcommand{\scrp}{{\mathcal{P}}}
\newcommand{\scrr}{{\mathcal{R}}}
\newcommand{\scrs}{{\mathcal{S}}}
\newcommand{\scrt}{{\mathcal{T}}}
\newcommand{\scrv}{{\mathcal{V}}}
\newcommand{\scrw}{{\mathcal{W}}}
\newcommand{\scrz}{{\mathcal{Z}}}
\newcommand{\scrbt}{{\mathcal{BT}}}
\newcommand{\scrbf}{{\mathcal{BF}}}

\newcommand{\bv}{ {\bf v} }
\newcommand{\bw}{ {\bf w} }
\newcommand{\bzero}{ {\bf 0} }
\newcommand{\bone}{ {\bf 1} }
\def\qvbf{{q, \bv}}
\def\veff{{v_{{\rm eff}}}}
\newcommand\zgst{Z_G^{(s \leftrightarrow t)}}
\newcommand\zgsnt{Z_G^{(s \kern1pt\not\kern-1pt\leftrightarrow t)}}
\newcommand{\sfG}{{\sf G}}
\newcommand{\sfH}{{\sf H}}
\newcommand{\sfW}{{\sf W}}
\newcommand{\sfD}{{\sf D}}
\newcommand{\sfKbar}{{\bar{{\sf K}}}}

\newcommand{\Reg}{{\mathfrak{R}}}
\newcommand{\Irr}{{\mathfrak{I}}}
\newcommand{\Preper}{{\cal P}}
\newcommand{\Repelling}{{\cal R}}
\newcommand{\Trivial}{{\cal T}}

\newcommand{\comment}[1]{{\sc Comment from GFR: {\tt #1}}}

%\renewcommand{\comment}[1]{}

% Array for subscripts

\newenvironment{sarray}{
	  \textfont0=\scriptfont0
	  \scriptfont0=\scriptscriptfont0
	  \textfont1=\scriptfont1
	  \scriptfont1=\scriptscriptfont1
	  \textfont2=\scriptfont2
	  \scriptfont2=\scriptscriptfont2
	  \textfont3=\scriptfont3
	  \scriptfont3=\scriptscriptfont3
	\renewcommand{\arraystretch}{0.7}
	\begin{array}{l}}{\end{array}}

\newenvironment{scarray}{
	  \textfont0=\scriptfont0
	  \scriptfont0=\scriptscriptfont0
	  \textfont1=\scriptfont1
	  \scriptfont1=\scriptscriptfont1
	  \textfont2=\scriptfont2
	  \scriptfont2=\scriptscriptfont2
	  \textfont3=\scriptfont3
	  \scriptfont3=\scriptscriptfont3
	\renewcommand{\arraystretch}{0.7}
	\begin{array}{c}}{\end{array}}

\section{Introduction}

The roots of the chromatic polynomial of a graph,
and their location in the complex plane,
have been extensively studied both by combinatorial mathematicians
and by statistical physicists
\cite{Jackson_03,MR2187739,Royle_BCC2009}.
Combinatorial mathematicians were originally motivated
by attempts (thus far unsuccessful)
to use analytic techniques to prove the four-colour theorem
\cite[p.~357]{Birkhoff_46},
while statistical physicists are motivated by the deep connections
to the partition function of the $q$-state Potts model
and the Yang--Lee theory of phase transitions
\cite{MR2187739}.

For both groups of researchers, one of the fundamental questions
that arises is to find bounds on the location of chromatic roots
in terms of graph structure or parameters of the graph.
Early conjectures that there might be {\em absolute}\/ bounds
on the location of chromatic roots,
such as being restricted to the right half-plane \cite{MR558614},
%% (see Farrell \cite{MR558614})
were disproved by the following strong result:

\begin{thm} %% [Sokal \protect\cite[Corollary~1.3]{MR2047238}]
%%   {$\!\!\!$ \bf (Sokal \protect\cite[Corollary~1.3]{MR2047238}) \ }
{$\!\!$ \bf \protect\cite[Theorems 1.1--1.4]{MR2047238} \ }
\label{thm_sokaldense}
Chromatic roots are dense in the whole complex plane.
Indeed, even the chromatic roots of the generalized theta graphs
$\Theta^{(s,p)}$
are dense in the whole complex plane with the possible exception
of the disc $|q-1| < 1$.\footnote{
   The generalized theta graph $\Theta^{(s,p)}$ consists of
   a pair of endvertices joined by $p$ internally disjoint paths,
   each path consisting of $s$ edges.
   The letters $s$ and $p$ are chosen to indicate ``series''
   and ``parallel'', respectively.
}
%% \cite[Theorems~1.1, 1.2 and 1.4]{MR2047238}.
\end{thm}

%% \noindent
%% --- which form a tiny subset of series-parallel graphs,
%% which in turn form a tiny subset of planar graphs,
%% which form a tiny subset of all graphs ---

Biggs, Damerell and Sands \cite{MR0294172}
were the first to suggest, in the early 1970s, that the degree (i.e.\ valency)
of a regular graph might be relevant to the location of its chromatic roots.
They conjectured (on rather limited evidence) the existence of
a function $f$ such that the chromatic roots of a regular graph
of degree $r$ lie in the disc $|q| \leq f(r)$.
Two decades later, Brenti, Royle and Wagner \cite{MR1260339}
extended this conjecture to not-necessarily-regular graphs
of maximum degree $r$.\footnote{
   This latter conjecture actually follows from the former one,
   as indicated by Thomassen \cite[p.~505]{Thomassen_97}:
   If $r$ is odd, then there is a graph $H$ with all vertices but one
   having degree $r$ and the remaining vertex having degree 1.
   Then, given any graph $G$ of maximum degree $r-1$ or $r$,
   we glue enough copies of this ``gadget'' $H$ (using its vertex of degree 1)
   to the vertices of degree less than $r$ in $G$,
   thereby yielding a regular graph of degree $r$
   whose chromatic roots are the union of the chromatic roots
   of the original graph $G$ and those of $H$.
}
This latter conjecture was finally confirmed by one of us,
who used cluster-expansion techniques from statistical physics
to show that taking $f(r) \approx 8r$ would suffice:

\begin{thm}  %% [Sokal \cite[Corollary~5.3 and Proposition~5.4]{MR1827809}]
%%   {$\!\!\!$ \bf (Sokal \protect\cite[Corollary~5.3 and Proposition~5.4]{MR1827809}) \ }
{$\!\!$ \bf \protect\cite[Corollary~5.3 and Proposition~5.4]{MR1827809} \ }
\label{sokalbound}
The chromatic roots of a graph of maximum degree $\Delta$
lie in the disc $|q| \leq 7.963907 \Delta$.
\end{thm}

\noindent
Moreover, almost the same bound holds when the largest degree $\Delta$
is replaced by the second-largest degree $\Delta_2$:
namely, all the chromatic roots
lie in the disc $|q| \leq 7.963907 \Delta_2+1$
\cite[Corollary~6.4]{MR1827809}.\footnote{
   Note that it is {\em not}\/ possible to go farther and obtain a bound
   in terms of the third-largest degree $\Delta_3$,
   as the chromatic roots of the generalized theta graphs $\Theta^{(s,p)}$
   --- which have $\Delta = \Delta_2 = p$ but $\Delta_3 = 2$ ---
   are dense in the whole complex plane with the possible exception
   of the disc $|q-1| < 1$.
   %% \cite[Theorems~1.1, 1.2 and 1.4]{MR2047238}.
}
The constant $7.963907$ (see also \cite{Borgs_06})
is an artifact of the proof and is not likely to be close to the true value.
Fern\'andez and Procacci \cite{MR2396349}
have recently improved the constant in Theorem~\ref{sokalbound}
to 6.907652 (see also \cite{Jackson-Procacci-Sokal}),
but this is probably still far from best possible.
Of course, the linear dependence on $\Delta$ is indeed best possible,
since the complete graph $K_{\Delta+1}$
has chromatic roots $0,1,2,\ldots,\Delta$.\footnote{
   Perhaps surprisingly, the complete graph $K_{\Delta+1}$ is {\em not}\/
   the extremal graph for this problem (except presumably for $\Delta =1,2,3$),
   and a bound $|q| \le \Delta$ is {\em not}\/ possible.
   In fact, a non-rigorous (but probably rigorizable) asymptotic analysis,
   confirmed by numerical calculations, shows \cite{Salas-Sokal_bipartite}
   that the complete bipartite graph $K_{\Delta,\Delta}$
   has a chromatic root $\alpha \Delta + o(\Delta)$, where
   $\alpha = - 2 / W(-2/e)  \approx 0.678345 + 1.447937 i$;
   here $W$ denotes the principal branch of the Lambert $W$ function
   (the inverse function of $w \mapsto w e^w$) \cite{Corless_96}.
   So the constant in Theorem~\ref{sokalbound} cannot be better than
   $|\alpha| \approx 1.598960$.
   One of us has conjectured \cite[Conjecture~6.6]{Royle_BCC2009}
   that, for $\Delta \ge 4$, the complete bipartite graph $K_{\Delta,\Delta}$
   has the chromatic root of largest modulus
   (and also largest imaginary part)
   among all graphs of maximum degree $\Delta$.
   Furthermore, it seems empirically that the largest modulus
   of a chromatic root of $K_{\Delta,\Delta}$, divided by $\Delta$,
   is an increasing function of $\Delta$.
   If these conjectures are correct, then the optimal constant in
   Theorem~\ref{sokalbound} would be $|\alpha|$.
}

The parameters $\Delta$ and $\Delta_2$ are, however,
unsatisfactory in various ways.
For example, $\Delta$ and $\Delta_2$ can be made arbitrarily large
by gluing together blocks at a cut-vertex,
yet this operation does not alter the chromatic roots.
The underlying reason for this discrepancy is that
the chromatic polynomial is essentially a property of the cycle matroid
of the graph, but vertex degrees are not.
Therefore it would be of great interest to find a matroidal parameter
that could play the role of maximum degree (or second-largest degree)
in results of this type.
Motivated by some remarks of Shrock and Tsai \cite{Shrock_98e,shrocktsai99a},
a few years ago
Sokal \cite[Section~7]{MR1827809}
and Jackson and Sokal \cite{jacksonsokal2}
suggested considering a graph parameter
that they called {\em maxmaxflow}\/, defined as follows:
If $x$ and $y$ are distinct vertices in a graph $G$,
then let $\lambda_G(x,y)$ denote the maximum flow from $x$ to $y$:
\begin{subeqnarray}
\lambda_G(x,y)
& = & \text{max. number of edge disjoint paths from $x$ to $y$}\\
& = & \text{min. number of edges separating $x$ from $y$.}
  \label{def.lambdaG}
\end{subeqnarray}
Then define the maxmaxflow $\Lambda(G)$ to be the maximum of these values
over all pairs of distinct vertices:
\begin{equation}
\Lambda(G) \;=\; \max_{x \neq y} \lambda_G(x,y)  \;.
\label{def_Lambda}
\end{equation}
Although this definition appears to use the non-matroidal concept
of a ``vertex'' in a fundamental way,
Jackson and Sokal \cite{jacksonsokal2}
proved that maxmaxflow has a ``dual'' formulation
in terms of cocycle bases:
namely,
\begin{equation}
\Lambda(G) \;=\; \min_{\cal B} \max_{C \in {\cal B}} |C|
\label{def_maxmaxflow_matroidal}
\end{equation}
where $M(G)$ is the cycle matroid of the graph $G$,
the min runs over all bases ${\cal B}$ of the cocycle space of $M(G)$
[over $GF(2)$], and the max runs over all cocycles in the basis ${\cal B}$.
Thus, by taking \eqref{def_maxmaxflow_matroidal} as the {\em definition}\/
of $\Lambda(M)$ for an arbitrary binary matroid $M$,
we obtain a matroidal parameter that specializes to maxmaxflow
for a graphic matroid.
Furthermore, for graphs, $\Lambda(G)$ behaves exactly
as one would wish with respect to gluing together blocks at a cut-vertex:
namely, the maxmaxflow of a graph is the maximum of the maxmaxflows
of its blocks.
Furthermore, from either description it is immediate that
\begin{equation}
\Lambda(G) \;\leq\; \Delta_2(G)  \;.
\end{equation}
It is therefore natural to make the following conjecture
\cite{MR1827809,jacksonsokal2},
which if true would extend the known bound on chromatic roots
in terms of second-largest degree:

\begin{conjecture} %%  [\cite[Conjecture~1.1]{jacksonsokal2}]
{$\!\!$ \bf \cite[Conjecture~1.1]{jacksonsokal2} \ }
\label{conj.chrom.1}
There exist universal constants $C(\Lambda) < \infty$ such that all
the chromatic roots (real or complex) of all loopless graphs of
maxmaxflow $\Lambda$ lie in the disc $|q| \le C(\Lambda)$.
Indeed, it is conjectured that
$C(\Lambda)$ can be taken to be linear in $\Lambda$.
\end{conjecture}

\noindent
However, there are some serious difficulties in modifying
the existing cluster-expansion proof of Theorem~\ref{sokalbound}
to get an analogous bound in terms of $\Lambda$;
and although some progress has been made in this direction
\cite{jacksonsokal2}, a number of obstacles remain.\footnote{
   See \cite[Section~9.2]{MR2187739} for a brief discussion.
}

In this paper, we restrict our attention to {\em series-parallel graphs}\/
and use an entirely different approach to prove the following main result:
%% that the chromatic roots of a series-parallel graph of maxmaxflow $\Lambda$
%% lie in the disc 
%% \begin{equation}
%%   |q - 1| \;\leq\; (\Lambda-1)/\log 2 \;\approx\; 1.443 (\Lambda-1)  \;.
%% \end{equation}

\begin{thm}
\label{thm_main}
Fix an integer $\Lambda \ge 2$,
and let $G$ be a loopless series-parallel graph of maxmaxflow at most $\Lambda$.
Then all the roots (real or complex) of the chromatic polynomial $P_G(q)$
lie in the disc $|q-1| < (\Lambda-1) / \log 2 \approx 1.442695 (\Lambda-1)$.
\end{thm}

\noindent
Since there are series-parallel graphs of maxmaxflow $\Lambda$
having chromatic roots arbitrarily close to
every point of the circle $|q-1| = \Lambda-1$
(see Appendix~\ref{app_tree} below),
the constant in Theorem~\ref{thm_main} is non-sharp by at most a factor
$1/\log 2 \approx 1.442695$.
Moreover, a bound $|q-1| \le \Lambda-1$ {\em cannot}\/ hold in general,
since at least for $\Lambda=3$ we can exhibit a 94-vertex series-parallel graph
with a chromatic root at $|q-1| \approx  2.009462$
(see Section~\ref{sec.Lambda=3}).

Let us also remark that in this paper we use only the
definition \reff{def_Lambda} of maxmaxflow;
we do not use the result \reff{def_maxmaxflow_matroidal}.

The essence of our approach is to view the chromatic polynomial $P_G(q)$
as a special case of the multivariate Tutte polynomial
$Z_G(q, \bv)$ of a graph equipped with edge weights
${\bf v} = \{v_e\}_{e\in E}$:
namely, the case in which all the edge weights
take the special value $v_e = -1$.\footnote{
See \cite{MR2187739} for a review on the multivariate Tutte polynomial
(which is also known in statistical physics as the partition function
of the $q$-state Potts model in the Fortuin--Kasteleyn representation).
}
By working within the more flexible framework
of the multivariate Tutte polynomial,
we can use the rules for series and parallel reduction
\cite[Sections~4.4 and 4.5]{MR2187739}
to transform a graph $G$ into a smaller graph
with different edge weights and the same (or closely related)
multivariate Tutte polynomial.
In particular, a series-parallel graph can be transformed
into a one-edge graph with a complicated weight
(a messy rational function of $q$ and $\{v_e\}$) on its single edge.
Although this weight is complicated,
we are able in certain circumstances to bound 
where it lies in the complex plane and thereby to ensure that
the multivariate Tutte polynomial is nonvanishing.
After some fairly straightforward real and complex analysis,
we can prove Theorem~\ref{thm_main}.

We shall actually prove a result that is slightly stronger than
Theorem~\ref{thm_main} in two ways:
First of all, the chromatic roots will be shown to lie in a disc
$|q-1| < Q^\star_\Lambda$, where $Q^\star_\Lambda$
is the solution of a particular polynomial equation of degree $2\Lambda-3$
and satisfies $Q^\star_\Lambda < (\Lambda-\smfrac{3}{2}\log 2) / \log 2
                 < (\Lambda-1) / \log 2$.
[Note that $\smfrac{3}{2}\log 2 \approx 1.039721$.]
Secondly, the chromatic polynomial $P_G(q)$ can be replaced by
the multivariate Tutte polynomial $Z_G(q, \bv)$
where the edge weights ${\bf v} = \{v_e\}_{e\in E}$ lie in a suitable set.
See Theorem~\ref{thm_main_sec5} for details.

At this point the reader might well wonder:
Since series-parallel graphs form a tiny subset of planar graphs,
which in turn form a tiny subset of all graphs,
what is the interest of a result restricted to the former?
The answer is that Theorem~\ref{thm_sokaldense}
already shows that even series-parallel graphs can exhibit ``wild'' behavior
in their chromatic roots.
If one wishes to bound those roots,
then {\em some}\/ additional parameter is clearly needed.
It is a nontrivial fact that maxmaxflow is such a parameter.
Whether or not this is good evidence for the truth of the more general
Conjecture~\ref{conj.chrom.1} remains to be seen.

The techniques used in proving Theorem~\ref{thm_main}
lend themselves to a number of direct extensions.
For example, one fairly easy extension is to permit the original graph
to have edge weights throughout the ``real antiferromagnetic regime'',
i.e.\ taking $v_e \in [-1,0]$ independently for each edge $e$.
It turns out that exactly the same bound holds:

\begin{thm}
\label{thm_main_AF}
Fix an integer $\Lambda \ge 2$.
Let $G=(V,E)$ be a loopless series-parallel graph
of maxmaxflow at most $\Lambda$,
and let the edge weights $\bv = \{v_e\}_{e \in E}$
satisfy $v_e \in [-1,0]$ for all $e$.
Then all the roots (real or complex) of the
multivariate Tutte polynomial $Z_G(q, \bv)$
lie in the disc $|q-1| < (\Lambda-1) / \log 2 \approx 1.442695 (\Lambda-1)$.
\end{thm}

\noindent
Once again, we shall actually prove a slightly stronger result,
in which the chromatic roots are shown to lie in the disc
$|q-1| < Q^\star_\Lambda$,
and in which the edge weights ${\bf v} = \{v_e\}_{e\in E}$
are allowed to lie in a set that is somewhat larger than $[-1,0]$.
See Theorem~\ref{thm_main_AF_sec7}.

A second extension is to consider graphs that are not series-parallel
but are nevertheless built up by using series and parallel compositions
from a fixed ``starting set'' of graphs. 
For instance, we can prove the following:

\begin{thm}
  \label{thm_main_Wheatstone}
Let $\sfG = (G,s,t)$ be a 2-terminal graph that
can be obtained from $K_2$ and the Wheatstone bridge $\sfW$
by successive series and parallel compositions.\footnote{
   The {\em Wheatstone bridge}\/ is the 2-terminal graph $\sfW = (W,s,t)$
   obtained from $W = K_4- e$ by taking the two vertices of degree $2$
   to be the terminals $s$ and $t$.
   See Section~\ref{subsec.Wheatstone}.
}
If $G$ has maxmaxflow at most $\Lambda$ (where $\Lambda \ge 3$),
then all the roots (real or complex) of the chromatic polynomial $P_G(q)$
lie in the disc $|q-1| < (\Lambda- \log 2) / \log 2$.
\end{thm}

\noindent
Once again, we shall actually prove a slightly stronger result:
see Theorem~\ref{thm_main_Wheatstone_sec8}
and Corollary~\ref{cor_main_Wheatstone_sec8}.

\bigskip

The plan of this paper is as follows:
In Section~\ref{sec2} we review the properties of the
multivariate Tutte polynomial, with emphasis on its behavior under
series and parallel composition.
In Section~\ref{sec.serpar} we discuss series-parallel graphs
and decomposition trees for 2-terminal graphs.
In Section~\ref{sec.abstract} we state and prove an abstract result
that gives a sufficient condition for the multivariate Tutte polynomial
to be nonzero, involving sets
$S_1 \subseteq S_2 \subseteq \cdots \subseteq S_{\Lambda-1}$
in the complex plane satisfying certain conditions.
In Section~\ref{sec_discs} we prove Theorem~\ref{thm_main}
(and the stronger Theorem~\ref{thm_main_sec5})
by constructing suitable sets 
$S_1 \subseteq S_2 \subseteq \cdots \subseteq S_{\Lambda-1}$.
In Section~\ref{sec.Lambda=3} we give a slightly sharper result
for the case $\Lambda=3$.
In Section~\ref{sec.antiferro} we prove Theorem~\ref{thm_main_AF}
(and the stronger Theorem~\ref{thm_main_AF_sec7})
by a slight generalization of our previous construction.
In Section~\ref{sec.Wheatstone} we prove Theorem~\ref{thm_main_Wheatstone}
(and the stronger Theorem~\ref{thm_main_Wheatstone_sec8}
 and Corollary~\ref{cor_main_Wheatstone_sec8}).
In Appendix~\ref{app_riemann} we define parallel and series connection
for edge weights lying in the Riemann sphere.
In Appendix~\ref{app_tree} we prove Theorem~\ref{thm.leaf-joined}
on the chromatic roots of leaf-joined trees
by using methods from the theory of holomorphic dynamics.

\section{The multivariate Tutte polynomial}  \label{sec2}

In this section we begin by reviewing the definition and elementary properties
of the multivariate Tutte polynomial (Section~\ref{subsec.elementary}).
We then discuss the technical tools
that will play a central role in this paper:
parallel and series reduction of edges (Section~\ref{subsec.serpar}),
the partial multivariate Tutte polynomials and
``effective weights'' $v_{\rm eff}$ for 2-terminal graphs
(Section~\ref{subsec.partial}),
and the parallel and series composition of 2-terminal graphs
(Section~\ref{subsec.tutte.serpar.2term}).

\subsection{Definition and elementary properties}  \label{subsec.elementary}

Let $G = (V,E)$ be a finite undirected graph
(which may have loops and/or multiple edges).
Then the {\em multivariate Tutte polynomial}\/ of $G$ is the polynomial
\begin{equation}
Z_G(q,\bv) \;=\; \sum_{A \subseteq E} q^{k(A)} \prod_{e \in A} v_e  \;,
\label{multivariatetutte}
\end{equation}
where $q$ and ${\bf v} = \{v_e\}_{e\in E}$ are commuting indeterminates
and $k(A)$ is the number of connected components in the subgraph $(V,A)$.
See \cite{MR2187739} for a review on the multivariate Tutte polynomial.
In this paper we shall sometimes consider $Z_G(q,\bv)$ algebraically
as a polynomial belonging to the polynomial ring $\Z[q,\bv]$ or $\C[q,\bv]$,
but we shall most often take an analytic point of view
and consider $Z_G(q,\bv)$ to be a polynomial function
of the complex variables $q$ and $\{v_e\}$.

If $q$ is a positive integer, then the multivariate Tutte polynomial
is equal to the partition function of the {\em $q$-state Potts model}\/
in statistical mechanics, which is defined by
\begin{equation}
Z_G^{\rm Potts}(q,\bv)   \;=\;
\sum_{ \sigma \colon\, V \to \{ 1,2,\ldots,q \} }
\; \prod_{e \in E}  \,
\biggl[ 1 + v_e \delta(\sigma(x_1(e)), \sigma(x_2(e))) \biggr]
\;.
\label{def.ZPotts}
\end{equation}
where the sum runs over all maps $\sigma\colon\, V \to \{ 1,2,\ldots,q \}$,
the $\delta$ is the Kronecker delta
\begin{equation}
\delta(a,b)   \;=\;    \begin{cases}
1  & \text{if $a=b$} \\
0  & \text{if $a \neq b$}
\end{cases}
\end{equation}
and $x_1(e), x_2(e) \in V$ are the two endpoints of the edge $e$
(in arbitrary order).
More precisely, we have:

\begin{theorem}[Fortuin--Kasteleyn representation of the Potts model]
\label{thm.FK}
\hfill\break
\vspace*{-0.4cm}
\par\noindent
For integer $q \ge 1$,
\begin{equation}
Z_G^{\rm Potts}(q, \bv) \;=\;  Z_G(q, \bv)  \;.
\label{eq.FK.identity}
\end{equation}
That is, the Potts-model partition function
is simply the specialization of the multivariate Tutte polynomial
to $q \in \Z_+$.
\end{theorem}

\noindent
See e.g.\ \cite[Section~2.2]{MR2187739} for the easy proof.

We shall adopt the terminology from statistical mechanics
to designate various sets of values for the edge weights $v_e$.
In particular, we shall say that a real weight $v_e$
is {\em ferromagnetic}\/ if $v_e \geq 0$
and {\em antiferromagnetic}\/ if $-1 \leq v_e \leq 0$.
We shall also sometimes say that a complex weight $v_e$
is {\em complex ferromagnetic}\/ if $|1+v_e| \ge 1$
and {\em complex antiferromagnetic}\/ if $|1+v_e| \le 1$.
Finally, we shall say that a set of weights ${\bf v} = \{v_e\}_{e\in E}$
is ferromagnetic or antiferromagnetic if {\em all}\/ of the $v_e$ are.

The {\em zero-temperature limit of the antiferromagnetic Potts model}\/
arises when $v_e = -1$ for all edges $e$:
then $Z_G^{\rm Potts}$ gives weight 1 to each proper coloring
and weight 0 to each improper coloring,
and so counts the proper colorings.
It follows from Theorem~\ref{thm.FK} that the number of proper $q$-colorings
of $G$ is in fact the restriction to $q \in \Z_+$
of a polynomial in $q$, namely the {\em chromatic polynomial}\/
\begin{equation}
P_G(q) \;=\; Z_G(q,\{-1\})  \;.
\end{equation}

The multivariate Tutte polynomial factorizes in a simple way
over connected components and blocks.
If $G$ is the disjoint union of $G_1$ and $G_2$, then trivially
\begin{equation}
Z_G(q,\bv)  \;=\;  Z_{G_1}(q,\bv) \, Z_{G_2}(q,\bv)
\;.
\label{eq.components}
\end{equation}
%% i.e.\ $Z_G$ ``factorizes over components''.
If $G$ consists of subgraphs $G_1$ and $G_2$ joined at a single cut vertex,
then it is not hard to see \cite[Section~4.1]{MR2187739} that
\begin{equation}
Z_G(q,\bv)  \;=\;  {Z_{G_1}(q,\bv) \, Z_{G_2}(q,\bv)  \over q}
\;.
\label{eq.blocks}
\end{equation}
%% i.e.\ $Z_G$ ``factorizes over blocks'' modulo a factor $q$.
Therefore, when studying the multivariate Tutte polynomial,
it suffices to restrict attention to {\em nonseparable}\/ graphs $G$.\footnote{
   See Section~\ref{sec.serpar} for a precise definition
   of ``nonseparable'' for graphs that may contain loops.
}

Note also that a loop $e$ contributes a trivial prefactor $1+v_e$
to $Z_G(q,\bv)$.
If $v_e = -1$ (as it is e.g.\ for the chromatic polynomial),
this causes $Z_G$ to be identically zero as a polynomial in $q$;
if $v_e \neq -1$, the loop does not affect the roots of $Z_G$ at all.
Since in this paper we want to allow $v_e = -1$,
we shall assume in our main theorems that the graph $G$ is loopless.

Finally, if $G$ consists of a single vertex and no edges (i.e.\ $G=K_1$),
then $Z_G(q,\bv) = q$.
So we can assume without loss of generality that
$G$ is loopless, nonseparable and contains at least one edge.

There are several reasons why it can be advantageous to consider
the multivariate Tutte polynomial, even when the ultimate goal is
to obtain results on the chromatic polynomial.
The first reason is that $Z_G(q, \bv)$ is {\em multiaffine}\/
in the variables $\bv$ (i.e.\ of degree 1 in each $v_e$ separately);
and often a multiaffine polynomial in many variables is easier to handle
than a general polynomial in a single variable
(e.g.\ it may permit simple proofs by induction
on the number of variables).
Secondly, allowing unequal edge weights $v_e$
permits more flexibility in inductive proofs;
indeed, in some cases the stronger result is much {\em easier}\/ to prove.
In particular, local operations on graphs can be reflected in
local changes to the edge weights of the affected edges,
which is impossible if all edge weights are constrained to be equal.\footnote{
   One striking example of this phenomenon is the three-line proof
   of the multivariate Brown--Colbourn property for series-parallel graphs
   \cite[Remark 3 in Section 4.1]{MR1827809}
   \cite[Theorem 5.6(c) $\Longrightarrow$ (a)]{Royle-Sokal},
   which contrasts with the 20-page proof of the corresponding
   univariate result \cite{Wagner_00}.
   See \cite{Jackson-Sokal_zerofree} for several further instances
   in which results on the chromatic polynomial can be proven
   more easily by working within the more general framework of the
   multivariate Tutte polynomial.
}
In this context, two of the most important such ``local operations''
are parallel and series reductions,
to be discussed in the next subsection.

\subsection{Parallel and series reduction}  \label{subsec.serpar}

We say that edges $e$, $f \in E$ are {\em in parallel}\/
if they connect the same pair of distinct vertices $x$ and $y$.
In this case they can be replaced, without changing the value of
the multivariate Tutte polynomial,
by a single new edge with ``effective weight''
\begin{equation}
v_{\rm eff}  \;=\; (1+v_e) (1+v_f)- 1  \;.
\end{equation}
This operation of replacing two parallel edges by a single edge
is called {\em parallel reduction}\/, and
we write $v_e \pll v_f$ as a shorthand for $(1+v_e)(1+v_f)-1$.

We say that edges $e$, $f \in E$ are
{\em in series (in the narrow sense)}\/\footnote{
Note that this definition of ``edges in series''
is more restrictive than the matroidal definition of elements in series,
but the distinction is not important in our context.
See \cite[Section~4.5]{MR2187739} for further discussion.
}
if there are vertices $x,y,z \in V$ with 
$x \not= y$ and $y\not= z$ such that
$e$ connects $x$ and $y$, $f$ connects $y$ and $z$, and $y$ has degree 2.
In this case, replacing the edges $e$ and $f$
with a single edge of effective weight
\begin{equation}
v_{\rm eff} \;=\; \frac{v_{e} v_{f}} {q+v_{e}+v_{f}}
\label{series_reduction}
\end{equation}
yields a graph whose multivariate Tutte polynomial ---
when multiplied by the prefactor $q+v_{e}+v_{f}$ ---
is the same as that of the original graph,
provided that $q+v_{e}+v_{f} \neq 0$.
More formally, we can consider the new graph
to be obtained from $G$ by contracting $f$,
and we can write
\begin{equation}
Z_G(q, \bv_{\neq e,f}, v_e, v_f)  \;=\;
(q + v_e + v_f) \,
Z_{G / f}(q, \bv_{\neq e,f},
v_e v_f / (q + v_e + v_f))
\;.
\label{eq.series2}
\end{equation}
See \cite[Section~4.5]{MR2187739} for the easy proof.
Naturally this operation is called {\em series reduction}\/,
%% and we write $v_e \serq v_f$ as a shorthand for $v_ev_f/(q+v_e+v_f)$. 
and we write
\be
v_e \serq v_f  \;=\;
\begin{cases}
\frac{\displaystyle v_{e} v_{f}} {\displaystyle q+v_{e}+v_{f}} 
& \text{\rm if $q+v_{e}+v_{f} \neq 0$}   \\[2mm]
\hbox{undefined}
& \text{\rm if $q+v_{e}+v_{f} = 0$}
\end{cases}
\label{eq.v.serq}
\ee
where ``undefined'' is a special value (not a complex number).
We furthermore declare that any $\pll$ or $\serq$ operation
in which one or both of the inputs is ``undefined''
yields an output that is also ``undefined''.
The operators $\pll$ and $\serq$ are thus maps
$\Chat \times \Chat \to \Chat$,
where $\Chat = \C \cup \{\hbox{undefined}\}$.\footnote{
But see the Remark at the end of this subsection,
as well as Appendix~\ref{app_riemann}.
%% {\em Some}\/ of the cases with $q+v_{e}+v_{f} = 0$
%% can be interpreted as an unambiguous value $v_e \serq v_f = \infty$,
%% i.e.\ we could, if we wished, define the map
%% $\serq \colon\; \C \times \C \to \C \cup \{\infty,\hbox{undefined}\}$
%% by
%% $$
%% v_e \serq v_f  \;=\;
%% \begin{cases}
%%     \frac{v_{e} v_{f}} {q+v_{e}+v_{f}} 
%%         & \text{\rm if $q+v_{e}+v_{f} \neq 0$}   \\
%%     \infty
%%         & \text{\rm if $q+v_{e}+v_{f} = 0$ and $v_e v_f \neq 0$}  \\
%%     \hbox{undefined}
%%         & \text{\rm if $q+v_{e}+v_{f} = 0$ and $v_e v_f = 0$}
%% \end{cases}
%% $$
%% But this would be of little use unless we {\em also}\/ extended
%% the map $\serq$ to allow $\infty$, in at least some cases,
%% as an {\em input}\/ value.
%% This too can be done, but we do not need it,
%% so we prefer to declare $v_e \serq v_f = \hbox{undefined}$
%% whenever $q+v_{e}+v_{f} = 0$.
}

There are other ways to parametrize the edge weights
occurring in the multivariate Tutte polynomial,
and there are often advantages in using the variables
that give the simplest expression for the immediate task at hand.
In particular, in this paper we will use {\em three}\/ sets of variables,
namely the edge weights $\{v_e\}$,
the {\em transmissivities}\/ $\{t_e\}$ defined by
\begin{equation}
t_e \,=\, \frac{v_e}{q+v_e}, \qquad v_e \,=\, \frac{qt_e}{1-t_e} \;,
\label{def_te}
\end{equation}
and a third set of variables $\{y_e\}$ given by
\begin{equation}
y_e \,=\, 1 + v_e, \qquad v_e \,=\, y_e  - 1 \;.
\label{def_ye}
\end{equation}
There are two main reasons for using these different sets of variables.
The first reason is that the variables $\{t_e\}$ and $\{y_e\}$ 
each make one of the operations of series and parallel reduction trivial.
More precisely, let $\pll^V$, $\pll^T$ and $\pll^Y$ denote
the parallel-reduction operation expressed in the $v$, $t$ and $y$ variables,
respectively, and similarly for $\ser^V$, $\ser^T$ and $\ser^Y$. Then we have
\begin{eqnarray}
v_e \pll^V v_f & = &(1+v_e)(1+v_f)-1  \label{eq_vpar} \\[1mm]
v_e \ser_q^V v_e & = &\frac{v_{e} v_{f}} {q+v_{e}+v_{f}}
\label{eq_vser} \\[5mm]
t_e \pll^T_q t_f & = & \frac{t_e + t_f + (q-2) t_e t_f}{1+(q-1) t_e t_f}
\label{eq_tpar} \\[1mm]
t_e \ser^T t_f & = & t_e t_f \\[5mm]
y_e \pll^Y y_f & = & y_e y_f \\[1mm]
y_e \ser_q^Y y_f & = & \frac {q-1+y_ey_f} {q-2+y_e+y_f}
\label{eq_yser}
\end{eqnarray}
where it is understood in \reff{eq_vser}/\reff{eq_tpar}/\reff{eq_yser}
that the result is declared to be ``undefined''
whenever the denominator vanishes, as in \reff{eq.v.serq}.
We have given the operators a $q$-subscript
whenever the corresponding expression depends on $q$.
Note that series reduction is particularly easy in the $t$-variables,
while parallel reduction is particularly easy in the $y$-variables.
We shall also use the obvious notations
\begin{eqnarray}
A \pll^V B    & = &  \{a \pll^V b \colon\;  a \in A,\, b \in B \}  \\[1mm]
A \ser_q^V B  & = &  \{a \ser_q^V b \colon\;  a \in A,\, b \in B \}
\end{eqnarray}
when $A$ and $B$ are subsets of the complex plane,
and analogously for the other variables.

The second reason for introducing these different sets of variables
is that the regions we are attempting to bound have different shapes
in the complex $v$-plane, $t$-plane and $y$-plane,
and we will ultimately choose the variables in which the regions
are the easiest to effectively bound.
Of course, since the maps \reff{def_te} and \reff{def_ye}
are M\"obius transformations, discs in any one of these planes
always map to discs (or their exteriors) in any other one of these planes;
but discs centered at the origin do not in general map
to discs centered at the origin,
and concentric discs do not in general map to concentric discs.
It is convenient, as we shall see, to choose variables in which
we can use discs centered at the origin.

To avoid notational overload, we will normally specify
the variables being used in each section of the paper
and use the convention that $\pll$ and $\ser$ with no superscript
refer to the expressions applicable to the current choice.

\medskip

{\bf Remark.}  The definitions given in this section concerning the
use of the value ``undefined'' are convenient for the main purposes of
this paper, where we will be dealing with regions that belong to the
finite plane simultaneously in both the $v$- and $t$-variables,
but they are somewhat unnatural because the conditions for being ``undefined''
in \reff{eq_vser}/\reff{eq_yser} and \reff{eq_tpar} do {\em not}\/ correspond:
$q+v_e+v_f=0$ is {\em not}\/ equivalent to $1+ (q-1) t_e t_f = 0$.
A more natural approach is to define the operations $\pll$ and $\ser$
on the Riemann sphere $\Cbar = \C \cup \{\infty\}$
in such a way that the conditions for being ``undefined''
are the same no matter which variables are used.
See Appendix~\ref{app_riemann} for a brief description of this approach.
We shall employ this approach in Appendix~\ref{app_tree}
in studying the chromatic roots of leaf-joined trees.

\subsection{Partial multivariate Tutte polynomials and $v_{\rm eff}$
for 2-terminal graphs}   \label{subsec.partial}

A {\em 2-terminal graph}\/ $\sfG = (G,s,t)$
is a graph $G$ with two distinguished vertices $s$ and $t$ ($s \neq t$),
called the {\em terminals}\/.
(We do not insist here that $G$ be connected,
but in practice it always will be.)
Given a 2-terminal graph $(G,s,t)$, we define 
the {\em partial multivariate Tutte polynomials}\/
\begin{eqnarray}
\zgsnt(q,\bv)   & = &
\!\!\!
\sum_{\begin{scarray}
E' \subseteq E \\
E' \, {\rm does\,not\,connect} \, s \, {\rm to} \, t
\end{scarray}
}
\!\!\!\!\!  q^{k(E')} \;  \prod_{e \in E'}  v_e
\label{def.zgsnt} \\[4mm]
\zgst(q,\bv)   & = & \!\!\!
\sum_{\begin{scarray}
E' \subseteq E \\
E' \, {\rm connects} \, s \, {\rm to} \, t
\end{scarray}
}
\!\!\!\!\!  q^{k(E')} \;  \prod_{e \in E'}  v_e
\label{def.zgst}
\end{eqnarray}
{}From \reff{multivariatetutte} we have trivially
\begin{equation}
Z_G(q,\bv)  \;=\;  \zgsnt(q,\bv) \,+\, \zgst(q,\bv)
\label{eq2.G}
\;.
\end{equation}
Since clearly $k(E') \ge 2$ (resp.\ 1)
whenever $E'$ does not connect (resp.\ connects) $s$ to $t$,
it is convenient to define
\begin{eqnarray}
A_{G,s,t}(q, \bv)   & = &  q^{-2} \zgsnt(\qvbf)
\label{def.AGst}  \\[2mm]
B_{G,s,t}(q, \bv)   & = &  q^{-1} \zgst(\qvbf)
\label{def.BGst}
\end{eqnarray}
$A_{G,s,t}(q, \bv)$ and $B_{G,s,t}(q, \bv)$ are thus defined by
sums like \reff{def.zgsnt}/\reff{def.zgst} but in which
only those connected components
{\em not containing one or both of the terminals $s,t$}\/
receive a factor $q$.
We also define the ``effective weight''
\begin{equation}
v_{\rm eff}(G,s,t)  \;\equiv\;
{B_{G,s,t}(q,\bv) \over A_{G,s,t}(q,\bv)}
\;=\;  {q \zgst(q,\bv) \over \zgsnt(q,\bv)}
\;,
\label{def.veff}
\end{equation}
which is a rational function of $q$ and $\{v_e\}$.
[Note that the polynomial $\zgsnt(q,\bv)$ cannot vanish identically,
 because the term $E' = \emptyset$ in \reff{def.zgsnt} contributes
 $q^{|V(G)|}$.]
More precisely:

\begin{lemma}
   \label{lemma_veff_nontrivial}
Let $(G,s,t)$ be a 2-terminal graph.
\begin{itemize}
   \item[(a)]  If $G$ contains an $st$-path, then 
      $v_{\rm eff}(G,s,t)$ is a rational function
      of $q$ and $\{v_e\}$ that depends nontrivially on $\{v_e\}$.
   \item[(b)]  If $G$ does not contain an $st$-path, then 
      $v_{\rm eff}(G,s,t) \equiv 0$.
\end{itemize}
\end{lemma}

\begin{proof}
(a) If $G$ contains an $st$-path, then $B_{G,s,t}(q,\bv) \not\equiv 0$,
and every monomial in $B_{G,s,t}(q,\bv)$ contains at least one factor $v_e$.
On the other hand, $A_{G,s,t}(q,\bv)$ contains a monomial $q^{|V(G)|-2}$
(coming from $E' = \emptyset$) that contains no factors $v_e$.
Therefore, it cannot happen that
$B_{G,s,t}(q,\bv) = f(q) \, A_{G,s,t}(q,\bv)$.

(b) is trivial.
\qed
\end{proof}

\bigskip

{\bf Remarks.}
1.  The ``effective transmissivity''
$t_{{\rm eff}} \equiv v_{{\rm eff}} / (q + v_{{\rm eff}})$
is given by the simple formula
\begin{equation}
   t_{{\rm eff}}(G,s,t)
%%    \;=\;  {B_{G,s,t}(\qvbf) \over q A_{G,s,t}(\qvbf) + B_{G,s,t}(\qvbf)}
    \;=\;  {\zgst(q,\bv) \over Z_G(q,\bv)}
\end{equation}
and thus represents the ``probability'' that $s$ is connected to $t$.
In fact, when $\bv \ge 0$ this is a true probability in the
random-cluster model \cite{Grimmett_06}.

2.  If $G$ is a graph and $s,t$ are distinct vertices of $G$,
we define $G/st$ to be the graph in which $s$ and $t$ are contracted
to a single vertex.  (N.B.:  If $G$ contains one or more edges $st$,
then these edges are {\em not}\/ deleted, but become loops in $G/st$.)
%% There is a canonical one-to-one correspondence between the edges of $G$
%% and the edges of $G/st$;  for simplicity (though by slight abuse of notation)
%% we denote an edge of $G$ and the corresponding edge of $G/st$
%% by the same letter.  In particular, we can apply a given set of edge weights
%% $\{ v_e \} _{e \in E}$ to both $G$ and $G/st$.
It is then easy to see that
\begin{equation}
Z_{G/st}(q,\bv)  \;=\;  \zgst(q,\bv) \,+\, q^{-1} \zgsnt(q,\bv)
\;.
\label{eq2.Gst}
\end{equation}
One convenient way of calculating $\zgst$ and $\zgsnt$
is to first calculate $Z_G$ and $Z_{G/st}$
(for instance, by deletion-contraction)
and then solve \reff{eq2.G}/\reff{eq2.Gst} for $\zgst$ and $\zgsnt$.
See \cite[Section~4.6]{MR2187739} for more information
on the partial multivariate Tutte polynomials.
\qed

\medskip

Let us now justify the name $v_{\rm eff}$ by showing that
when $(G,s,t)$ is inserted inside a larger graph,
it acts essentially (modulo a prefactor) as a single edge
with effective weight $v_{\rm eff}(G,s,t)$.
The precise construction is as follows:
Let $H$ be a graph, and let $e_\star$ be an edge of $H$
with endpoints $a$ and $b$.\footnote{
   The result of Proposition~\ref{prop.veff.blackbox} below
   is valid even when $a=b$ (i.e.\ $e_\star$ is a loop),
   although we will never use it in this situation.
}
Let us denote by $H[(e_\star,a,b) \to (G,s,t)]$
the graph obtained from the disjoint union of $H \setminus e_\star$ and $G$
by identifying  $s$ with $a$ and $t$ with $b$.
So the edge set of $H[(e_\star,a,b) \to (G,s,t)]$
can be identified with $(E(H) \setminus \{e_\star\}) \cup E(G)$.
Now put weights $\bv = \{v_e\}_{e \in E(H)}$ on the edges of $H$
and weights $\bw = \{w_e\}_{e \in E(G)}$ on the edges of $G$,
so that $v_{\rm eff}(G,s,t) = B_{G,s,t}(q,\bw)/A_{G,s,t}(q,\bw)$
is a rational function of $q$ and $\bw$.
We use the notation
$\bv_{\neq e_\star} = \{v_e\}_{e \in E(H) \setminus \{e_\star\}}$
and hence $Z_H(q,\bv) = Z_H(q,\bv_{\neq e_\star},v_{e_\star})$.
We then have:

\begin{proposition}
   \label{prop.veff.blackbox}
When a 2-terminal graph $(G,s,t)$ is inserted into a graph $H$ as above,
\be
   Z_{H[(e_\star,a,b) \to (G,s,t)]}(q,\bv_{\neq e_\star},\bw)
   \;=\;
   A_{G,s,t}(q,\bw) \:
   Z_H(q,\bv_{\neq e_\star},v_{\rm eff}(G,s,t))
   \;.
 \label{eq.prop.veff.blackbox}
\ee
\end{proposition}

\begin{proof}
The sets $A \subseteq (E(H) \setminus \{e_\star\}) \cup E(G)$
contributing to the multivariate Tutte polynomial \reff{multivariatetutte}
of $H[(e_\star,a,b) \to (G,s,t)]$
can be classified according to whether $a$ is or is not connected to $b$
via edges in $E(G)$.
Those that do not connect $a$ to $b$ give a factor $A_{G,s,t}(q,\bw)$
and correspond to the sets $A' \not\ni e_\star$
contributing to the multivariate Tutte polynomial \reff{multivariatetutte}
of $H$,
while those that connect $a$ to $b$ give a factor $B_{G,s,t}(q,\bw)$
and correspond to the sets $A' \ni e_\star$
contributing to the multivariate Tutte polynomial \reff{multivariatetutte}
of $H$.
Since $v_{\rm eff}(G,s,t) = B_{G,s,t}(q,\bw)/A_{G,s,t}(q,\bw)$,
the formula \reff{eq.prop.veff.blackbox} is an immediate consequence
of this correspondence.
\qed
\end{proof}

\medskip

{\bf Remarks.}
1.  The graphical construction of inserting $(G,s,t)$ inside $H$
depends on the chosen order of endpoints for the edge $e_\star$,
but the resulting multivariate Tutte polynomial does not.
That is, $H[(e_\star,a,b) \to (G,s,t)]$ and $H[(e_\star,b,a) \to (G,s,t)]$
are in general nonisomorphic as graphs,
but Proposition~\ref{prop.veff.blackbox}
shows that they have the same multivariate Tutte polynomial.

2.  The formula \reff{eq.series2} for series reduction is
a special case of \reff{eq.prop.veff.blackbox},
in which the inserted graph $(G,s,t)$ is a two-edge path.

\subsection{Parallel and series composition of 2-terminal graphs}
\label{subsec.tutte.serpar.2term}

If $\sfG_1 = (G_1,s_1,t_1)$ and $\sfG_2 = (G_2,s_2,t_2)$ 
are 2-terminal graphs on disjoint vertex sets,
then their {\em parallel composition}\/ is the 2-terminal graph
\begin{equation}
\sfG_1 \pll \sfG_2 \;=\; (H,s_1, t_1)
\end{equation}
where $H$ is obtained from $G_1 \cup G_2$ by identifying
$s_2$ with $s_1$ and $t_2$ with $t_1$,
and their {\em series composition}\/ is the 2-terminal graph
\begin{equation}
\sfG_1 \ser \sfG_2 \;=\; (H, s_1, t_2)
\end{equation}
where $H$ is obtained from $G_1 \cup G_2$ by identifying $t_1$ with $s_2$.
For future use (see Section~\ref{subsec.decomp}),
let us say that a 2-terminal graph is {\em prime}\/
if it cannot be written as the parallel or series composition of
two strictly smaller 2-terminal graphs.\footnote{
   We say ``{\em strictly}\/ smaller'' because every 2-terminal graph $\sfG$
   can be written as $\sfG = \sfG \pll \sfKbar_2$
   where $\sfKbar_2$ is the graph with two vertices (the terminals)
   and no edges.  It is to exclude this trivial type of parallel composition
   that we write ``each have at least one edge''
   in Lemmas~\ref{lem_niceness_0}(d) and \ref{lem_serpar_0}(c).
   In Section~\ref{subsec.decomp} and thereafter,
   this trivial case will be excluded by requiring that
   all graphs appearing in a decomposition tree be connected.
   We could avoid all these technicalities by requiring connectedness
   from the start, but we refrain from doing so because connectedness
   plays no role in the formulae of the present section.
 \label{footnote_trivial_parallel}
}

\bigskip

{\bf Remark.}  We trivially have $\sfG_1 \pll \sfG_2 = \sfG_2 \pll \sfG_1$.
On the other hand, $\sfG_1 \ser \sfG_2 \neq \sfG_2 \ser \sfG_1$,
if only because the terminals are different in the two cases;
moreover, even the graphs underlying
$\sfG_1 \ser \sfG_2$ and $\sfG_2 \ser \sfG_1$
(ignoring the terminals) need not be isomorphic,
as can be seen by simple examples.
But this subtlety will play no role in this paper,
because $\sfG_1 \ser \sfG_2$ and $\sfG_2 \ser \sfG_1$
will have the same multivariate Tutte polynomial;
indeed, they will have the same partial multivariate Tutte polynomials
\reff{def.zgsnt}/\reff{def.zgst} and hence also the same $v_{\rm eff}$.
This is a reflection of the fact that the multivariate Tutte polynomial
of a graph $G$ depends only on the graphic matroid $M(G)$
[except for an overall prefactor $q^{|V(G)|}$]
and that series connection of {\em matroids}\/
does not depend on any orientation.
\qed

\bigskip

Let us now show how the partial multivariate Tutte polynomials
$\zgsnt$ and $\zgst$ of a parallel or series composition
of 2-terminal graphs $(G_1,s_1,t_1)$ and $(G_2,s_2,t_2)$
can be computed from the partial multivariate Tutte polynomials
of the two input graphs.
It is convenient to use the modified partial multivariate Tutte polynomials
$A_{G,s,t}$ and $B_{G,s,t}$ defined in \reff{def.AGst}/\reff{def.BGst}.

\begin{proposition}
   \label{prop.serpar.AB}
\quad \par
\begin{itemize}
\item[(a)]  Consider a parallel composition
$(G,s,t) \,=\, (G_1,s_1,t_1) \,\pll\, (G_2,s_2,t_2)$.
Writing $A = A_{G,s,t}$ and $A_i = A_{G_i,s_i,t_i}$ for $i=1,2$
and likewise for $B$, we have
\begin{eqnarray}
A    & = &   A_1 A_2   
\label{eq.par.A12}  \\[2mm]
B    & = &   A_1 B_2 + A_2 B_1 + B_1 B_2
\label{eq.par.B12}
\end{eqnarray}
and in particular
\be
   \hspace*{-1.2cm}
   A+B  \;=\;  (A_1 + B_1) (A_2 + B_2)
\label{eq.par.A+B}
\ee
and
\be
  \,
v_{\rm eff}(G,s,t)  \;=\;
v_{\rm eff}(G_1,s_1,t_1) \,\pll\, v_{\rm eff}(G_2,s_2,t_2)
\;.
\label{eq.par.veff}
\ee
\item[(b)]  Consider a series composition
$(G,s,t) \,=\, (G_1,s_1,t_1) \,\ser\, (G_2,s_2,t_2)$.
Writing $A = A_{G,s,t}$ and $A_i = A_{G_i,s_i,t_i}$ for $i=1,2$
and likewise for $B$, we have
\begin{eqnarray}
A  & = &   A_1 B_2 + A_2 B_1 + q A_1 A_2
\label{eq.ser.A12}  \\[2mm]
B  & = &   B_1 B_2
\label{eq.ser.B12}
\end{eqnarray}
and in particular
\be
   \hspace*{-1.2cm}
   qA+B  \;=\;  (qA_1 + B_1) (qA_2 + B_2)
\label{eq.ser.qA+B}
\ee
and
\be
  \:
v_{\rm eff}(G,s,t)  \;=\;
v_{\rm eff}(G_1,s_1,t_1) \,\serq\, v_{\rm eff}(G_2,s_2,t_2)
\;.
\label{eq.ser.veff}
\ee
\end{itemize}
\end{proposition}

\begin{proof}
We recall that $A_{G,s,t}(q, \bv)$ and $B_{G,s,t}(q, \bv)$ are defined by
sums like \reff{def.zgsnt}/\reff{def.zgst} but in which
only those connected components
not containing one or both of the terminals $s,t$
(let us call these ``non-terminal components'') receive a factor $q$.

For a parallel composition,
$s$ is connected to $t$ in a spanning subgraph of $G$
if and only if it is connected in the corresponding spanning subgraph
of $G_1$ or $G_2$ or both;
and the number of non-terminal components in $G$
is the sum of those in $G_1$ and $G_2$.
This proves \reff{eq.par.A12}/\reff{eq.par.B12};
then \reff{eq.par.A+B} and \reff{eq.par.veff} are an immediate consequence.

For a series composition,
$s$ is connected to $t$ in a spanning subgraph of $G$
if and only if $s_i$ is connected to $t_i$
in the corresponding spanning subgraph of $G_i$
for {\em both}\/ $i=1$ and $i=2$;
and the number of non-terminal components in $G$
is the sum of those in $G_1$ and $G_2$
{\em except that}\/
there is an extra non-terminal component containing the
``inner terminal'' $s_2 = t_1$
whenever $s_i$ is disconnected from $t_i$ in $G_i$
for {\em both}\/ $i=1$ and $i=2$
[this explains the factor $q$ in front of $A_1 A_2$ in \reff{eq.ser.A12}].
This proves \reff{eq.ser.A12}/\reff{eq.ser.B12};
then \reff{eq.ser.qA+B} and \reff{eq.ser.veff} are an immediate consequence.
\qed
\end{proof}

Of course, it is no accident that $v_{\rm eff}$ satisfies
\reff{eq.par.veff} and \reff{eq.ser.veff}
under parallel and series composition:
by Proposition~\ref{prop.veff.blackbox},
$v_{\rm eff}$ {\em must}\/ behave under parallel and series composition
exactly like the parallel and series connection of single edges.
Indeed, this argument gives an alternate way of proving
\reff{eq.par.veff} and \reff{eq.ser.veff}.

\section{Series-parallel graphs and decomposition trees}  \label{sec.serpar}

In this section we begin by making some further remarks on series and parallel
composition of 2-terminal graphs (Section~\ref{subsec.nice});
we then discuss series-parallel graphs (Section~\ref{subsec.serpargraphs}),
decomposition trees for 2-terminal graphs (Section~\ref{subsec.decomp}),
and the use of decomposition trees to compute the multivariate
Tutte polynomial (Section~\ref{subsec.computing}).
Finally, we introduce an important family of example graphs,
the leaf-joined trees (Section~\ref{subsec.leaf-joined}).

Before starting, however,
% Before defining ``series-parallel'' for graphs {\em tout court}\/,
%% (as opposed to 2-terminal graphs),
we need to clarify our usage of the term ``nonseparable''
as concerns graphs with loops.
%% and for this purpose it is useful to turn to
%% the terminology of matroid theory.
%% So let us define a graph with at least one edge
%% to be {\em nonseparable}\/ if its cycle matroid is connected,
%% and {\em separable}\/ otherwise.
% {\bf But this ``matroidal'' definition allows $G$ to contain
% isolated vertices, which we presumably want to exclude!!!
% I think we might do better to be explicit and say that
% $G$ is separable if it can be decomposed into two nontrivial pieces by
% either disjoint union or gluing at a vertex (and saying precisely
% what we mean by ``nontrivial'').}
So let us call a graph {\em separable}\/ if it is either disconnected
or can be obtained by gluing at a vertex two graphs that each have
at least one edge;  otherwise we call it {\em nonseparable}\/.
Equivalently, a graph is nonseparable if it is either 
a single vertex with no edges,
a single vertex with a single loop,
a pair of vertices connected by one or more edges,
or a 2-connected graph.
Note in particular that, in our definition,
a nonseparable graph must be loopless
unless it consists of a single vertex with a single loop.
(By contrast, the usual definition of ``separable'' for connected graphs ---
namely, a graph with a cut-vertex --- deems a single vertex with
multiple loops to be nonseparable.
This definition has the disadvantage of not being invariant under
planar duality.)
Our definition of ``nonseparable'' agrees with
the usual definition when restricted to loopless graphs.

\subsection{Nice 2-terminal graphs}  \label{subsec.nice}

As preparation for a more detailed study of
series and parallel composition of 2-terminal graphs,
we wish to single out a class of 2-terminal graphs that are ``well behaved''
in the sense that they connect the terminals without containing
``dangling ends''.
More precisely, let us say that a 2-terminal graph $(G,s,t)$ is {\em nice}\/
if $G$ is connected and $G+st$ is nonseparable.
(Here $G+st$ denotes the graph obtained from $G$ by adding a new edge
 from $s$ to $t$, irrespective of whether or not such an edge was
 already present.)
Equivalently, $(G,s,t)$ is nice if either
$G$ is nonseparable or else $G$ is a block path (with more than one block)
in which $s$ lies in one endblock and $t$ in the other
and neither of them is a cut vertex.
In the latter case $(G,s,t)$ can be written uniquely
as a series composition $\sfH_1 \ser \sfH_2 \ser \cdots \ser \sfH_k$
where $k \ge 2$ and all the $\sfH_i$ are nonseparable.\footnote{
   Saying ``$\sfH_i$ is nonseparable'' is a convenient shorthand for
   the more precise but pedantic statement
   ``$\sfH_i = (H_i,s_i,t_i)$ with $H_i$ nonseparable''.
   In what follows we shall repeatedly use this shorthand
   in order to avoid ponderous locutions.
}
Conversely, if $(G,s,t)$ is nice and {\em not}\/ the series composition
of two smaller 2-terminal graphs, then $G$ must be nonseparable.

The following facts are easily verified:

\begin{lem}
  \label{lem_niceness_0}
Let $\sfG_1 = (G_1,s_1,t_1)$ and $\sfG_2 = (G_2,s_2,t_2)$ 
be 2-terminal graphs.  Then:
\begin{itemize}
   \item[(a)]  The series composition $\sfG_1 \ser \sfG_2$
      is always separable.
   \item[(b)]  The series composition $\sfG_1 \ser \sfG_2$ is nice
      if and only if both $\sfG_1$ and $\sfG_2$ are nice.
   \item[(c)]  If $\sfG_1$ and $\sfG_2$ are nice,
      then the parallel composition $\sfG_1 \pll \sfG_2$
      is nonseparable (and hence also nice).
   \item[(d)]  Conversely, if $G_1$ and $G_2$ each have at least one edge
      and the parallel composition $\sfG_1 \pll \sfG_2$ is nice,
      then $\sfG_1$ and $\sfG_2$ are both nice
      (and hence $\sfG_1 \pll \sfG_2$ is nonseparable).
\end{itemize}
In particular, any 2-terminal graph formed by successive
series and parallel compositions of nice 2-terminal graphs is nice.
\end{lem}

\subsection{Series-parallel graphs}  \label{subsec.serpargraphs}

In the literature one can find two slightly different concepts of
``series-parallel graph'':
one applying to graphs, and the other applying to 2-terminal graphs.
In this paper we shall need to use {\em both}\/ of these concepts.
We therefore begin by reviewing the two definitions
and the theorems relating them.

In Section~\ref{subsec.tutte.serpar.2term}
we defined the parallel and series composition of 2-terminal graphs.
We now define a {\em 2-terminal series-parallel graph}\/
to be a 2-terminal graph that is either $K_2$
(with the two vertices as terminals)
or else the parallel or series composition of
two smaller 2-terminal series-parallel graphs.
Note that a 2-terminal series-parallel graph is always loopless.
Note also that if $(G,s,t)$ is 2-terminal series-parallel,
then it is nice, i.e.\ $G$ is connected and $G+st$ is nonseparable:
this is an immediate consequence of Lemma~\ref{lem_niceness_0}
and the fact that $K_2$ is nice.

For 2-terminal series-parallel graphs
we have the following analogue of Lemma~\ref{lem_niceness_0}:

\begin{lem}
  \label{lem_serpar_0}
Let $\sfG_1 = (G_1,s_1,t_1)$ and $\sfG_2 = (G_2,s_2,t_2)$ 
be 2-terminal graphs.  Then:
\begin{itemize}
   \item[(a)]  The series composition $\sfG_1 \ser \sfG_2$
      is 2-terminal series-parallel
      if and only if both $\sfG_1$ and $\sfG_2$
      are 2-terminal series-parallel.
   \item[(b)]  If $\sfG_1$ and $\sfG_2$ are 2-terminal series-parallel,
      then the parallel composition $\sfG_1 \pll \sfG_2$
      is 2-terminal series-parallel.
   \item[(c)]  Conversely, if $G_1$ and $G_2$ each have at least one edge
      and the parallel composition $\sfG_1 \pll \sfG_2$ is
      2-terminal series-parallel,
      then $\sfG_1$ and $\sfG_2$ are 2-terminal series-parallel.
\end{itemize}
\end{lem}

\begin{proof}
The ``if'' part of (a) is obvious.
For the ``only if'', we observe that if $\sfG = \sfG_1 \ser \sfG_2$
is 2-terminal series-parallel, then it is nice and separable
and hence can be written uniquely as
$\sfH_1 \ser \sfH_2 \ser \cdots \ser \sfH_k$ with $k \ge 2$
and all the $\sfH_i$ nonseparable;
moreover, we must have $\sfG_1 = \sfH_1 \ser \cdots \ser \sfH_\ell$
and $\sfG_2 = \sfH_{\ell+1} \ser \cdots \ser \sfH_k$ for some $\ell$.
We now claim that all the $\sfH_i$ are 2-terminal series-parallel
(so that $\sfG_1$ and $\sfG_2$ are as well),
and we shall prove this by induction on $k$.
If $k=2$, the last operation in the series-parallel construction of $\sfG$
must have been the series connection of $\sfH_1$ with $\sfH_2$,
so $\sfH_1$ and $\sfH_2$ must be 2-terminal series-parallel.
If $k>2$, then the last operation in the series-parallel construction of $\sfG$
must have been the series connection of
$\sfH_1 \ser \cdots \ser \sfH_m$ with
$\sfH_{m+1} \ser \cdots \ser \sfH_k$ for some $m$,
so both of these must be 2-terminal series-parallel;
and since they each have less than $k$ blocks, the inductive hypothesis
implies that all the $\sfH_i$ are 2-terminal series-parallel.

(b) is obvious.
For (c), we observe that $\sfG = \sfG_1 \pll \sfG_2$ is nice,
and hence by Lemma~\ref{lem_niceness_0}(d) it is nonseparable.
Now any nonseparable 2-terminal graph $\sfG$
can be written uniquely (modulo ordering) as
$\sfG = \sfH_1 \pll \sfH_2 \pll \cdots  \pll \sfH_k$
where none of the $\sfH_i$ can be further decomposed as a
nontrivial parallel composition.\footnote{
   We say that a parallel composition is {\em nontrivial}\/
   if each of the graphs occurring in it has at least one edge.
   See footnote~\ref{footnote_trivial_parallel} above.
}
(The summands $\sfH_i$ are the $st$-bridges in $G$.)
An argument essentially identical to the one used in part (a)
shows that if $\sfG$ is 2-terminal series-parallel,
then all the $\sfH_i$ are 2-terminal series-parallel,
and moreover  $\sfG_1$ and $\sfG_2$ are obtained by
parallel composition of some (complementary) nonempty subsets of the $\sfH_i$.
\qed
\end{proof}

\bigskip

Let us now turn to the definition of ``series-parallel graph''
{\em tout court}\/.
Unfortunately, there seems to be no completely standard
definition of ``series-parallel graph'';
a plethora of slightly different definitions can be found in the literature
\cite{MR0175809,Colbourn_87,MR849395,Oxley_92,MR1686154,Royle-Sokal}.
So let us be completely precise about our own usage:
we shall call a loopless graph {\em series-parallel}\/
if it can be obtained from a forest by a finite (possibly empty)
sequence of series and parallel extensions of edges
(i.e.\ replacing an edge by two edges in series or two edges in parallel).
We shall call a general graph (allowing loops) series-parallel
if its underlying loopless graph is series-parallel.
Some authors write ``obtained from a tree'', ``obtained from $K_2$''
or ``obtained from $C_2$'' in place of ``obtained from a forest'';
in our terminology these definitions yield, respectively,
all {\em connected}\/ series-parallel graphs,
all connected series-parallel graphs whose blocks form a path,
or all {\em nonseparable}\/ series-parallel graphs
with the exception of $K_2$.
See \cite[Section 11.2]{MR1686154} for a more extensive bibliography.

%  A nonseparable graph is called a {\em series-parallel graph}
%  if it is a loop or if it has distinct vertices $s$ and $t$ such
%  that $(G,s,t)$ is a 2-terminal series-parallel graph. Finally,
%  a (possibly separable) graph is called a series-parallel
%  graph if each of its {\em blocks}, i.e. maximal nonseparable
%  subgraphs, is a nonseparable series-parallel graph.
% 
% We note that there are a number of alternative definitions of
% series-parallel graph in the literature. These give essentially the
% same class, though they often do not explicitly address loops and/or
% separable graphs. For example, one common definition is that a graph
% is series-parallel if it can be obtained from $K_2$ by repeatedly
% replacing an edge by two edges in series, or by two edges in parallel,
% and clearly every graph in this class is loopless. If $G$ is a graph
% constructed in this fashion, then $G$ is nonseparable if and only if
% the first operation creates a double edge from the initial $K_2$. In
% this situation, the class of graphs constructed coincides with the
% loopless nonseparable series-parallel graphs as defined above. If the
% first operation is a series replacement then the resulting graph is
% separable and consists of loopless series-parallel blocks arranged in a
% path (see \cite{MR1686154}). In any case, the definitions all coincide
% for loopless nonseparable graphs with at least one edge.

The precise relationship between the 2-terminal and pure-graph
definitions of ``series-parallel'' is given by the following theorem,
which follows from results of Duffin \cite{MR0175809}
(see also Oxley \cite{MR849395}):

\begin{thm}
   \label{thm_serpar_1}
If $G$ is a loopless {\em nonseparable}\/ graph with at least one edge,
then the following are equivalent:
\renewcommand{\labelenumi}{(\arabic{enumi})}
\begin{enumerate}
\setlength{\itemsep}{0pt}
\item $G$ is series-parallel.
\item $(G,s,t)$ is 2-terminal series-parallel for {\em some} pair of
  vertices $s$, $t$.
\item $(G,s,t)$ is 2-terminal series-parallel for {\em every} pair of
 {\em adjacent}\/ vertices $s$, $t$.
\end{enumerate}
\end{thm}

One useful consequence of Theorem~\ref{thm_serpar_1} is the following:

\begin{cor}
   \label{cor_serpar_2}
Let $(G,s,t)$ be a 2-terminal graph,
where $G$ is loopless and has at least one edge, and $G+st$ is nonseparable.
%% [Here $G+st$ denotes the graph obtained from $G$ by adding a new edge
%%  from $s$ to $t$, irrespective of whether or not such an edge was
%%  already present.]
Then the following are equivalent:
\renewcommand{\labelenumi}{(\arabic{enumi})}
\begin{enumerate}
\setlength{\itemsep}{0pt}
\item $(G,s,t)$ is 2-terminal series-parallel.
\item $(G+st,s,t)$ is 2-terminal series-parallel.
\item $G+st$ is series-parallel.
\end{enumerate}
\end{cor}

\begin{proof}
Applying Theorem~\ref{thm_serpar_1} to $G+st$
proves the equivalence of (2) and (3).
Moreover, (1) $\implies$ (2) is trivial,
and (2) $\implies$ (1) is a special case of Lemma~\ref{lem_serpar_0}(c).
%
% we observe that $G+st$ cannot be $K_2$ (because $G$ has at least one edge),
% nor can $(G+st,s,t)$ be a series composition (because $G+st$ is nonseparable);
% so $(G+st,s,t)$ must be a parallel composition of two smaller 2-terminal
% series-parallel graphs.
% One of the these graphs must contain the edge $st$;
% this graph must be nonseparable,
% {\bf Why?????}
% so if it is not $st$ alone, then the same argument shows that
% it must be a parallel composition of two smaller 2-terminal
% series-parallel graphs.
% Continuing this ``peeling off'' process, we eventually reach a stage
% where one of the parallel summands is $st$ alone.
% Then the parallel composition of all the other summands
% must be $(G,s,t)$.
% {\bf Is this proof correct?????}
\qed
\end{proof}

\medskip

The reason for using the 2-terminal notion of series-parallel graph
in this paper
is that, although we are unable to precisely control the maxmaxflow
of a series-parallel graph, we {\em can}\/ control the flow
between its terminals via the following trivial fact:

\begin{lem}\label{lem_terminals}
Let $(G_1,s_1,t_1)$ and $(G_2,s_2,t_2)$ be 2-terminal graphs
(not necessarily series-parallel). Then 
\begin{eqnarray}
\lambda_{G_1 \ser G_2} (s,t) & = &
  \min[\lambda_{G_1}(s_1,t_1),\, \lambda_{G_2}(s_2,t_2)] \\[1mm]
\lambda_{G_1 \pll G_2}(s,t) &=&
  \lambda_{G_1}(s_1,t_1) \,+\, \lambda_{G_2}(s_2,t_2)
\end{eqnarray}
where $s$ and $t$ denote the terminals of
$G_1 \ser G_2$ and $G_1 \pll G_2$, respectively.
[Recall that $\lambda_G(x,y)$ denotes the maximum flow in $G$ from $x$ to $y$,
 as defined in \reff{def.lambdaG}.]
\end{lem}

\subsection{Decomposition trees}  \label{subsec.decomp}

Let $\sfG = (G,s,t)$ be a 2-terminal graph,
where we now assume that $G$ is connected and loopless.
A {\em decomposition tree}\/ for $(G,s,t)$ 
is a rooted binary tree with three types of nodes ---
called {\em $s$-nodes}\/, {\em $p$-nodes}\/ and {\em leaf nodes}\/ ---
in which the children of each $s$-node are ordered,
and each node is a {\em connected}\/ 2-terminal graph
(whose underlying graph is a subgraph of $G$), as follows:
%{\bf Check also \url{http://www.springerlink.com/content/mwx718l48326200k/}
%   This is \emph{Lecture Notes in Comp Sci}\/ \#299 (1988).
%   And check MathSciNet for other uses of ``decomposition tree''
%   and ``2-terminal graph''.}
The root node is $\sfG$;
if $\sfH$ is an $s$-node and $\sfH_1$ and $\sfH_2$ are its children (in order),
then $\sfH = \sfH_1 \ser \sfH_2$;
if $\sfH$ is a $p$-node and $\sfH_1$ and $\sfH_2$ are its children
 (in either order),
then $\sfH = \sfH_1 \pll \sfH_2$;
and if $\sfH$ is a leaf node, then it has no children.\footnote{
   The concept of a decomposition tree for a 2-terminal graph
   is very natural and has been used sporadically in the literature,
   albeit with no standard definition.
   Brandst\"adt, Le and Sprinrad \cite[Section~11.2]{MR1686154}
   define decomposition trees essentially as we do,
   but only for series-parallel graphs.
   Bodlaender and van Antwerpen - de Fluiter \cite{Bodlaender_01}
   likewise define decomposition trees for series-parallel graphs,
   with a definition that differs slightly from ours by allowing
   non-binary trees (see Remark 2 below).
   %% {\bf Or maybe:  Many authors [GIVE LIST] define decomposition trees
   %%    \emph{for series-parallel graphs},
   %%    either as we do or in the variant that allows  non-binary trees
   %%    (see Remark 2 below).}
   Bern {\em at al.}\/ \cite{Bern_87}
   and Borie {\em at al.}\/ \cite{Borie_92}
   define decomposition trees in the more general setting of
   $k$-terminal graphs for any fixed $k$;
   their definitions specialized to $k=2$ are almost the same as ours.
   (Borie {\em at al.}\/ require the graphs at leaf nodes
    to have no nonterminal vertices --- something we do not wish to do,
    as it would restrict us to series-parallel graphs only ---
    but they immediately add \cite[p.~558]{Borie_92} that
    ``this could be generalized to permit additional base graphs''.)
   See also Spinrad \cite[Section~11.3]{Spinrad_03}
   for a brief description of this latter work.
%%{\bf    Check also Wimer and Hedetniemi,
%    MR0988649 (90d:05183):   Congr. Numer.  63  (1988), 161--176.
%   I have requested this on NYU inter-library loan (22 July 2010).
%   And check also Eppstein,
%   MR1161075 (92m:05180):  Inform. and Comput.  98  (1992),  no. 1, 41--55;
%   He, MR1114919 (92e:68152):  J. Algorithms  12  (1991),  no. 3, 409--430.}

%%    {\bf What about the earlier work of Bern, Lawler and Wong
%%     that appears to do a similar thing???
%%    See MR0890873 (88i:68084):  J. Algorithms  8  (1987), 216--235.
%%    They seem to have been the first.
%%    And check also Wimer and Hedetniemi,
%%    MR0988649 (90d:05183):   Congr. Numer.  63  (1988), 161--176;
%%    He and Yesha,
%%    MR0925600 (89e:68050):   J. Algorithms  9  (1988),  92--113;
%%    Mahajan and Peters,
%%    MR1300249 (95h:05051):  Discrete Appl. Math.  54  (1994), 229--250.}
  \label{footnote_decomposition}
}
%% If $G = G_1 \ser G_2$, then the root of the tree is an $s$-node
%% whose children (in order) are decomposition trees of $G_1$ and $G_2$;
%% if $G = G_1 \pll G_2$, then the root of the tree is a $p$-node
%% whose children (in either order) are decomposition trees of $G_1$ and $G_2$;
%% while if $G = K_2$, then the root is a leaf node. 
% Note that the multivariate Tutte polynomial of each node
% of the decomposition tree is invariant under permuting
% the order of the children and so the ordering conditions are irrelevant
% for our purposes, but as isomorphism class is {\em not} invariant
% under reordering we have used the standard definition.
If $G$ is edge-weighted, then the graph at each node
is also edge-weighted with the weights inherited from its parent.
The graphs that appear as nodes in this decomposition tree are called
the {\em constituents}\/ of $\sfG$
(with respect to the particular decomposition tree),
and a constituent is {\em proper}\/
if it is not equal to $\sfG$.  

A given 2-terminal graph $\sfG = (G,s,t)$
can have many distinct decomposition trees,
and this for two separate reasons.
Firstly, one is free to stop the decomposition at any stage.
Indeed, in the extreme case the decomposition tree can consist of
the single node $\sfG$ (which is then a leaf node);
we call this the {\em trivial}\/ decomposition tree.
At the other extreme,
we say that a decomposition tree is {\em maximal}\/
%% {\bf I think this may be a better word than ``complete'' ---
%%   what do you think???}
if each leaf node corresponds to a {\em prime}\/ 2-terminal graph.
Secondly, if $\sfG$ or one of its constituents is formed by placing
{\em three}\/ or more 2-terminal graphs in series or in parallel,
then these may be paired up in various ways.
(This nonuniqueness arises from our insistence that a decomposition tree
 is a {\em binary}\/ tree.)

\bigskip

{\bf Remarks.}
1.  The order of the children at an $s$-node is important to
reconstructing the graph (since $\sfG_1 \ser \sfG_2 \neq \sfG_2 \ser \sfG_1$)
but is irrelevant to the multivariate Tutte polynomial.

%\comment{Do we distinguish between $(G,s,t)$ and $(G,t,s)$ Or do we need to keep track of ``which way round'' the graphs are connected when parallel connection is done?}

2.  We have insisted here that the decomposition tree be a {\em binary}\/ tree:
this means that we need only consider parallel or series composition
of {\em pairs}\/ of 2-terminal subgraphs, but it also means that the
maximal decomposition tree is nonunique whenever $G$ or one of its constituents
is formed by placing {\em three}\/ or more 2-terminal graphs
in series or in parallel, since these may be paired up in various ways.
Alternatively, we could allow the decomposition tree to be a {\em general}\/
rooted tree:
then the maximal decomposition tree would be unique,
but we would need to consider consider parallel and series composition
of an arbitrary number of 2-terminal subgraphs.
Six of one, half dozen of the other.

3.  Many authors have defined and applied decomposition trees
for 2-terminal {\em series-parallel}\/ graphs
(see footnote~\ref{footnote_decomposition} above);
and most of the present paper is indeed concerned with this special case
(Theorems~\ref{thm_main} and \ref{thm_main_AF}).
But the technique set forth here is more general,
and applies to graphs that are not series-parallel
but are nevertheless built up by using series and parallel compositions
from a fixed starting set of 2-terminal ``base graphs''
(see Section~\ref{sec.Wheatstone}).
A simple example of such a result is Theorem~\ref{thm_main_Wheatstone},
where the set of base graphs is taken to be $K_2$ and the Wheatstone bridge.
It is for this reason that we have developed the theory of decomposition
trees for 2-terminal graphs that are not necessarily series-parallel.
\qed

\bigskip

% {\bf Somewhere we need to explain that the root is a $p$-node
%    if and only if the graph is nonseparable.
%    CAREFUL!!!  Root is $p$-node $\implies$ $G$ is nonseparable;
%    root is $s$-node $\implies$ $G$ is separable;
%    but root is leaf node $\implies$ nothing in general,
%    though it implies that $G$ is nonseparable if the decomposition tree
%    is maximal.  Furthermore, all these results seem to require the
%    hypothesis that $G+st$ is nonseparable --- otherwise $G$ could
%    contain ``dangling ends''.}

We have the following basic facts concerning the structure
of decomposition trees:

\begin{lem}
  \label{lem_niceness}
Let $(G,s,t)$ be a 2-terminal graph,
with $G$ connected and loopless, and fix a decomposition tree for it.
Then the following are equivalent:
\begin{itemize}
\setlength{\itemsep}{0pt}
   \item[(1)]  The root node $(G,s,t)$ is nice.
   \item[(2)]  Every node is nice.
   \item[(3)]  Every leaf node is nice.
\end{itemize}
Moreover, when these equivalent conditions hold,
every $p$-node is nonseparable and every $s$-node is separable.
%% and in particular $G$ is nonseparable (resp.\ separable)
%% if the root node is a $p$-node (resp.\ an $s$-node).
\end{lem}

\begin{proof}
This is an immediate consequence of Lemma~\ref{lem_niceness_0}.
\qed
\end{proof}

An analogous result holds for 2-terminal series-parallel graphs:

\begin{lem}
  \label{lem_serpar}
Let $(G,s,t)$ be a 2-terminal graph,
with $G$ connected and loopless, and fix a decomposition tree for it.
Then the following are equivalent:
\begin{itemize}
\setlength{\itemsep}{0pt}
   \item[(1)]  $(G,s,t)$ is 2-terminal series-parallel.
   \item[(2)]  Every node is 2-terminal series-parallel.
   \item[(3)]  Every leaf node is 2-terminal series-parallel.
\end{itemize}
\end{lem}

\begin{proof}
This is an immediate consequence of Lemma~\ref{lem_serpar_0}.
\qed
\end{proof}

Among 2-terminal graphs, the series-parallel ones
can be characterized as follows:

\begin{lem}
  \label{lem_SP_decomposition_tree}
Let $(G,s,t)$ be a 2-terminal graph.  Then the following are equivalent:
\renewcommand{\labelenumi}{(\arabic{enumi})}
\begin{enumerate}
\setlength{\itemsep}{0pt}
\item $(G,s,t)$ is 2-terminal series-parallel.
\item $(G,s,t)$ has a decomposition tree in which all leaf nodes
   are single edges.
\item In every maximal decomposition tree for $(G,s,t)$,
   all leaf nodes are single edges.
\end{enumerate}
\end{lem}

\begin{proof}
(1) $\iff$ (2) follows directly from the definition of
``2-terminal series-parallel''.
Furthermore, (3) $\implies$ (2) is trivial because every 2-terminal graph
does possess a maximal decomposition tree.
Finally, to show (1) $\implies$ (3), we observe from
Lemma~\ref{lem_serpar} that every leaf node is 2-terminal series-parallel;
so if a leaf node is not a single edge, then it must be either a
series or parallel composition of two smaller 2-terminal series-parallel
graphs, contradicting the hypothesis that the decomposition tree is maximal.
\qed
\end{proof}

Let us now note a simple but important fact that will play a key role
in the remainder of this paper:

\begin{lem}
  \label{lem_constituents}
Let $\sfG = (G,s,t)$ be a 2-terminal graph,
and consider a decomposition tree for $\sfG$ in which the root is a $p$-node.
[If $G$ is nonseparable, then {\em every} decomposition tree
 other than the trivial one has this property.]
If $G$ has maxmaxflow $\Lambda$,
then all its {\em proper} constituents $(H,a,b)$
have between-terminals flow $\lambda_H(a,b)$ at most $\Lambda-1$.
\end{lem}

\begin{proof}
Suppose that there is a proper constituent $(H,a,b)$ such that 
$\lambda_H(a,b) \geq \Lambda$.
Let $(F,c,d)$ be the first ancestor of $(H,a,b)$ that is a $p$-node
(such a node must exist since the root is a $p$-node).
Then one of the children of $(F,c,d)$ is a connected series extension
$(F_1,c,d)$ of $(H,a,b)$ [possibly $(H,a,b)$ itself],
while the other child $(F_2,c,d)$
is connected and has no edges in common with $H$.
Therefore, by concatenating a $cd$-path from $F_2$
with $ac$- and $bd$-paths from $F_1 \setminus H$
(these paths will degenerate to empty paths if $a=c$ or $b=d$, respectively),
we obtain an $ab$-path in $F$ that uses only edges not in $H$.
Therefore, in $F$ (and hence also in $G$)
there are at least $\Lambda+1$ edge-disjoint paths between $a$ and $b$,
which contradicts the hypothesis that $G$ has maxmaxflow $\Lambda$.
\qed
\end{proof}

In less formal terms, the key point of this lemma is that
if a 2-terminal series-parallel graph of maxmaxflow $\Lambda$
is constructed via a sequence of series and parallel compositions,
with the last stage being a parallel composition,
then every intermediate graph has between-terminals flow at most $\Lambda-1$
(as well as of course having maxmaxflow at most $\Lambda$).

%% {\bf This paragraph will be rewritten or deleted, as we are now
%%    going to define decomposition trees for arbitrary 2-terminal graphs,
%%    not just series-parallel ones.}
%% The 2-terminal series-parallel graphs are constructed by series and
%% parallel composition operations starting with $K_2$ as the only
%% initial graph. However we could consider a wider class of graphs
%% by using the same operations of series and parallel composition,
%% but with {\em any} fixed set ${\cal B}$ of 2-terminal graphs as the
%% initial graphs. In this case, we get an analogous decomposition tree
%% where the leaf nodes correspond to the graphs in ${\cal B}$, and the
%% obvious analogue of Lemma~\ref{lem_constituents} holds. We consider
%% such a class of graphs in Section~\ref{wheatstone}.

\subsection{Computing the multivariate Tutte polynomial using a
   decomposition tree}   \label{subsec.computing}

Let $(G,s,t)$ be a 2-terminal graph,
with $G$ connected and loopless, and fix a decomposition tree for it.
We will now describe a simple algorithm for computing
the partial multivariate Tutte polynomials
$A_{G,s,t}(q,\bv)$ and $B_{G,s,t}(q,\bv)$ 
--- and more generally
the partial multivariate Tutte polynomials
$A_{H,a,b}(q,\bv)$ and $B_{H,a,b}(q,\bv)$ 
for each node $(H,a,b)$ in the decomposition tree ---
given the partial multivariate Tutte polynomials
of all the {\em leaf}\/ nodes.
In particular, we will be able to compute the multivariate Tutte polynomial
\be
   Z_G(q,\bv)  \;=\;  q^2 A_{G,s,t}(q,\bv) \,+\, q B_{G,s,t}(q,\bv)
   \;.
 \label{eq.ZG.AB}
\ee

Before stating the algorithm, however, let us remark briefly on the
different ways that it can be interpreted.
Since $A_{G,s,t}(q,\bv)$, $B_{G,s,t}(q,\bv)$ and $Z_G(q,\bv)$
are polynomials with {\em integer}\/ coefficients
--- i.e.\ they belong to the polynomial ring $\Z[q,\bv]$ ---
they induce well-defined polynomial {\em functions}\/
on every commutative ring $R$,
i.e.\ $A_{G,s,t} \colon\, R \times R^E \to R$
and likewise for the other two.
Therefore, if $R$ is an arbitrary commutative ring
and $q$ and $\{v_e\}$ are given specified values in $R$,
then it makes sense to compute the value (which again lies in $R$)
of the polynomial functions $A_{G,s,t}(q,\bv)$, $B_{G,s,t}(q,\bv)$
and $Z_G(q,\bv)$.
This is what our algorithm will do,
using only addition and multiplication in the ring $R$;
it thus works, without any modification,
for an {\em arbitrary}\/ choice of the commutative ring $R$.
The two most interesting choices for our purposes are:
\begin{itemize}
   \item $R = \Z[q,\bv]$, with $q$ and $\{v_e\}$ taken to be indeterminates.
      This allows us to compute {\em symbolically}\/ the various
      multivariate Tutte polynomials.
   \item $R = \C$ (or $\R$ or $\Q$ or $\Z$),
       with $q$ and $\{v_e\}$ given specified numerical values.
      This allows us to compute the {\em numerical values}\/ of the various
      multivariate Tutte polynomials.
\end{itemize}

Let us now state the algorithm, which is in fact a trivial application
of Proposition~\ref{prop.serpar.AB}:

\bigskip

{\bf Algorithm 1.}
Fix a commutative ring $R$,
and fix values $q \in R$ and $\bv = \{v_e\} \in R^E$.
We assume that the values of $A_{H,a,b}(q,\bv)$ and $B_{H,a,b}(q,\bv)$ 
are known for every {\em leaf}\/ node $(H,a,b)$.
Then we proceed inductively up the tree,
computing $A_{H,a,b}(q,\bv)$ and $B_{H,a,b}(q,\bv)$ 
using Proposition~\ref{prop.serpar.AB}:
\begin{itemize}
 \item If $(H,a,b)$ is a $p$-node whose children $(H_1,s_1,t_1)$ and
   $(H_2,s_2,t_2)$ have already been computed, we set
\begin{subeqnarray}
A_{H,a,b}    & = &   A_{H_1,s_1,t_1} A_{H_2,s_2,t_2}   \\[2mm]
B_{H,a,b}    & = &   A_{H_1,s_1,t_1} B_{H_2,s_2,t_2} \,+\,
                     A_{H_2,s_2,t_2} B_{H_1,s_1,t_1} \,+\,
                     B_{H_1,s_1,t_1} B_{H_2,s_2,t_2}
\end{subeqnarray}
 \item If $(H,a,b)$ is an $s$-node whose children $(H_1,s_1,t_1)$ and
   $(H_2,s_2,t_2)$ have already been computed, we set
\begin{subeqnarray}
A_{H,a,b}  & = &   A_{H_1,s_1,t_1} B_{H_2,s_2,t_2}  \,+\,
                   A_{H_2,s_2,t_2} B_{H_1,s_1,t_1}  \,+\,
                   q A_{H_1,s_1,t_1} A_{H_2,s_2,t_2} \\[2mm]
B_{H,a,b}  & = &   B_{H_1,s_1,t_1} B_{H_2,s_2,t_2}
\end{subeqnarray}
\end{itemize}
\qed

The validity of this algorithm
is an immediate consequence of Proposition~\ref{prop.serpar.AB}.

\bigskip

If the ring $R$ is in fact a {\em field}\/ $F$,
then this algorithm can be usefully rephrased in terms of the
``effective weights''
$v_{\rm eff}(H,a,b)  = B_{H,a,b}(q,\bv) / A_{H,a,b}(q,\bv)$,
{\em provided that we are careful to avoid division by zero}\/.
The two most interesting choices of the field $F$ are:
\begin{itemize}
   \item $F = \Q(q,\bv)$, the field of rational functions
      with rational coefficients in the indeterminates $q$ and $\{v_e\}$.
      This will allow us to compute {\em symbolically}\/ the various
      multivariate Tutte polynomials.
   \item $F = \C$ (or $\R$ or $\Q$).
      This will allow us to compute the {\em numerical values}\/
      of the various multivariate Tutte polynomials,
      when $q$ and $\{v_e\}$ are given specified numerical values.
\end{itemize}
In this version, the algorithm works as follows
[for simplicity we concentrate on computing $Z_G(q,\bv)$
 and $v_{\rm eff}(H,a,b)$]:

\bigskip

{\bf Algorithm 2.}
Fix a field $F$
and fix values $q \in F$ and $\bv = \{v_e\} \in F^E$.
We assume that the values of $A_{H,a,b}(q,\bv)$ and $B_{H,a,b}(q,\bv)$
are known for every {\em leaf}\/ node $(H,a,b)$,
{\em with all the $A_{H,a,b}(q,\bv)$ nonzero}\/.
We can therefore define
$v_{\rm eff}(H,a,b)  = B_{H,a,b}(q,\bv) / A_{H,a,b}(q,\bv)$
for each leaf node.
We now proceed inductively up the tree:
\begin{itemize}
 \item If $(H,a,b)$ is a $p$-node whose children $(H_1,s_1,t_1)$ and
   $(H_2,s_2,t_2)$ have already been labelled with values
   $v_{\rm eff}(H_1,s_1,t_1)$ and $v_{\rm eff}(H_2,s_2,t_2)$,
   we then label $(H,a,b)$ with
   $$v_{\rm eff}(H,a,b) \;=\;
    v_{\rm eff}(H_1,s_1,t_1) \,\pll\, v_{\rm eff}(H_2,s_2,t_2)  \;.$$
 \item If $(H,a,b)$ is an $s$-node whose children $(H_1,s_1,t_1)$ and
   $(H_2,s_2,t_2)$ have already been labelled with values
   $v_{\rm eff}(H_1,s_1,t_1)$ and $v_{\rm eff}(H_2,s_2,t_2)$,
   we then label $(H,a,b)$ with
   $$v_{\rm eff}(H,a,b) \;=\;
    v_{\rm eff}(H_1,s_1,t_1) \,\serq\, v_{\rm eff}(H_2,s_2,t_2)$$
   {\em provided that
     $q + v_{\rm eff}(H_1,s_1,t_1) + v_{\rm eff}(H_2,s_2,t_2) \neq 0$}\/;
   otherwise we give $v_{\rm eff}(H,a,b)$ the value ``undefined''
   and terminate the algorithm.
   In the former case,
   we also mark the node $(H,a,b)$ as carrying a prefactor
   $q + v_{\rm eff}(H_1,s_1,t_1) + v_{\rm eff}(H_2,s_2,t_2)$.
\end{itemize}
If the algorithm succeeds in labeling the entire decomposition tree
(i.e.\ does not encounter any value ``undefined''),
we then set $Z_G(q,\bv)$ equal to $q[q + v_{\rm eff}(G,s,t)]$
times the product of the prefactors associated to all the $s$-nodes
times the product of the $A_{H,a,b}(q,\bv)$ from all the leaf nodes.
%%   Or explain first how to get $A_{G,s,t}$ and $B_{G,s,t}$????
\qed

Since this algorithm for computing $v_{\rm eff}(H,a,b)$ and $Z_G(q,\bv)$
is simply a rephrasing of Algorithm 1
combined with \reff{eq.ZG.AB}, its validity follows immediately.

Of course, Algorithm 2 is not really an algorithm
(i.e.\ a process that is always guaranteed to give an answer)
because it could fail by encountering an ``undefined'' value at some $s$-node.
%% ---
%% and this despite the fact that the various multivariate Tutte polynomials
%% are always well-defined.
But we can say the following:

First of all, Algorithm 2 is guaranteed to succeed when it is carried out
symbolically, i.e.\ over the field $\Q(q,\bv)$ of rational functions
in the indeterminates $q$ and $\{v_e\}$.
More precisely:

\begin{proposition}
   \label{prop.algorithm}
Let $(G,s,t)$ be 2-terminal graph,
with $G$ connected and loopless, and fix a decomposition tree for it.
Then, for each node $(H,a,b)$ in the decomposition tree,
the quantities $v_{\rm eff}(H,a,b)$, considered as elements of
the field $\Q(q,\bv)$ of rational functions
in the indeterminates $q$ and $\{v_e\}_{e \in E(H)}$,
have the following properties:
\begin{itemize}
   \item[(a)] $v_{\rm eff}(H,a,b)$ is a rational function
       of $q$ and $\{v_e\}_{e \in E(H)}$
       that depends nontrivially on $\{v_e\}_{e \in E(H)}$.
   \item[(b)] At an $s$-node $(H,a,b)$ with children $(H_1,s_1,t_1)$ and
      $(H_2,s_2,t_2)$, one can never have
      $q + v_{\rm eff}(H_1,s_1,t_1) + v_{\rm eff}(H_2,s_2,t_2) = 0$
      in $\Q(q,\bv)$.
\end{itemize}
Therefore, Algorithm 2 never encounters an ``undefined'' value
when it is carried out over the field $\Q(q,\bv)$.
\end{proposition}

\begin{proof}
Statement (a) is simply Lemma~\ref{lemma_veff_nontrivial}(a)
since by hypothesis $H$ is connected.
Statement (b) holds because, by (a),
$v_{\rm eff}(H_1,s_1,t_1)$ and $v_{\rm eff}(H_2,s_2,t_2)$
depend nontrivially on disjoint sets of indeterminates.
%% The proof is by induction upwards from the leaves.
%% Assertion (a) holds at each leaf node because ???????????????j??.
%% Now consider any non-leaf node $(H,a,b)$
%% with children $(H_1,s_1,t_1)$ and $(H_2,s_2,t_2)$.
%% One can never have
%% $q + v_{\rm eff}(H_1,s_1,t_1) + v_{\rm eff}(H_2,s_2,t_2) = 0$
%% because, by the inductive hypothesis (a),
%% $v_{\rm eff}(H_1,s_1,t_1)$ and $v_{\rm eff}(H_2,s_2,t_2)$
%% depend nontrivially on disjoint sets of indeterminates.
%% And for the same reason,
%% the effective weight $v_{\rm eff}(H,a,b)$ will depend
%% nontrivially on $\{v_e\}_{e \in E(H)}$.
%% {\bf Is this proof correct and complete????}
\qed
\end{proof}

\medskip

On the other hand,
Algorithm 2 {\em can fail}\/ when it is carried out over $\C$
(i.e.\ with numerical values of $q$ and $\{ v_e \}$).
Suppose, for instance, that $(G,s,t)$ [or some constituent thereof]
consists of an edge $e$ in series with
the parallel combination of edges $f$ and $g$.
Then the multivariate Tutte polynomial for this graph is unambiguously
\be
   Z_G(q,\bv)  \;=\;  q(q+v_e)(q+v_f+v_g + v_f v_g)  \;.
\ee
But if we choose $v_e = -q$ (where $q$ is any complex number),
$v_f = -1/2$ and $v_g = 1$,
then Algorithm 2 first computes $v_f \pll v_g = 0$
and then tries to compute $v_e \serq (v_f \pll v_g) = -q \serq 0$,
yielding an ``undefined'' result of $0/0$.
%% {\bf Put instead (or in addition) Gordon's examples of the 3-cycle
%%    and 4-cycle???  But we need to distinguish the two different
%%    interpretations of when we declare something ``undefined''.}

It is nevertheless worth stressing once again that whenever Algorithm 2,
carried out over $\C$ (or any other field), does give an answer,
that answer is guaranteed to be correct.

In the remainder of this paper, when we use Algorithm 2
over $\C$, we will do so in the context of additional hypotheses
that guarantee that no intermediate answer is ever ``undefined''.

\bigskip
\bigskip

{\bf Some remarks concerning computational complexity.}
1.  There exists a linear-time algorithm for taking a 2-terminal graph
$(G,s,t)$ and finding a maximal decomposition tree for it
(see \cite{Valdes:1982fk}).
%%    Valdes, Tarjan and Lawler \cite{Valdes_82} in linear time
%%    test whether a graph is series-parallel, but it appears
%%    that the type of ``decomposition tree'' they find for it
%%    is the opposite of ours:  it decomposes the graph by repeated
%%    series and parallel reductions of edges, not by repeated
%%    series and parallel compositions of 2-terminal subgraphs.
%\comment{\bf Is this true???  Can we find a published reference??? Yes it is true and follows from Valdes.}
Then Lemma~\ref{lem_SP_decomposition_tree} tells us
in particular that $(G,s,t)$ is 2-terminal series-parallel
if and only if all the leaves of this maximal decomposition tree are $K_2$'s.

2.  Given a maximal decomposition tree for
a 2-terminal series-parallel graph $(G,s,t)$,
Algorithm 2 provides a linear-time algorithm
for computing $Z_G(q,\bv)$ as well as $v_{\rm eff}(H,a,b)$
for every constituent $(H,a,b)$,
provided that we work in a computational model where
each field operation (in $\C$ or in $\Q(q,\bv)$ as the case may be)
is assumed to take a time of order 1,
and provided we take into account the possibility of failure
when we work over $\C$.

\bigskip

\subsection{Leaf-joined trees}   \label{subsec.leaf-joined}

Given a positive integer $r \geq 2$,
we can form a graph $G_n^r$ by taking a complete $r$-ary rooted tree of 
height $n \ge 1$ and then identifying all the leaves into a single vertex.
As an example, Figure~\ref{leafjoinedtree}
shows the graphs $G_3^2$ and $G_3^3$.

\begin{figure}[t]
\begin{center}
\includegraphics{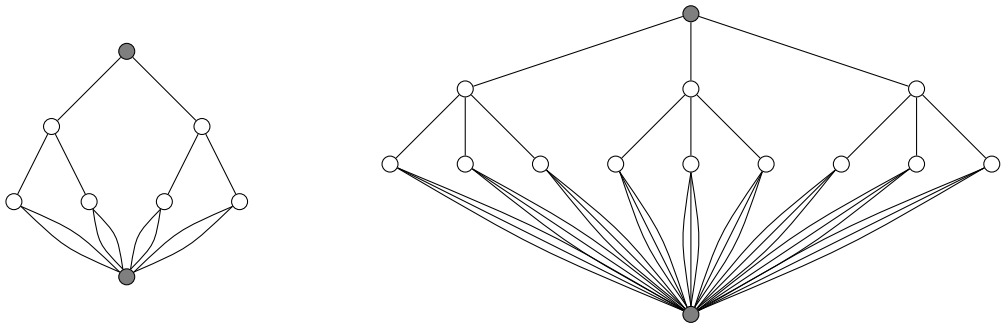}
\end{center}
\caption{$G_3^2$ is a complete binary tree of height three with all leaves identified, and 
$G^3_3$ is a complete ternary tree of height three with all leaves identified.}
\label{leafjoinedtree}
\end{figure}

We consider $G_n^r$ as a 2-terminal graph in which
the terminals are the root and the identified-leaves vertex.
It is easy to see that $G_n^r$ is in fact 2-terminal series-parallel,
as it can be defined recursively as follows:
\begin{subeqnarray}
%% G_0^r  & = &  K_1  \\[1mm]
G_1^r  & = &  K_2^{(r)} \\[1mm]
G_{n+1}^r & = & (K_2 \ser G_n^r)^{\pll r}
\end{subeqnarray}
where $K_2^{(r)}$ is the graph with two vertices
connected by $r$ parallel edges,
and $G^{\pll r}$ denotes the parallel composition of $r$ copies of $G$.
It then follows that $G_n^r$ has $(r^n+r-2)/(r-1)$ vertices,
that the flow between its terminals is $r$ (for $n \ge 1$),
and that its maxmaxflow is $r+1$ (for $n \ge 2$). 

In Appendix~\ref{app_tree} we shall prove the following:

\begin{thm}
  \label{thm.leaf-joined}
For fixed $r \ge 2$, every point of the circle $|q-1| = r$
is a limit point of chromatic roots
for the family $\{G_n^r\}_{n \ge 1}$
of leaf-joined trees of branching factor $r$.
[More precisely, for every $q_0$ satisfying $|q_0-1| = r$
 and every $\epsilon > 0$, there exists $n_0 = n_0(q_0,\epsilon)$
 such that for all $n \ge n_0$ the graph $G_n^r$ has a chromatic root $q$
 lying in the disc $|q-q_0| < \epsilon$.]
\end{thm}

%% {\bf State and prove in Appendix!!!!!!!
%%    We will have $v_{\rm eff}(G_{n+1}) = R_q(v_{\rm eff}(G_n))$
%%      where $R_q(v) = (v \serq v_0)^{\pll r}$
%%      and $v_0$ is the weight I put on each edge (usually $v_0 =-1$).
%%    Remember that I need to show that the initial condition is attracted
%%    to the attractive fixed point!  Will I use results on critical points
%%    for this?}

\section{An abstract theorem on excluding roots}  \label{sec.abstract}

The multivariate Tutte polynomial of the graph $G=K_2$
having a single non-loop edge of weight $v_e$
is $Z_{K_2}(q, \bv) = q(q+v_e)$,
which has roots at $q=0$ and $q=-v_e$.
Given a 2-terminal series-parallel graph $G$
with arbitrary complex edge weights $\{v_e\}$ and a fixed complex number $q$,
we can apply series and parallel reductions
as in Section~\ref{subsec.computing} until $G$
has been reduced to a single edge with some ``effective weight''
$v_{\rm eff} \in \C \cup \{\hbox{undefined}\}$.
If $v_{\rm eff} \neq \hbox{undefined}$,
then we can be sure that none of the prefactors of the form $q+v_{e_1}+v_{e_2}$
generated during the series reductions were 0,
and we can therefore conclude that
$Z_G(q,\bv) = 0$ if and only if $q=0$ or $v_{\rm eff} = -q$.
%% (We remark that if $v_{\rm eff} = \infty$,
%%  then it does not necessarily follow that $Z_G(q,\bv) = 0$.)

This observation then gives us a strategy for determining root-free regions
for the multivariate Tutte polynomials of families of series-parallel graphs.
For a fixed $q \neq 0$ in the conjectured root-free region,
we bound the regions of the (finite) complex $v$-plane
where $v_{\rm eff}$ can lie for any graph in the family,
and we show that these regions do not contain the point $v_{\rm eff} = -q$
that would correspond to a zero of $Z_G(q,\bv)$.
If we can do this, then we have shown that $Z_G(q,\bv) \neq 0$.
The precise result is as follows:

\begin{thm}
  \label{thm_abstract}
Let $q \neq 0$ be a fixed complex number
and let $\Lambda \ge 2$ be a fixed integer.
Let $S_1 \subseteq S_2 \subseteq \cdots \subseteq S_{\Lambda-1}$
be sets in the (finite) complex $v$-plane such that
\begin{itemize}
\item[(1)] $S_k \tvserq{V} S_\ell \subseteq S_{\min(k,\ell)}$
   for all $k,\ell$
\item[(2)] $S_k \tvpll{V} S_\ell\subseteq S_{k+\ell}$
   for $k + \ell \leq \Lambda-1$
\end{itemize}
Now consider any 2-terminal series-parallel graph $(G,s,t)$
and any maximal decomposition tree for $(G,s,t)$
in which all the proper constituents
have between-terminals flow at most $\Lambda-1$,
and equip $G$ with edge weights $v_e \in S_1$.
Then, for every node $(H,a,b)$ of the decomposition tree
that has between-terminals flow $\lambda_H(a,b) \le \Lambda-1$,
we have $v_{\rm eff}(H,a,b) \in S_{\lambda_H(a,b)}$.

Now assume further that, in addition to (1) and (2),
the following hypotheses hold:
\begin{itemize}
\item[(3)] $-q \notin S_{\Lambda-1}$
\item[(4)] $-q \notin S_k \tvpll{V} S_{\ell}$ for $k+\ell = \Lambda$
\end{itemize}
Then, for any loopless series-parallel graph $G$
with maxmaxflow at most $\Lambda$,
we have $Z_G(q,\bv) \not= 0$ whenever $v_e \in S_1$ for all edges.
(In particular, if $-1 \in S_1$, then $q$ is not a chromatic root of $G$.)
\end{thm}

{\bf Remarks.}
1.  It is implicit in condition (1) that
the operation in question is always well-defined
(i.e.\ does not take the value ``undefined''),
or in other words that $q+v_1+v_2 \neq 0$
whenever $v_1 \in S_k$ and $v_2 \in S_\ell$.

2.  When the root of the decomposition tree is a $p$-node
(which occurs in particular whenever $G$ is nonseparable and not $K_2$)
and $G$ has maxmaxflow $\Lambda$,
then Lemma~\ref{lem_constituents} guarantees that every {\em proper}\/
constituent has between-terminals flow at most $\Lambda-1$.
The root node $(G,s,t)$, by contrast, {\em might}\/ have
between-terminals flow as large as $\Lambda$.
\qed

\medskip

\begin{proof}
Let $(G,s,t)$, its maximal decomposition tree
and its edge weights be as specified.
% Let $(G,s,t)$ be a 2-terminal series-parallel graph
% with between-terminals flow $\lambda_G(s,t) = \lambda$,
% all of whose constituents
% [with respect to some maximal decomposition tree for $(G,s,t)$
%  that we fix henceforth]
% have between-terminals flow at most $\Lambda-1$,
% equipped with edge weights $v_e \in S_1$.
We want to prove that $v_{\rm eff}(H,a,b) \in S_{\lambda_H(a,b)}$
for all nodes $(H,a,b)$ that satisfy $\lambda_H(a,b) \le \Lambda-1$.
We shall prove this claim by induction upwards from the leaves
of the decomposition tree.
By Lemma~\ref{lem_SP_decomposition_tree},
a leaf of the decomposition tree is an edge $e$ of $G$,
hence has between-terminals flow equal to 1,
and $v_{\rm eff} = v_e \in S_1$ by hypothesis.
So let $(H,a,b)$ be a non-leaf node of the decomposition tree
and suppose that the children of $(H,a,b)$,
call them $(H_1,a_1,b_1)$ and $(H_2,a_2,b_2)$,
have between-terminals flow $k$ and $\ell$, respectively.
Since $(H_1,a_1,b_1)$ and $(H_2,a_2,b_2)$ are proper constituents,
we have by hypothesis $k,\ell \le \Lambda-1$;
so by the inductive hypothesis, we have
$v_{\rm eff}(H_1,a_1,b_1) \in S_k$ and $v_{\rm eff}(H_2,a_2,b_2) \in S_\ell$.
Using Lemma~\ref{lem_terminals},
it is clear that conditions (1) and (2) ensure that
$v_{\rm eff}(H,a,b) \in S_{\lambda_H(a,b)}$ holds
%% when $(H,a,b)$ is an $s$-node or a $p$-node, respectively.
whenever $\lambda_H(a,b) \le \Lambda-1$
(which holds for all proper constituents
 and might or might not hold for the root node).
This proves the first half of the theorem.

As $Z_G(q, \bv)$ is multiplicative over blocks,
and the maxmaxflow of a separable graph is the maximum of the maxmaxflows
of its blocks,
it suffices to prove the second half of the theorem when $G$ is a
loopless {\em nonseparable}\/ series-parallel graph
of maxmaxflow at most $\Lambda$.
Since the result holds trivially when $G = K_1$,
we can assume that $G$ has at least one edge.
Therefore, by Theorem~\ref{thm_serpar_1},
$G$ has a pair of vertices $s$, $t$ such that
$(G,s,t)$ is a 2-terminal series-parallel graph
and hence described by a maximal decomposition tree
whose leaf nodes are single edges.
Furthermore, by Lemma~\ref{lem_constituents},
all of the proper constituents of $(G,s,t)$
have between-terminals flow at most $\Lambda-1$.
Therefore, if $(H,a,b)$ is a proper constituent of $(G,s,t)$
with between-terminals flow $\lambda_H(a,b) = \lambda$,
we can apply the first half of the theorem
to conclude that $v_{\rm eff}(H,a,b) \in S_\lambda$.

By condition (3) [and the nesting $S_i \subseteq S_{\Lambda-1}$],
we have $v_{\rm eff}(H,a,b) \neq -q$
whenever $(H,a,b)$ is a proper constituent of $(G,s,t)$.
On the other hand, the final step (at the root of the decomposition tree)
constructs $(G,s,t)$ as the parallel composition of two proper constituents
whose between-terminal flows sum to $\lambda_G(s,t) \le \Lambda$,
so conditions (4) and (2)/(3) together ensure that
$v_{\rm eff}(G,s,t) \neq -q$.
Therefore, by Algorithm 2 of Section~\ref{subsec.computing},
$Z_G(q,\bv)$ is equal to a nonzero prefactor
--- namely, the product over $s$-nodes of
 $q + v_{\rm eff}(G_1,s_1,t_1) + v_{\rm eff}(G_2,s_2,t_2)$,
 a quantity that is nonvanishing by virtue of Remark~1
 preceding this proof ---
multiplied by $q[q+v_{\rm eff}(G,s,t)]$,
and is therefore nonzero as claimed.
\qed
\end{proof}

Of course, to {\em apply}\/ this theorem it is necessary to actually
{\em identify}\/ suitable sets
$S_1 \subseteq S_2 \subseteq \ldots \subseteq S_{\Lambda-1}$.
In practice one usually starts from a specified set $\scrv \subseteq \C$
of ``allowed edge weights'' ---
for instance, $\scrv = \{-1\}$ for the chromatic polynomial ---
and one attempts to find sets
$S_1 \subseteq S_2 \subseteq \ldots \subseteq S_{\Lambda-1}$
satisfying $S_1 \supseteq \scrv$
along with the hypotheses (1)--(4) of Theorem~\ref{thm_abstract}.
For any particular combination of $q$, $\Lambda$ and $\scrv$,
there is always a collection of {\em minimal regions}\/
$S_1 \subseteq S_2 \subseteq \ldots \subseteq S_{\Lambda-1}$
where $S_1 \supseteq \scrv$
and conditions (1) and (2) are satisfied.
If one knows this collection of minimal regions,
then conditions (3) and (4) become a ``final check''
certifying that $q$ is not a root.

In practice, though, it is almost always impossible
to describe the minimal regions even for specific values of $q$ and $\Lambda$,
let alone symbolically
(but see Section~\ref{sec_discs} for some computer-generated approximations).
Therefore it is necessary to bound the optimal regions
inside larger regions with shapes that are more amenable to analysis.
But it is also important to fit the bounding regions as tightly as possible
to the optimal regions, as conditions (1) and (2)
cause any ``unnecessary points'' included in the approximation to a region
to have a cascading effect on the approximations for the other regions,
thereby incorporating still more possibly unnecessary points, and so on.

There are, in fact, two slightly different reasons why
including unnecessary points in the regions $S_i$ can lead to poor bounds.
Firstly, if we have chosen $S_2,S_3,\ldots$ to be much larger
than they need to be, {\em for the given set $S_1$}\/,
then the bounds one obtains from Theorem~\ref{thm_abstract}
may (not surprisingly) be much weaker than the truth.
Secondly, it is important to observe that even if we are ultimately
interested in proving $Z_G(q,\bv) \neq 0$
for weights $v_e$ lying in a specified set $\scrv$,
we will get from Theorem~\ref{thm_abstract},
whether we like it or not, the same result for all $v_e \in S_1$.
Of course, if $S_1$ is exactly the minimal region containing the
given $\scrv$ and satisfying conditions (1) and (2),
then nothing is lost, as any bound valid for all series-parallel graphs
of maxmaxflow $\Lambda$ with weights in $\scrv$ will also be valid for
weights in $S_1$ (since any $v$ lying in the minimal region $S_1$
is in fact the $v_{\rm eff}$ for a suitable 2-terminal series-parallel graph
of maxmaxflow $\Lambda$ and between-terminals flow 1,
with edge weights in $\scrv$).
But if the chosen $S_1$ is significantly larger than the minimal region,
then even the best-possible bound for weights in $S_1$ may be much weaker
than the corresponding bound for weights in $\scrv$.
In particular, if $S_1$ extends much outside the
``complex antiferromagnetic regime'' $|1 + v_e | \le 1$
--- where ``much outside'' means, roughly,
 more than a distance of order $1/|q|$ ---
then one {\em expects}\/ the $q$-plane roots of $Z_G(q,\bv)$
to grow exponentially in $\Lambda$ rather than linearly
(see \cite{Jackson-Procacci-Sokal} for further discussion,
 and see also footnote~\ref{footnote_expgrowth} below).
%% {\bf What is the maximal $S_1$ contained in the complex antiferromagnetic
%%    regime???}

The simplest types of region to manipulate analytically are discs,
especially discs centered at the origin, and so it is natural to try to bound
the optimal regions inside suitable discs.
If one insists on using discs centered at the origin,
then it furthermore matters whether one uses the $v$-variables,
the $y$-variables or the $t$-variables.
If one makes a poor choice ---
e.g.\ the optimal regions are either far from being discs,
 or far from being centered at the origin in the chosen variables ---
then one will obtain poor bounds,
e.g.\ bounds that grow exponentially rather than linearly in $\Lambda$.

% Although this can be done, this strategy yields rather poor bounds
% --- in fact, bounds that are not linear in $\Lambda$.
% {\bf We should show this somewhere --- maybe in an appendix.}
% The reason for this appears to be that the optimal regions
% are actually not at all disc-like 
% at least, our computer-generated approximations to the optimal regions
% are not at all disc-like) and so the bounding discs include
% far too many unnecessary points.
% {\bf This needs clarification:  Is the problem that the regions are
%    approximately discs centered far from the origin,
%    or that they are not approximately discs at all?????
%    Indeed, we seem to be mixing up here \emph{two distinct problems}:
%    one is that in the $v$-plane the discs are centered far from the origin,
%    which explains why we prefer to work with the $t$ variables;
%    and the other is that, even in the $t$ variables,
%    we mustn't use just discs, but rather ``disc + point''
%    (or ``disc + stalk'').}

It turns out that the optimal regions are not too far from being discs
centered at the origin {\em if we use the $t$-variables}\/,
but are quite far from being discs centered at the origin
if we use the $v$- or $y$-variables.
We shall therefore use the $t$-variables in the remainder of this paper.
%% However it turns out that far better results can be obtained after
%% a variable substitution from the edge-weights $\{v_e\}$
%% to the transmissivities $\{t_e\}$, because the optimal regions
%% {\em in the complex $t$-plane} appear to be qualitatively much more
%% similar to discs centred at the origin. 
Let us recall that the important points $v=-1$, $v = \infty$ and $v= -q$
correspond to $t=1/(1-q)$, $t=1$ and $t =\infty$, respectively.
We can therefore re-express Theorem~\ref{thm_abstract}
in the language of transmissivities $\{t_e\}$.
For simplicity we suppress the statements about $v_{\rm eff}$
(or $t_{\rm eff}$) and concentrate on the conclusion that
$Z_G(q, \bv) \neq 0$.
%% we can conclude that $Z_G(q,\bv) \neq 0$
%% if the effective transmissivity is finite
%% and none of the intermediate transmissivities are equal to 1.
% The next result expresses Theorem~\ref{thm_abstract}
% in the language of transmissivities;
% for future convenience we make explicit the set $\scrv$
% of ``allowed edge weights''.
%% specialized to chromatic polynomials
%% (i.e the case where all $v_e=-1$ or equivalently all $t_e=1/(1-q)$),
%% as this is the precise form in which we use it in Section~\ref{sec_discs}.

%\bigskip
%{\bf [[ HOW DO PREFACTORS WORK IN t-VARIABLES? We have 
%$$
%q + v_e + v_f \equiv \frac{(1-t_et_f)} {(1-t_e)(1-t_f)}
%$$
%ANSWER FROM ALAN:
%I think the prefactors are already handled in the proof of
%Theorem~\ref{thm_abstract}, and here we are merely translating
%the \emph{hypotheses} from the $v$-plane to the $t$-plane.
%]]
%}

\begin{thm}
  \label{regions}
Let $q \not= 0$ be a fixed complex number
and let $\Lambda \ge 2$ be a fixed integer.
%% and let $\scrv \subseteq \C$.
Let $S_1 \subseteq S_2 \subseteq \cdots \subseteq S_{\Lambda-1}$
be sets in the (finite) complex $t$-plane such that
\begin{itemize}
 %% \item[(0)] $\{ v/(q+v) \colon\; v \in \scrv \} \subseteq S_1$
\item[(1)] $S_k \tvser{T} S_\ell \subseteq S_{\min(k,\ell)}$
   for all $k,\ell$
\item[(2)] $S_k \tvpllq{T} S_\ell\subseteq S_{k+\ell}$
   for $k + \ell \leq \Lambda-1$
\item[(3${}'$)] $1 \notin S_{\Lambda-1}$
\item[(4)] $S_k \tvpllq{T} S_{\ell} \subseteq \C$ for $k+\ell = \Lambda$
   (i.e.\ does not ever take the value ``undefined'')
\end{itemize}
Then, for any series-parallel graph $G$ with maxmaxflow at most $\Lambda$,
we have $Z_G(q,\bv) \not= 0$ whenever $v_e/(q+v_e) \in S_1$ for all edges.
\end{thm}

\noindent
In particular, to handle chromatic polynomials it suffices to arrange that
$1/(1-q) \in S_1$.

\bigskip

{\bf Remark.}
Condition (3) states merely that the set $S_{\Lambda-1}$
avoids the point $t=1$, but in practice we will always have
$S_{\Lambda-1} \subseteq \{|t| < 1\}$.
Indeed, if $S_{\Lambda-1}$ contains any point with $|t|=1$ (resp.\ $|t|>1$),
then by condition (1) its closure $\overline{S}_{\Lambda-1}$
must contain the point $t=1$ (resp.\ $t=\infty$);
and while this is not explicitly forbidden,
it is hard to see how one could satisfy all the hypotheses
(1)--(4) in such a case.

\medskip

\proofof{Theorem~\ref{regions}}
This is almost a direct translation of Theorem~\ref{thm_abstract}
into transmissivities.
Indeed, conditions (1) and (2) here are direct translations
of conditions (1) and (2) of Theorem~\ref{thm_abstract}.
Condition (3${}'$) here is equivalent to the hypothesis
in Theorem~\ref{thm_abstract} that the sets lie
in the {\em finite}\/ $v$-plane,
while condition (3) of Theorem~\ref{thm_abstract}
is equivalent to the hypothesis here that the sets lie
in the {\em finite}\/ $t$-plane.
Finally, condition (4) here is a direct translation
of condition (4) of Theorem~\ref{thm_abstract}.
\qed

Since the regions $S_i$ are assumed increasing,
the condition (1) is most stringent for $\ell = \Lambda-1$,
and it reduces to
\begin{itemize}
   \item[(1${}'$)] \quad
   $S_k \tvser{T} S_{\Lambda-1} \subseteq S_k$
                     \quad\hbox{for all $k$.}
\end{itemize}
Furthermore, there is a simple but very useful {\em sufficient}\/
condition for condition (1)/(1${}'$) to hold:

\begin{lemma}
   \label{lemma4.3}
If there exists $r > 0$ such that
\be
   D(r^2) \subseteq S_1 \subseteq S_2 \subseteq \cdots \subseteq
               S_{\Lambda-1} \subseteq D(r)
\ee
where $D(r) = \{t \in \C \colon\, |t| \le r \}$,
then condition (1) of Theorem~\ref{regions} holds.
\end{lemma}

\proof
$S_k \tvser{T} S_\ell  \subseteq  D(r)\tvser{T} D(r)  =  D(r^2)
 \subseteq  S_{\min(k,\ell)}$.
\qed

\section{Discs in the $t$-plane}  \label{sec_discs}

In this section we shall prove the following slight strengthening
of Theorem~\ref{thm_main}:

\begin{thm}
   \label{thm_main_sec5}
Fix an integer $\Lambda \ge 2$,
and let $G$ be a loopless series-parallel graph of maxmaxflow at most $\Lambda$.
Let $\rho^\star_\Lambda$ be the unique solution of
\be
   (1+\rho)^\Lambda  \;=\;  2 (1+\rho^2)^{\Lambda-1}
 \label{eq.rhostarlambda}
\ee
in the interval $(0,1)$ when $\Lambda \ge 3$,
and let $\rho^\star_2 = 1$.
Then the multivariate Tutte polynomial $Z_G(q, \bv)$
is nonvanishing whenever $|q-1| \ge 1/\rho^\star_\Lambda$
(with $\ge$ replaced by $>$ when $\Lambda=2$)
and the edge weights $\bv = \{v_e\}_{e \in E}$
satisfy
\be
   v_e \,=\, -1 \quad\hbox{or}\quad
   \left| {v_e \over q+v_e} \right| \:\le\: \rho \, {X - 1 \over 1 - \rho X}
 \label{eq.thm_main_sec5}
\ee
(again with strict inequality when $\Lambda=2$),
where
\be
   \rho \:=\: {1 \over |q-1|}
   \quad\hbox{and}\quad
   X \,=\, \biggl( {2 \over 1+\rho} \biggr)^{1/(\Lambda-1)}
   \;.
 \label{def.rho_and_X}
\ee

Furthermore we have
$\rho^\star_\Lambda > (\log 2)/(\Lambda-\smfrac{3}{2}\log 2)$,
so that in particular
all the roots (real or complex) of the chromatic polynomial $P_G(q)$
lie in the disc $|q-1| < (\Lambda-\smfrac{3}{2}\log 2) / \log 2$.
\end{thm}

{\bf Remark.}
We shall see in Lemma~\ref{lemma_sharp_rkbound_NEW}
that under the hypothesis $|q-1| \ge 1/\rho^\star_\Lambda$
(i.e.\ $\rho \le \rho^\star_\Lambda$)
we have
\be
   \rho \, {X - 1 \over 1 - \rho X}
   \;\ge\;
   \rho^2
   \;,
\ee
so that the conclusion of Theorem~\ref{thm_main_sec5} holds
under the more stringent but simpler condition
\be
   v_e \,=\, -1 \quad\hbox{or}\quad
   \left| {v_e \over q+v_e} \right| \:\le\: {1 \over |q-1|^2}
   \;.
\ee
\qed

\bigskip

We shall prove Theorem~\ref{thm_main_sec5} by exhibiting regions
$S_1 \subseteq S_2 \subseteq \cdots \subseteq S_{\Lambda-1}$
of the complex $t$-plane
that satisfy the conditions of Theorem~\ref{regions}
when $|q-1| \ge 1/\rho^\star_\Lambda$
and for which the set~$S_1$ corresponds precisely to \reff{eq.thm_main_sec5}.
%% As these regions $S_i$ are understood to lie in the complex $t$-plane,
Since in this section we shall always be working in the $t$-plane,
we shall henceforth drop the superscripts ${}^T$ from the operators
$\tvpllq{T}$ and $\tvser{T}$. 
Let us also recall that, in the $t$-plane,
series connection $\ser$ is simply multiplication.

Before beginning this proof, it is instructive to engage in some
informal motivation of our constructions.

% In principle, for any given $q$ and $\Lambda$
% there is an {\em optimal} set of regions in that
% there is a unique set of minimal regions satisfying
% the first three conditions of Theorem~\ref{regions}.
% If these minimal regions also satisfy the final two conditions
% of Theorem~\ref{regions} then they are a witness to the fact that
% $q$ is not a chromatic root for
% any series-parallel graph of maxmaxflow $\Lambda$. 

If we want to handle the chromatic polynomial using Theorem~\ref{regions},
then we must certainly have $1/(1-q) \in S_1$.
The set of {\em minimal}\/ regions $S_i$ that contain the point $1/(1-q)$
and satisfy the first two conditions of Theorem~\ref{regions}
can be approximated by computer, because these conditions can be viewed
as rules for constructing each $S_i$ from certain others.
By imposing a fine grid on the disc $|t|<1$
and ``rounding'' each complex number to the closest grid point,
we can restrict our attention to a finite number of points.
We start by marking $t_0 = 1/(1-q)$ as belonging to $S_1$
(and hence to each $S_i$);
we then iteratively construct approximations to the regions
$S_1, S_2, \ldots, S_{\Lambda-1}$
by using conditions (1) and (2) of Theorem~\ref{regions}
until the approximations are closed under further application
of the rules.\footnote{
   For instance, for $\Lambda=3$ the rules are simply
   $S_1 \subseteq S_2$, $S_1 \pll S_1 \subseteq S_2$,
   $S_1 S_2 \subseteq S_1$ and $S_2 S_2 \subseteq S_2$.
}
If the resulting region $S_{\Lambda-1}$ is contained in
the open unit disc $\{|t| < 1\}$,
then Theorem~\ref{regions} implies that $q$ is not a chromatic root
for any graph of maxmaxflow $\Lambda$.

Repeating these experiments for a range of different values of $q$
and moderate values of $\Lambda$ suggests that
although the minimal regions are generally complicated shapes,
they are often loosely ``disc-like''
and can be bounded reasonably well by a disc in the $t$-plane
centered at the origin.
Some examples with $\Lambda=3$ are shown in Figure~\ref{fig_S1S2},
and a more extensive set of plots is included with the preprint version
of this paper at \url{arXiv.org}.\footnote{
   See the ancillary files
   \url{S1S2_2.2.pdf}, \url{S1S2_2.4.pdf} and \url{S1S2_3.0.pdf}.
   Each of these files shows $S_1$ and $S_2$ in the complex $t$-plane
   for $\Lambda=3$ and a set of values of $q$ defined by $q-1 = R e^{i\theta}$,
   where $R$ takes the specified value (2.2, 2.4 or 3.0)
   and $\theta = k\pi/180$ for $k=0,5,10,\ldots,180$.
   These plots use the conventions explained in
   the caption of Figure~\ref{fig_S1S2}.
}

\begin{figure}[p]
\vspace*{-1.4cm}
\centering
\begin{tabular}{cc}
\includegraphics[width=165pt]{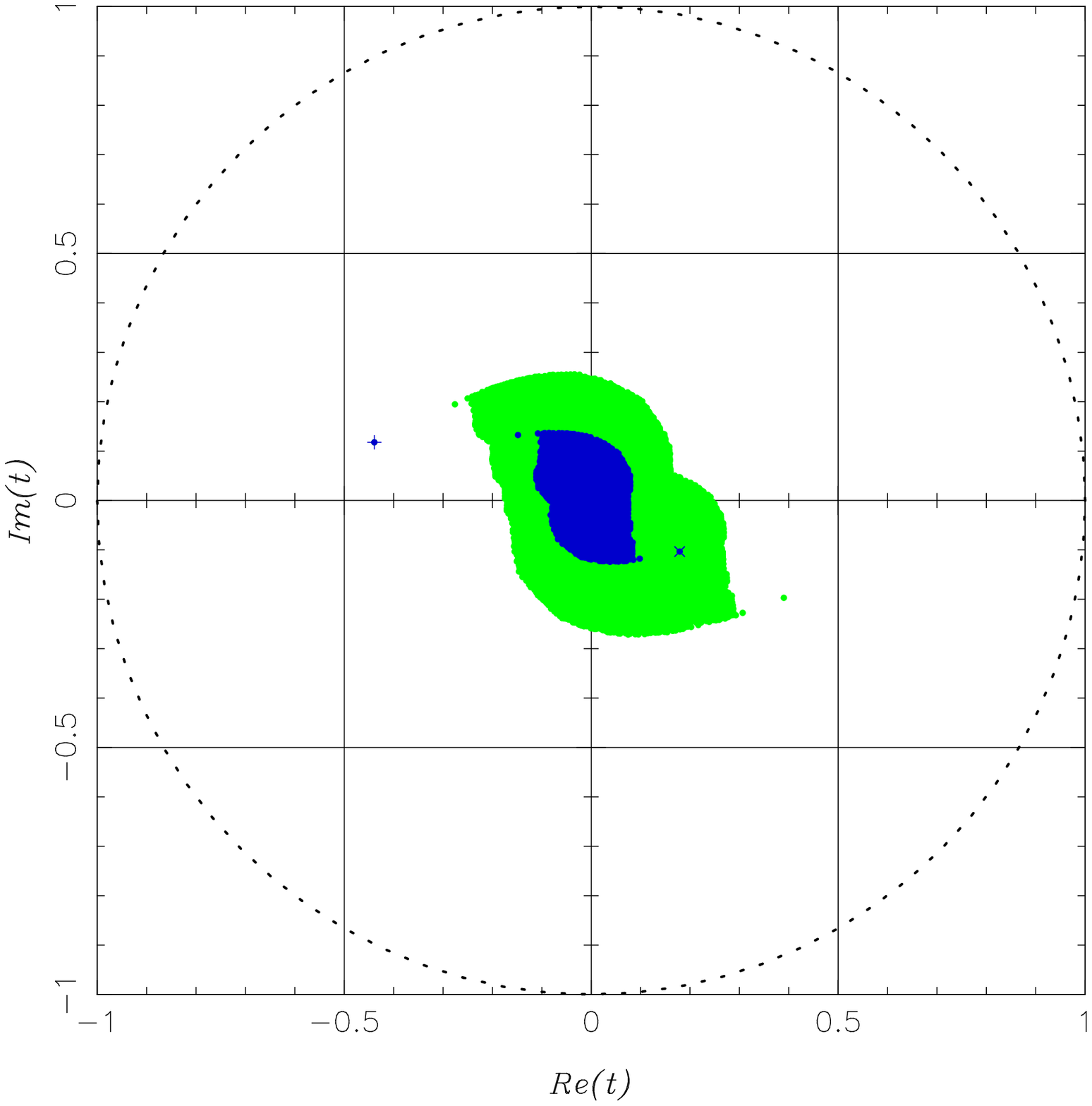} &  \qquad\qquad
\includegraphics[width=165pt]{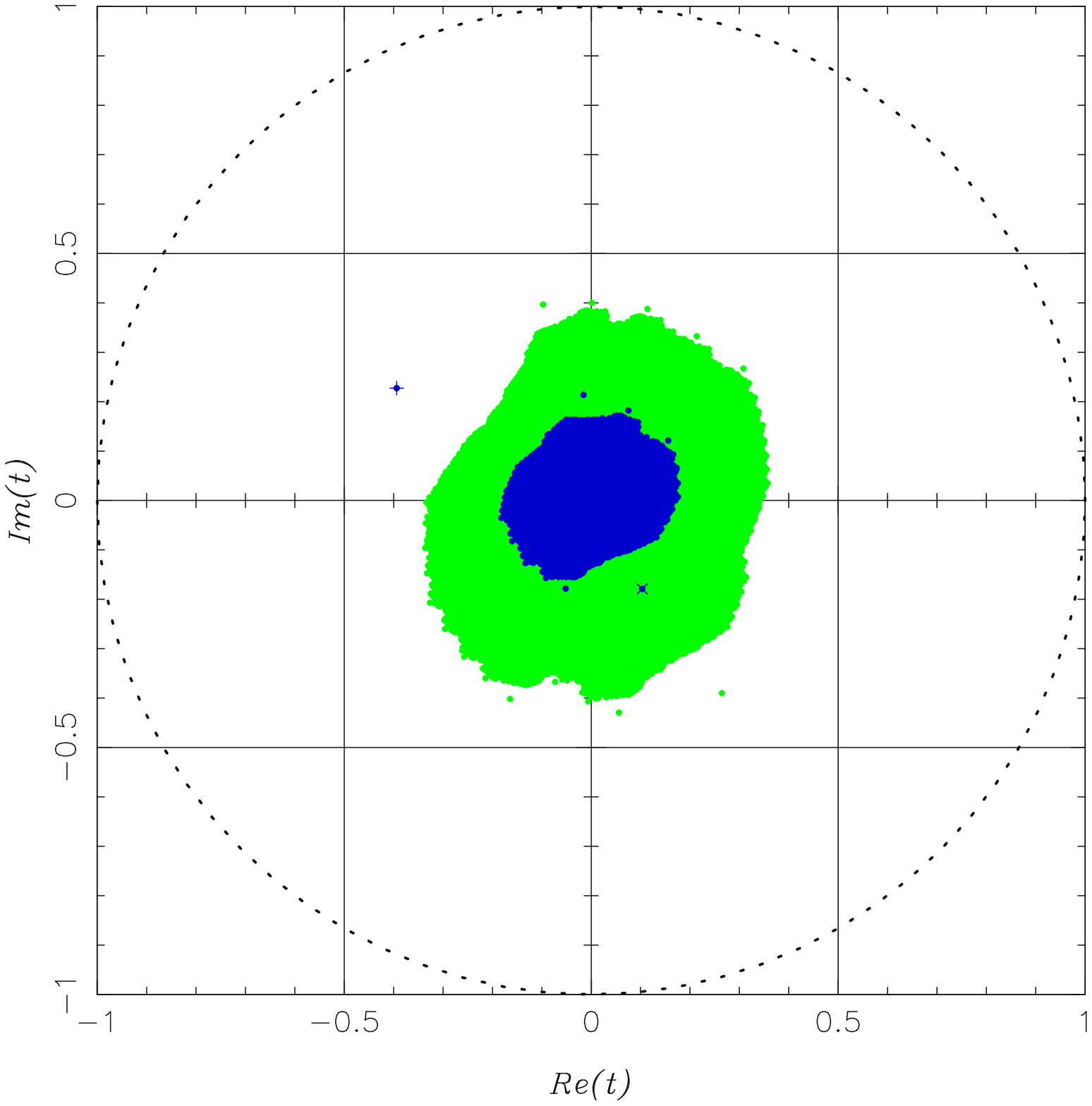} \\[1mm]
\footnotesize (a) $q = 3.125 + 0.569 i \approx 1 + 2.2 e^{\pi i/12}$  &
   \qquad\qquad
\footnotesize (b) $q = 2.905 + 1.100 i \approx 1 + 2.2 e^{\pi i/6}$   \\[8mm]
\includegraphics[width=165pt]{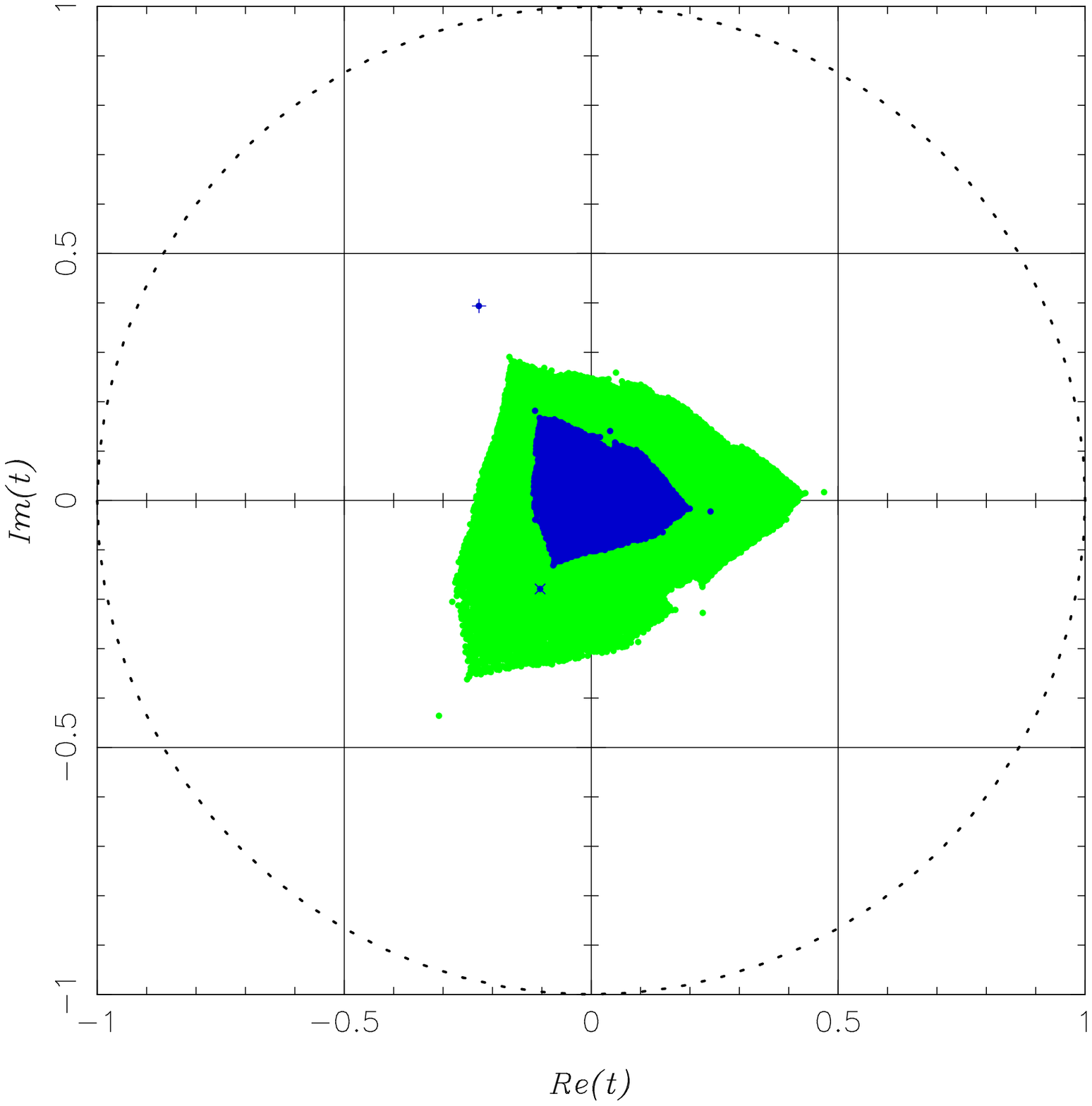} &  \qquad\qquad
\includegraphics[width=165pt]{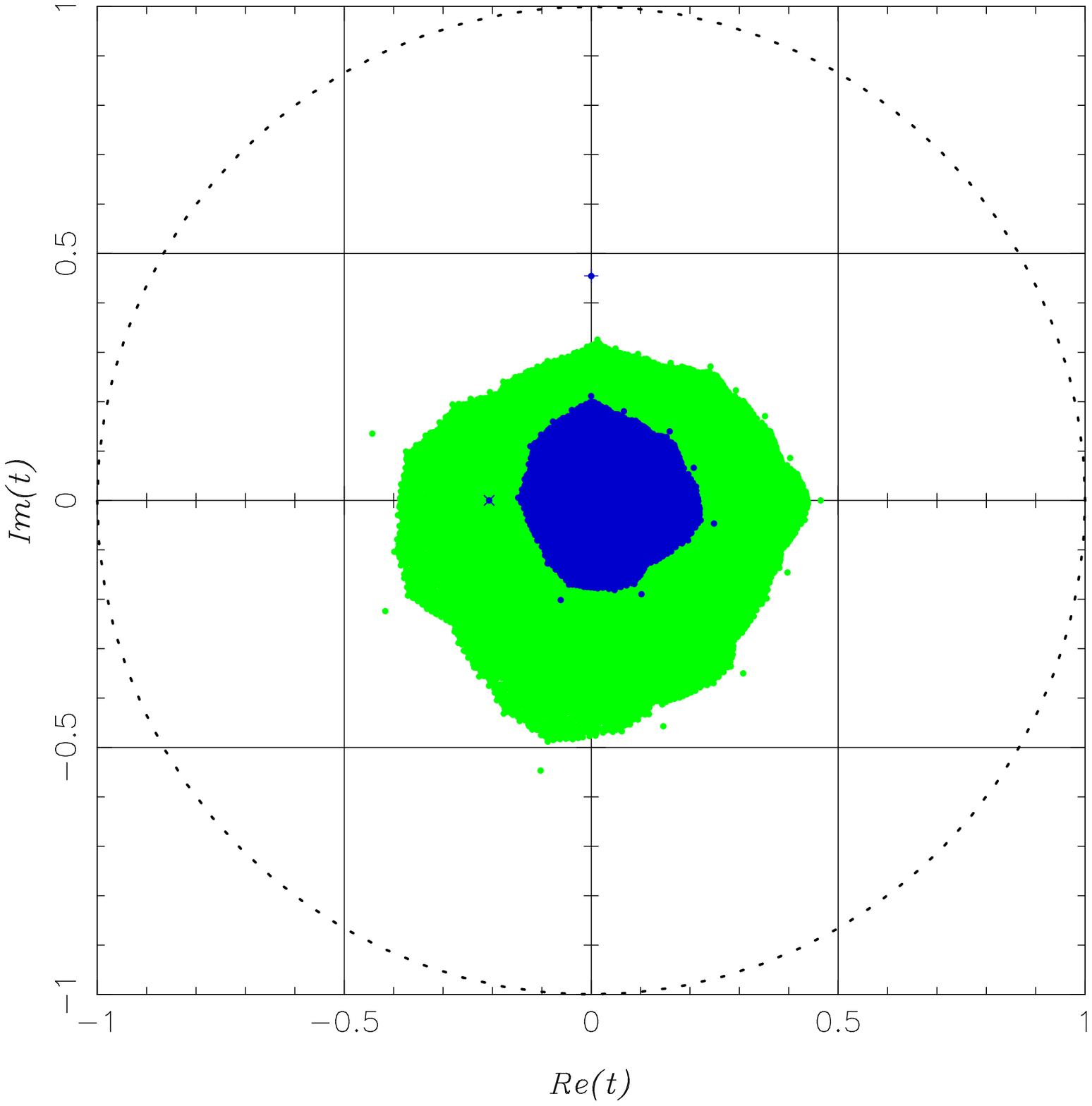} \\[1mm]
\footnotesize (c) $q = 2.100 + 1.905 i \approx 1 + 2.2 e^{\pi i/3}$  &
   \qquad\qquad
\footnotesize (d) $q = 1.000 + 2.200 i \approx 1 + 2.2 e^{\pi i/2}$  \\[8mm]
\includegraphics[width=165pt]{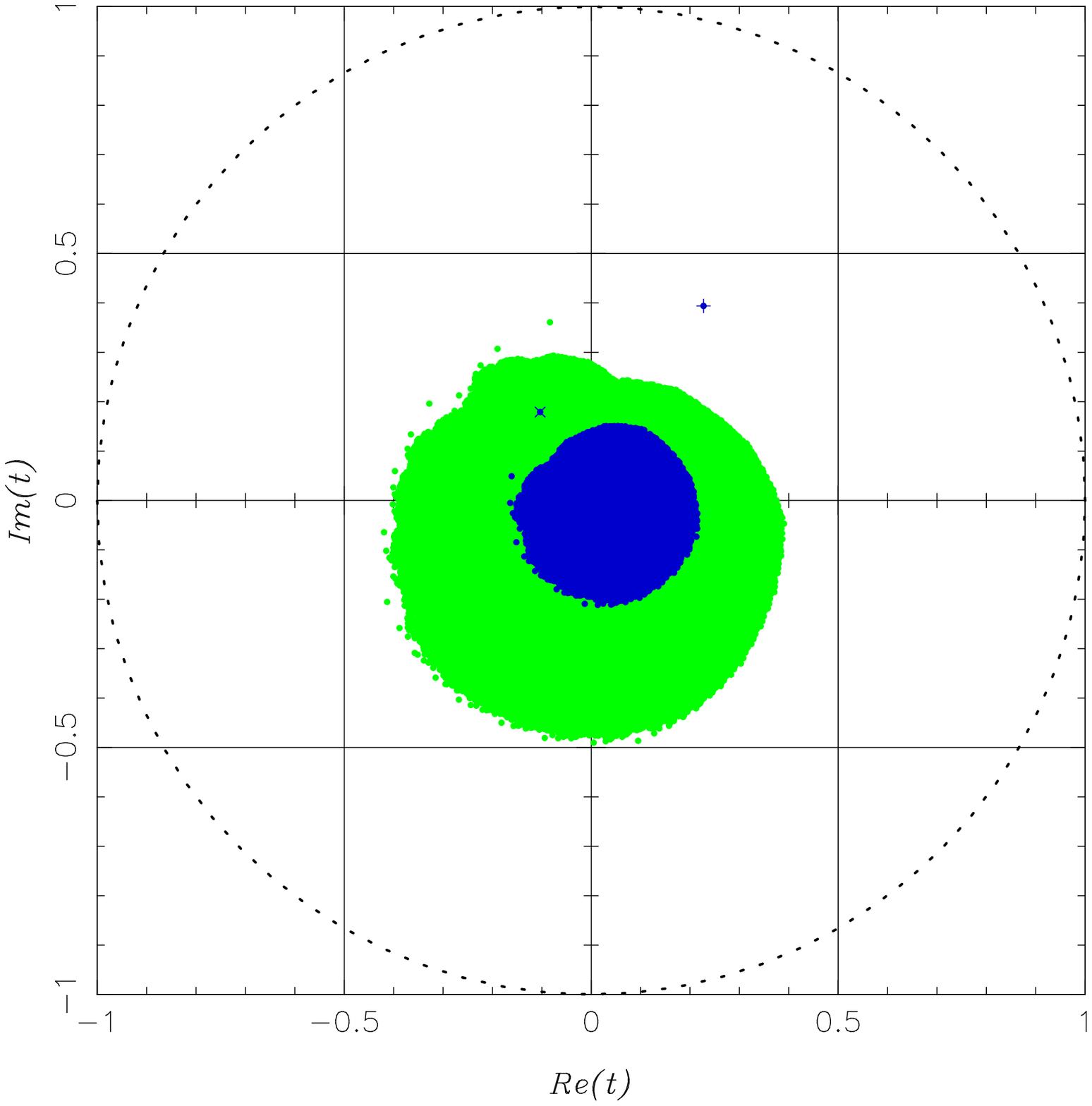} &  \qquad\qquad
\includegraphics[width=165pt]{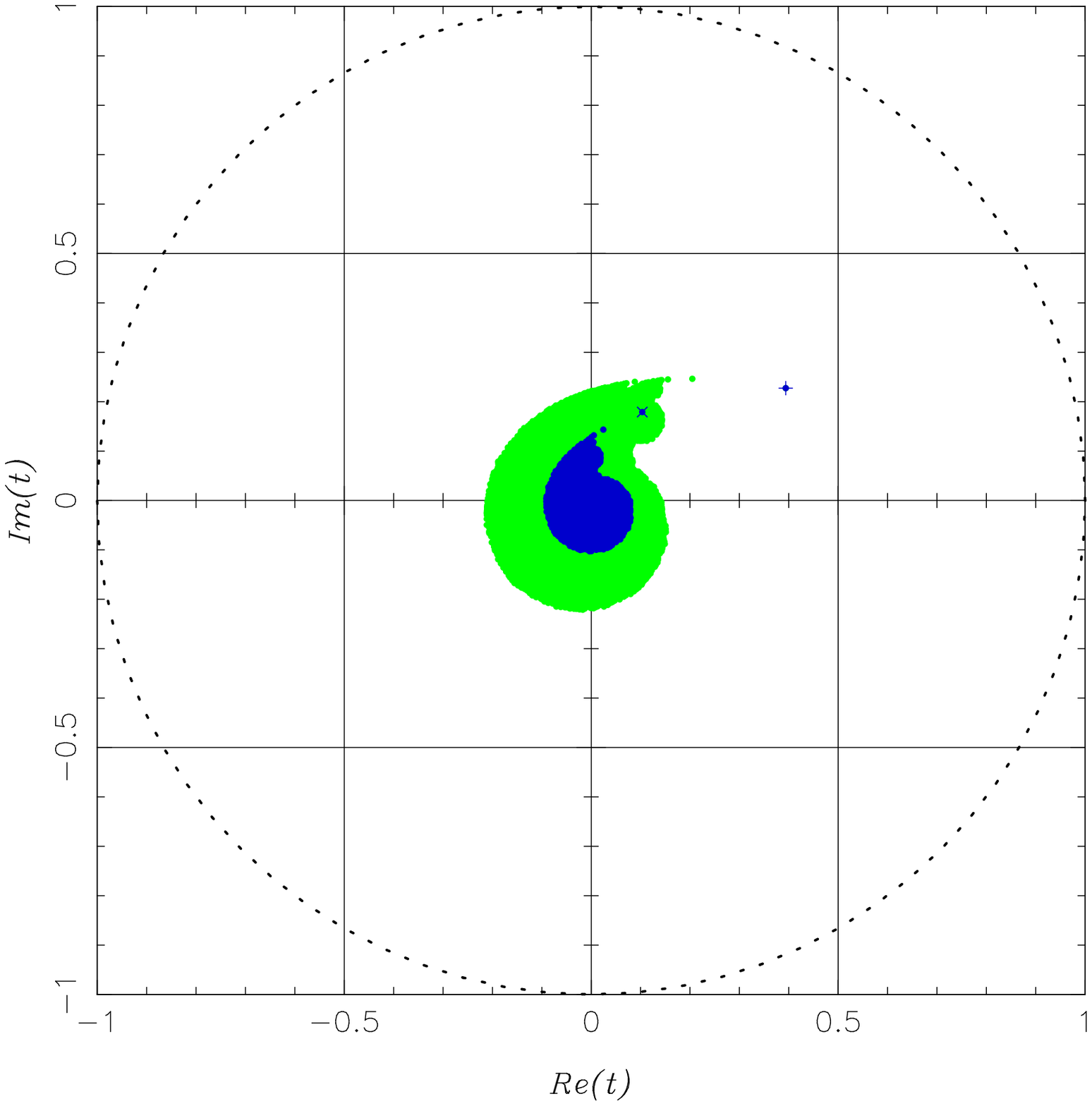} \\[1mm]
\footnotesize (e) $q =-0.100 + 1.905 i \approx 1 + 2.2 e^{2\pi i/3}$ &
   \qquad\qquad
\footnotesize (f) $q = -0.905 + 1.100 \approx 1 + 2.2 e^{5\pi i/6}$
\end{tabular}
\vspace*{4mm}
\caption{
  Computer-generated approximations to $S_1$ (dark blue) and
  $S_2$ (light green) in the complex $t$-plane,
  for $\Lambda=3$ and selected values of $q$.
  Note that we always have $S_1 \subseteq S_2$
  and $S_1 = \{t_0\} \cup t_0 S_2$ where $t_0 = 1/(1-q)$.
  The points $t_0$ and $t_0 \ser t_0 = t_0^2$,
  which both belong to $S_1$,
  are shown as dark blue $+$ and $\times$, respectively.
  The circle $|t|=1$ is shown for reference in dashed black.
}
  \label{fig_S1S2}
\end{figure}

In fact we need to be a bit more careful,
because every region $S_i$ must contain the point $t_0 = 1/(1-q)$,
but taking the smallest region $S_1$ to be a disc of radius
$\rho = |t_0| = |1/(1-q)|$ {\em cannot}\/ give very good bounds.
Indeed, with this choice of $S_1$ there exist graphs $G$
of maxmaxflow $\Lambda$ having roots $Z_G(q,v) = 0$
with $v \in S_1$ and $q$ growing {\em exponentially}\/ in $\Lambda$
(more precisely like $2^\Lambda$).\footnote{
   Just take $G = K_2^{(k)}$ (i.e.\ $k$ edges in parallel),
   which has maxmaxflow $k$.
   Consider $q < 0$, and write $q=-Q$ for simplicity.
   Then $\rho = 1/(1+Q)$, and the point $t = -\rho = -1/(1+Q) \in S_1$
   corresponds to $v = Q/(Q+2)$.
   Then $Z_{K_2^{(k)}}(q,v) = q + (1+v)^k - 1$
   vanishes when $[(2Q+2)/(Q+2)]^k = Q+1$,
   which occurs for large $k$ at $Q = 2^k - k - 1 + O(k^2/2^k)$.

   What is going on here is that $v = Q/(Q+2)$ is
   {\em strongly ferromagnetic}\/:
   for $Q \gg 1$ we have $v \approx 1$, hence $y = 1+v \approx 2$;
   so putting $k$ such edges in parallel
   leads to a weight that grows like $2^k$.
   Similar behavior will occur whenever $S_1$ contains any point having
   $|1+v|$ uniformly larger than 1.
   Indeed, we expect large roots in the $q$-plane whenever
   $S_1$ contains any point having $|1+v| - 1 \gg 1/|q|$.
 \label{footnote_expgrowth}
}

However,
%% taking the {\em largest}\/ region $S_{\Lambda-1}$
%% to be a disc of radius $\rho$,
%% and then letting every other region be a ``point + disc''
a slight modification works:
namely, we take each region $S_i$ to be a ``point + disc''
\begin{equation}
S_i \;=\;  \{1/(1-q)\} \cup D(r_i)
  \label{regionsequation}
\end{equation}
where $D(r_i)$ is a closed disc of radius $r_i$ centered at the origin.
This choice results in a situation that is both amenable to analysis
and also yields good bounds when the radii $r_i$ are suitably chosen,
as we will prove in this section.

The disc $D(r_1)$ must have radius at least $\rho^2$
because it must contain the point $t_0 \ser t_0 = 1/(1-q)^2$.
% [Soon we shall choose $r_1 = \rho^2$,
%  but for now let us assume only that $r_1 \ge \rho^2$.]
%% and so we may as well assume that $r_1 = \rho^2$.
So choose some $r_1 \ge \rho^2$;
this choice of $r_1$ sets a lower bound on the possible values for $r_2$
because $S_1 \pllq S_1 \subseteq S_2$.
%% and so we may as well take $r_2$ to the largest modulus
%% of all the points (other than $t_0$) in $S_1 \pllq S_1$.
Continuing in this fashion, $r_1$ and $r_2$ 
determine the minimum allowable value for $r_3$;
then $r_1$, $r_2$ and $r_3$ determine the minimum allowable value for $r_4$;
and so on.
Ultimately this process determines a minimum allowable value
for $r_{\Lambda-1}$;
and if $r_{\Lambda-1} \leq \rho$, then the set of radii
$r_1,r_2,\ldots,r_{\Lambda-1}$
yields a set of regions $S_i$ defined by \eqref{regionsequation}
that satisfies the conditions of Theorem~\ref{regions}.
%% {\bf Cite also Lemma~\ref{lemma4.3}!!!!!}
We formalize this observation in the following proposition:

\begin{prop}
  \label{prop_radii}
Let $\Lambda \ge 2$ be a fixed integer;
then let $q$ be a fixed complex number satisfying $|q-1|>1$,
and set $t_0 = 1/(1-q)$ and $\rho = |t_0|$.
If the real numbers $r_1, r_2, \ldots, r_{\Lambda-1}$ satisfy
\begin{equation}
\rho^2 \le r_1 \leq r_2 \leq \cdots \leq r_{\Lambda-1} \leq \rho
\end{equation}
and
\begin{equation}
\label{radiuscondition}
r_{s} \;\geq\;
   \max \{|t_e \pllq t_f|  \colon\;
          t_e \in D(r_k),\, t_f \in D(r_\ell),\, k+\ell = s\}
\end{equation}
for $2 \leq s \leq \Lambda-1$,
then the set of regions $S_1$, $S_2$, $\ldots$, $S_{\Lambda-1}$ defined by
\begin{equation}
S_i \;=\; \{1/(1-q)\} \cup D(r_i)
  \label{eq.prop_radii.S_i}
\end{equation}
satisfies the conditions of Theorem~\ref{regions}.
\end{prop}

\begin{proof}
We need to show that the four conditions of Theorem~\ref{regions} hold.
%% Condition (1) follows because the series operation is simply
%% complex multiplication, and as every element of $S_k$ and $S_{\ell}$
%% has modulus at most $\rho$, every element of $S_k \ser S_{\ell}$
%% has modulus at most $\rho^2$ and so is contained in $S_1$.
Condition (1) holds by Lemma~\ref{lemma4.3} with $r=\rho$.
To check condition (2), we observe that
\begin{equation}
   S_k \pllq S_{\ell}
   \;=\;
   \left( \{t_0\} \cup D(r_k) \right) \pllq
   \left( \{t_0\} \cup D(r_{\ell}) \right)
   \;=\; \{t_0\} \cup \left( D(r_k) \pllq D(r_\ell) \right)
\end{equation}
because $t_0 \pllq t = t_0$ for every $t$.
Therefore condition \eqref{radiuscondition} on the radii
is exactly what is needed to ensure that
$S_k \pllq S_\ell \subseteq S_{k+\ell}$.
Condition (3) holds because $S_{\Lambda-1} \subseteq D(\rho)$ and $\rho < 1$.
Finally, condition (4) fails only if
there are $t_e \in S_k$ and $t_f \in S_\ell$
(with $k+\ell = \Lambda$, though we do not even need to use this constraint)
such that $t_e t_f = 1/(1-q)$,
but this is impossible because $|t_e t_f| \leq \rho^2 < \rho = 1/|1-q|$.
\qed
\end{proof}

\medskip

To apply this theorem, we need to be able to bound the modulus of
\begin{equation}
t_e \pll_q t_f  \;=\;  \frac{t_e + t_f + (q-2) t_e t_f}{1+(q-1) t_e t_f}
  \label{parallel_t2}
\end{equation}
when $t_e \in D(r_k)$ and $t_f \in D(r_\ell)$.
Since the maximum modulus of $t_e \pllq t_f$ occurs
when $t_e$ and $t_f$ are on the boundaries of their respective discs,
let us define for $x,y \in [0,\rho)$ the function
\begin{equation}
  \label{fxydefinition}
f_q(x,y) \;:=\;
   \max \{|t_e \pllq t_f| \,\colon\; |t_e| = x, \, |t_f| = y\}  \;.
\end{equation}
If we bound \eqref{parallel_t2} in the most naive way
by replacing the numerator by an upper bound
and the denominator by a lower bound,
and we furthermore use $|q-2| \le |q-1| + 1$
to express the $q$-dependence in terms of the single number $|q-1|$,
then we get
\begin{equation}
  \label{radiusbound}
f_q(x,y) \;\leq\; F_q(x,y) \;:=\;
   \frac{x + y + (|q-1| + 1) xy} {1 - |q-1|xy}
   \;=\;
   \frac{x + y + (\rho^{-1} + 1) xy} {1 - \rho^{-1} xy}
   \;.
\end{equation}
The condition $x,y \in [0,\rho)$ ensures that the denominator of $F_q(x,y)$
is strictly positive.
Therefore, given the chosen value of $r_1$,
we can define a sequence of radii $r_2,r_3,\ldots$
satisfying \eqref{radiuscondition} using the iteration
\begin{equation}
%% r_1 & = & \rho^2  \slabel{equationmax1}  \\
r_s  \;=\; \max \{ F_q(r_k,r_\ell) \colon\; k+\ell = s \}
   \label{equationmax2}
\end{equation}
(stopping the iteration whenever a result $r_s$ becomes $\ge\rho$).
It is immediate that $r_1 \le r_2 \le \ldots \;$.
If the iteration remains well-defined up to $s=\Lambda-1$
and satisfies $r_{\Lambda-1} \le \rho$,
%% In particular, if $r_{s} \leq \rho$ for all $s \leq \Lambda-1$,
then the radii satisfy the hypotheses of Proposition~\ref{prop_radii}.
%%  (Note that the sequence $r_1$, $r_2$, $\ldots$
%%  given by \eqref{equationmax1} and \eqref{equationmax2}
%%  is only deemed to be properly defined while $r_s \leq \rho$.)
(Henceforth let us write $F$ in place of $F_q$ to lighten the notation.)

At first sight, this seems rather unappealing for analysis
because the $\max$ in \eqref{equationmax2} appears difficult to handle.
However, this difficulty is illusory because it turns out that
$F$ is actually an {\em associative}\/ function:

\begin{lem}
\label{lem:fabc}
Let $G$ be a function of the form
\begin{equation}
G(x,y) \;=\; \frac{x + y + A x y}{1 + B x y}
\end{equation}
where $A$, $B$ are arbitrary constants. Then
\begin{equation}
G(x,G(y,z)) \;=\; G(y,G(x,z)) \;=\; G(z,G(x,y))  \;.
\end{equation}
\end{lem}

\begin{proof}
Direct calculation shows that
\begin{equation}
G(x,G(y,z)) \;=\;
 \frac{(x+y+z) + A (xy + yz + xz) + (A^2 + B) xyz}{1 + B(xy+xz+yz) + ABxyz}
   \;,
\end{equation}
which is clearly symmetric under all permutations of $\{x,y,z\}$.
\qed
\end{proof}

\begin{cor}
  \label{cor.fabc}
If $F$ is given by \eqref{radiusbound}
and $r_2, \ldots, r_{\Lambda-1}$
by \eqref{equationmax2}, then 
\begin{equation}
F(r_k, r_\ell) \;=\;  F(r_1, r_{k+\ell-1})
\end{equation}
for all pairs $k,\ell$ of positive integers such that $k+\ell \leq \Lambda$.
\end{cor}

\begin{proof}
We prove this by induction on $s=k+\ell$.
The result clearly holds for $s=2$.
So suppose that the result is true for all $k'+\ell' < k+\ell$. Then 
\begin{equation}
F(r_k, r_\ell) = F( F(r_1, r_{k-1}), r_{\ell}) = F(r_1,F(r_{k-1},r_\ell)) = F(r_1, r_{k+\ell-1})
\end{equation}
and the result holds.
\qed
\end{proof}

The key point of this lemma (which was used implicitly in the proof)
is that all the terms in \eqref{equationmax2} are actually the same,
and so we can arbitrarily choose any one of them to define $r_s$.
So let us take $r_{s+1} = F(r_1, r_s)$, i.e.
\begin{equation}
  r_{s+1}  \;=\;  \frac{[1 + (\rho^{-1}+1)r_1] r_s + r_1}{1-\rho^{-1} r_1 r_s}
  \;.
 \label{rstransform}
\end{equation}
Since the map $r_s \mapsto r_{s+1}$ is a M\"obius transformation,
we can obtain an explicit expression for $r_k$:

\begin{lem}
  \label{lemma.fabc.bis}
For fixed real numbers $r_1$ and $\rho \neq 1$,
define a sequence $r_1, r_2, \ldots \in \R \cup \{\infty\}$ by
\begin{equation}
  r_{s+1}  \;=\;  \frac{[1 + (\rho^{-1}+1)r_1] r_s + r_1}{1-\rho^{-1} r_1 r_s}
  \;.
   \label{rk_explicit_0}
\end{equation}
Then 
\begin{equation}
r_k  \;=\; 
\rho \, \frac{ (1+r_1/\rho)^k - (1+r_1)^k } {(1+r_1)^k - \rho(1+r_1/\rho)^k} \;.
   \label{rk_explicit}
\end{equation}
\end{lem}

\begin{proof}
The map $r_s \mapsto r_{s+1}$ is a (real) M\"obius transformation of the form
\begin{equation}
x \;\mapsto\; \frac{a x + b}{cx + d}
\end{equation}
whose coefficients can be displayed in a suitable matrix
\begin{equation}
M \,=\, \left( \begin{array}{cc} a&b\\ c&d \end{array} \right) \,=\, 
\left( \begin{array}{cc} 1 + (\rho^{-1}+1)r_1 & r_1 \\
                        -\rho^{-1} r_1 & 1 \end{array} \right) \,.
\end{equation}
By standard results on M\"obius transformations,
the matrix $M^k$ represents the $k$th iterate of this transformation.
Now, the matrix $M$ has eigenvalues $1+r_1/\rho$ and $1+r_1$,
and it can be diagonalized by $M = QDQ^{-1}$ where
\begin{eqnarray}
D & = & \left(\begin{array}{cc}1+r_1/\rho & 0 \\0 & 1+r_1 \end{array}\right)
   \\[2mm]
Q & = & \frac{1}{1-\rho}  \left(\begin{array}{cc} 1 & -\rho \\ -1 & 1\end{array}
  \right)  \\[2mm]
Q^{-1} & = & \left(\begin{array}{cc}1 & \rho \\1 & 1\end{array}\right)
 \end{eqnarray}
It follows immediately that $M^k = QD^kQ^{-1}$ and so
\begin{equation}
M^k  \;=\;  \frac{1}{1-\rho} 
\left(\begin{array}{cc}
(1+r_1/\rho)^k - \rho(1+r_1)^k &
\rho [(1+r_1/\rho)^k - (1+r_1)^k] \\[2mm]
(1+r_1)^k-(1+r_1/\rho)^k &
(1+r_1)^k - \rho (1+r_1/\rho)^k  \\
\end{array}
\right).
\end{equation}
Treating this as a M\"obius transformation and applying it to $r_0 = 0$,
we get $r_k = (M^k)_{12} / (M^k)_{22}$ and thus
\begin{equation}
r_k  \;=\; 
\rho \, \frac{ (1+r_1/\rho)^k - (1+r_1)^k } {(1+r_1)^k - \rho(1+r_1/\rho)^k} \;.
\end{equation}
This also reproduces the correct value at $k=1$.
\qed
\end{proof}

{\bf Remarks.}
1. The formula \reff{rk_explicit}, once we have it,
can of course be proven by an easy induction on $k$.
But we thought it preferable to give a more conceptual proof that
shows where \reff{rk_explicit} comes from. Note also that we can
rewrite \reff{rk_explicit} as
\be
r_k \:=\: \rho \, \frac{X^k - 1}{1 - \rho X^k}
\quad \text{where} \quad
X \:=\: \frac{1+r_1/\rho} {1+r_1}
   \;;
\label{rk_explicitX}
\ee
this will be useful later.

2. The reasoning in Lemma~\ref{lem:fabc}, Corollary~\ref{cor.fabc}
and Lemma~\ref{lemma.fabc.bis} can be made even more explicit
by observing that the associative function $G(x,y) = (x+y+Axy)/(1+Bxy)$
is actually conjugate to $\widehat{G}(X,Y) = XY$:
it suffices to make the M\"obius change of variables
$X = f(x) := (1+\alpha x)/(1+\beta x)$ with
\begin{subeqnarray}
   \alpha  & = &  {A \pm \sqrt{A^2+4B}  \over 2}  \\[3mm]
   \beta   & = &  {A \mp \sqrt{A^2+4B}  \over 2}
\end{subeqnarray}
and we then have
\be
   f\Big( G\big( f^{-1}(X), \, f^{-1}(Y) \big) \Big)
   \;=\;
   XY
   \;.
\ee
In our application we have $A = 1 + \rho^{-1}$ and $B = -\rho^{-1}$,
hence $\alpha = \rho^{-1}$ and $\beta = 1$ (or the reverse).
Therefore, defining $R_k = f(r_k) := (1+ \rho^{-1} r_k)/(1+r_k)$,
we have simply $R_k = R_1^k$, which is equivalent to \reff{rk_explicit}.
Further information on associative rational functions
in two variables can be found in \cite{Brawley_01}.
\qed

\medskip

The final step in proving Theorem~\ref{thm_main_sec5}
is to show that, for suitable $q$,
we can choose $r_1 \ge \rho^2$ and have $r_k \leq \rho$
for $1 \le k \le \Lambda-1$.
Whenever this is the case, the radii $r_1,r_2,\ldots,r_{\Lambda-1}$
defined by \reff{rk_explicit_0}/\reff{rk_explicit}
will satisfy the conditions of Proposition~\ref{prop_radii},
and hence the set of nested ``point + disc'' regions $S_i$
will satisfy the conditions of Theorem~\ref{regions},
thereby certifying that $Z_G(q, \bv) \neq 0$ whenever
$G$ is a series-parallel graph of maxmaxflow at most $\Lambda$
and $v_e/(q+v_e) \in S_1$ for all edges $e$.

The simplest choice is to take $r_1 = \rho^2$ exactly;
then from \reff{rk_explicit} we have
\be
   r_k  \;=\;
   \rho \, \frac{X^k - 1}{1 - \rho X^k}
\quad \text{where} \quad
X \:=\: \frac{1+\rho}{1+\rho^2}
   \;.
 \label{eq.choice1}
\ee
When this choice works
(i.e.\ satisfies $r_k \le \rho$ for $1 \le k \le \Lambda-1$),
it yields the {\em minimal}\/ regions $S_i$
of the form \reff{eq.prop_radii.S_i}
that satisfy the conditions of Proposition~\ref{prop_radii}.
However, a slightly better choice is to take $r_{\Lambda-1} = \rho$ exactly;
simple algebra using \reff{rk_explicit} then shows that
\be
   r_k \;=\;  \rho \, \frac{X^k - 1}{1 - \rho X^k}
   \quad\hbox{where}\quad
   X \,=\, \biggl( {2 \over 1+\rho} \biggr)^{1/(\Lambda-1)}
   \;.
 \label{eq.choice2}
\ee
When this choice works
(i.e.\ satisfies $\rho^2 \le r_k \le \rho$ for $1 \le k \le \Lambda-1$),
it yields the {\em maximal}\/ regions $S_i$
of the form \reff{eq.prop_radii.S_i}
that satisfy the conditions of Proposition~\ref{prop_radii},
and hence the largest allowed set $S_1$ of edge weights.\footnote{
   Of course, for people who care only about the chromatic polynomial,
   these two choices are equally good.
   They differ only in the allowed set of edge weights $v_e \neq -1$.
}
The following lemma shows that these two choices work in precisely
the same set of circumstances, namely when $\rho \le \rho^\star_\Lambda$
where $\rho^\star_\Lambda$ is defined by
\reff{eq.rhostarlambda}/\reff{eq.lemma_sharp_rkbound_NEW}.
In the borderline case $\rho = \rho^\star_\Lambda$
both choices yield the same sequence,
which satisfies {\em both}\/ $r_1 = \rho^2$ {\em and}\/ $r_{\Lambda-1} = \rho$.
But when $\rho < \rho^\star_\Lambda$ we get different sequences,
and we prefer to use the second choice because it yields a larger region $S_1$.

\begin{lem}
   \label{lemma_sharp_rkbound_NEW}
For $\rho \in (0,1)$ and integer $\Lambda \ge 2$, the following are equivalent:
\begin{itemize}
   \item[(a)] There exist real numbers $r_1,\ldots,r_{\Lambda-1}$
       satisfying \reff{rk_explicit_0} and
       $\rho^2 \le r_1 \le \ldots \le r_{\Lambda-1} \le \rho$.
   \item[(b)] The sequence defined by \reff{eq.choice1}
       satisfies $\rho^2 \le r_k \le \rho$ for $1 \le k \le \Lambda-1$.
   \item[(c)] The sequence defined by \reff{eq.choice2}
       satisfies $\rho^2 \le r_k \le \rho$ for $1 \le k \le \Lambda-1$.
   \item[(d)] $(1+\rho)^\Lambda  \:\le\:  2 (1+\rho^2)^{\Lambda-1}$.
   \item[(e)] $\rho \le \rho^\star_\Lambda$,
where $\rho^\star_\Lambda$ is the unique solution of
\be
   (1+\rho)^\Lambda  \;=\;  2 (1+\rho^2)^{\Lambda-1}
 \label{eq.lemma_sharp_rkbound_NEW}
\ee
in the interval $(0,1)$  when $\Lambda \ge 3$, and $\rho^\star_2 = 1$.
\end{itemize}
\end{lem}

Let us remark that the equation \reff{eq.lemma_sharp_rkbound_NEW}
has $\rho=1$ as a root, so that after division by $\rho-1$
it reduces to a polynomial equation of degree $2\Lambda-3$.

\proofof{Lemma~\ref{lemma_sharp_rkbound_NEW}}
Fix $\rho \in (0,1)$ and $r_1 > 0$ and define a sequence
$r_1,r_2,\ldots,r_{\Lambda-1}$ by \reff{rk_explicit_0}/\reff{rk_explicit};
or equivalently, fix $\rho \in (0,1)$ and $X > 1$ and define
$r_1,r_2,\ldots,r_{\Lambda-1}$ by \reff{rk_explicitX}.
It is then easy to see that we have
$r_1 \le r_2 \le \ldots \le r_{\Lambda-1} < \infty$
if and only if $X < \rho^{-1/(\Lambda-1)}$
[so that the denominator in the expression \reff{rk_explicitX}
 for $r_k$ is positive for all $k \le \Lambda-1$];
and each $r_k$ is an increasing function of $X$ (for fixed $\rho$)
in the region $1 < X < \rho^{-1/(\Lambda-1)}$.
If we furthermore want to have $r_1 \ge \rho^2$ and $r_{\Lambda-1} \le \rho$,
then we must have
\be
   \frac{1+\rho}{1+\rho^2}
   \;\le\;
   X
   \;\le\;
   \biggl( {2 \over 1+\rho} \biggr)^{1/(\Lambda-1)}
 \label{eq.proof.lemma_sharp_rkbound_NEW}
\ee
(note that $[2/(1+\rho)]^{1/(\Lambda-1)} < \rho^{-1/(\Lambda-1)}$);
and by the just-observed monotonicity in $X$,
this condition is necessary and sufficient.
This proves the equivalence of (a), (b) and (c).
Moreover, there exists such an $X$ if and only if
\be
   \frac{1+\rho}{1+\rho^2}
   \;\le\;
   \biggl( {2 \over 1+\rho} \biggr)^{1/(\Lambda-1)}
   \;,
 \label{eq.proof.lemma_sharp_rkbound_NEW.2}
\ee
which is equivalent to (d).
So (a)--(d) are all equivalent.

Finally we shall prove the equivalence of (d) and (e).
We do this in slightly greater generality than is claimed,
namely for all {\em real}\/ $\Lambda \ge 2$.
Consider the function
\be
   f_\Lambda(\rho)  \;=\;  \Lambda \log(1+\rho) \,-\, (\Lambda-1) \log(1+\rho^2)
   \;.
\ee
Clearly (d) holds if and only if $f_\Lambda(\rho) \le \log 2$.
Now the first two derivatives of $f_\Lambda(\rho)$ are
\begin{subeqnarray}
   f'_\Lambda(\rho)
   & = &
   {\Lambda \over 1+\rho} \,-\, {2(\Lambda-1) \rho \over 1+\rho^2}  \\[3mm]
   f''_\Lambda(\rho)
   & = &
   - \, {4(1+\rho+\rho^2-\rho^3) \,+\,
            (\Lambda-2) (3+4\rho+2\rho^2 - 4\rho^3 - \rho^4)
         \over (1+\rho)^2 (1+\rho^2)^2
        }
\end{subeqnarray}
For $0 \le \rho \le 1$ we manifestly have
$1+\rho+\rho^2-\rho^3 \ge 1+\rho \ge 1$ and
$3+4\rho+2\rho^2 - 4\rho^3 - \rho^4 \ge 3 + \rho^2 \ge 3$,
so that $f_\Lambda$ is strictly concave on $[0,1]$ whenever $\Lambda \ge 2$.
We have $f_\Lambda(0) = 0$, $f'_\Lambda(0) = \Lambda > 0$,
$f_\Lambda(1) = \log 2$ and $f'_\Lambda(1) = -(\Lambda-2)/2$.
Therefore, for $\Lambda > 2$, there is a unique $\rho^\star_\Lambda \in (0,1)$
satisfying $f_\Lambda(\rho^\star_\Lambda) = \log 2$;
and for $\rho \in [0,1)$ we have $f_\Lambda(\rho) \le \log 2$
if and only if $\rho \le \rho^\star_\Lambda$.
This proves the equivalence of (d) and (e) for all real $\Lambda > 2$.
When $\Lambda = 2$, (d) holds for all $\rho \in [0,1]$,
so (d) is again equivalent to (e) with $\rho^\star_2 = 1$.
\qed

We have now completed the proof of the main part of
Theorem~\ref{thm_main_sec5}.
All that remains is to prove the final statement that
$\rho^\star_\Lambda > (\log 2)/(\Lambda- \smfrac{3}{2} \log 2)$
for all integers $\Lambda \ge 2$,
or equivalently (in view of Lemma~\ref{lemma_sharp_rkbound_NEW}d,e) that
$(1+\rho)^\Lambda < 2(1+\rho^2)^{\Lambda-1}$
when $\rho = (\log 2)/(\Lambda-\smfrac{3}{2} \log 2)$.
We shall actually prove this for all {\em real}\/
$\Lambda > \smfrac{5}{2} \log 2 \approx 1.732868$ (this ensures that $\rho < 1$).
Taking the logarithm of $2(1+\rho^2)^{\Lambda-1} / (1+\rho)^\Lambda$, substituting for $\Lambda$ in terms of $\rho$,  and parametrizing by $\rho \in (0,1)$, 
we see that this is equivalent to the following claim:

\begin{lem}
   \label{lemma_rkbound_log2}
The function
\begin{subeqnarray}
   g(\rho)
   & = &
   \rho \left[ \log 2 \:+\:
               \biggl( {\log 2 \over \rho} \,+\, \smfrac{3}{2} \log 2 \,-\, 1
               \biggr) \, \log(1+\rho^2)
           \:-\: \biggl( {\log 2 \over \rho} + \smfrac{3}{2} \log 2 \biggr)
               \, \log(1+\rho)
        \right]
      \nonumber \\ \\
   & = &
   \rho \log \left( \frac{2}{1+\rho} \right) 
      \,-\, (\log 2) \log \left( \frac{1+\rho}{1+\rho^2} \right)
      \,-\, (\smfrac{3}{2} \log 2 \,-\, 1) \, \rho \,
          \log \left( {1+\rho \over 1+\rho^2} \right)
\end{subeqnarray}
is strictly positive for $0 < \rho < 1$.
\end{lem}

\begin{proof}
The second derivative of $g$ is given by
$g''(\rho) = h(\rho) \times \rho / [ (1+\rho)^2 (1+\rho^2)^2 ]$ where
\be
   h(\rho) \;=\;
   - (2 - \smfrac{3}{2} \log 2) \rho^4  \,-\, (4 - 2\log 2) \rho^3
     \,-\, (8 - 5\log 2) \rho^2  \,-\, (12 - 14\log 2)\rho
     \,+\, (\smfrac{23}{2} \log 2 - 6)
   \;.
\ee
All the coefficients of $h(\rho)$ are strictly negative
except for the last (constant) term,
so we have $h'(\rho) < 0$ for all $\rho \geq 0$.
Since $h(0) = \smfrac{23}{2} \log 2 - 6 > 0$ and $h(1) = 34\log 2 - 32 < 0$
and $h$ is strictly decreasing for $\rho \ge 0$,
it follows that $h(\rho)$ has exactly one positive real root $\rho^*$
and that it lies between 0 and 1
(by computer $\rho^* \approx 0.417876$).
Therefore $g$ is strictly convex on $[0,\rho^*]$
and strictly concave on $[\rho^*,\infty)$.
Since $g(0) = g'(0) = 0$, we have $g(\rho) > 0$ for $\rho \in (0,\rho^*]$.
Moreover, since $g(\rho^*) > 0$ and $g(1) = 0$
and $g$ is strictly concave on $[\rho^*,1]$,
we have $g(\rho) > 0$ for $\rho \in [\rho^*,1)$.
Hence $g(\rho) > 0$ for all $\rho \in (0,1)$, as claimed.
\qed
\end{proof}

\medskip

{\bf Remark.}
% The constant $\log 2$ in Lemma~\ref{lemma_rkbound_log2}
% is indeed best possible.  To see this, set $\rho = c/k$
% for some constant $c>0$; then the quantity in \reff{eq.rkbound}
% behaves in the limit $k \to\infty$ as
% \be
%    \frac{ (1+c/k)^k - (1+c^2/k^2)^k } {(1+c^2/k^2)^k - (c/k)(1+c/k)^k}
%    \;\to\;  e^c - 1  \;,
% \ee
% which is $\le 1$ if and only if $c \le \log 2$.
% 
A straightforward calculation shows that the
large-$\Lambda$ asymptotic behavior of $\rho^\star_\Lambda$ is given by
\be
   \rho^\star_\Lambda
   \;=\;
   (\log 2)
   \Biggl[ {1 \over \Lambda-1} \,+\, {3 \log 2 - 2 \over 2 (\Lambda-1)^2}
              \,+\, {25 \log^2 2 - 24\log 2 + 6 \over 6 (\Lambda-1)^3}
              \,+\, \ldots
   \Biggr]
\ee
and hence
\be
   {1 \over \rho^\star_\Lambda}
   \;=\;
   {\Lambda-1 \over \log 2} \,-\, {3 \log 2 - 2 \over 2 \log 2}
              \,-\, {23\log 2 - 12 \over 12 (\Lambda-1)}
              \,+\, \ldots
   \;\;.
\ee
So the inequality
$\rho^\star_\Lambda > (\log 2)/(\Lambda - \smfrac{3}{2} \log 2)$
captures the first two terms of the large-$\Lambda$ asymptotic behavior.
\qed

\medskip

We have now completed the proof of Theorem~\ref{thm_main_sec5}.

\section{The case $\Lambda=3$}  \label{sec.Lambda=3}

Theorem~\ref{thm_main} is a strong result because it provides
a linear bound for the chromatic roots of series-parallel graphs
in terms of the maxmaxflow $\Lambda$, thereby achieving our main objective.
Furthermore, the constant $1/\log 2$ cannot be reduced below $1$
(see Appendix~\ref{app_tree}) and so it is reasonably close to optimal.
However, the result applies uniformly for all $\Lambda$,
its proof involves a number of steps where expressions
are replaced by fairly naive upper bounds,
and it only involves the {\em magnitude}\/ of $q-1$;
so for all these reasons,
Theorem~\ref{thm_main} does not give a very precise picture
of the root-free region for any particular value of $\Lambda$.

In this section we consider how to get sharper results
for the simplest nontrivial case, namely for $\Lambda = 3$. 
In this case, the bound given by Theorem~\ref{thm_main} is that 
chromatic roots for series-parallel graphs of maxmaxflow 3 are
contained in the disk
\be
|q-1| \;\le\; 2/(\log 2) \;\approx\; 2.8853900818 \;.
 \label{eq6.1}
\ee
An immediate improvement can be obtained from Theorem~\ref{thm_main_sec5}
by using the exact value of $\rho_3^\star$,
which gives the slightly better bound
\be
|q-1| \;\le\; 1/\rho_3^\star \;\approx\; 2.6589670819 \;.
 \label{eq6.2}
\ee

%In particular, we can make a more precise analysis
%of the very simple situation in which we take
%\begin{subeqnarray}
%   S_1  & = &  \{1/(1-q)\} \cup D(\rho^2)    \\
%   S_2  & = &  D(\rho)
%\end{subeqnarray}
%with, as usual, $\rho = 1/|1-q| < 1$.
%Then, by Proposition~\ref{prop_radii},
%all the conditions of Theorem~\ref{regions} are satisfied provided only that
%\begin{equation}
%D(\rho^2) \tvpllq{T} D(\rho^2) \;\subseteq\; D(\rho) \;,
%\end{equation}
%or equivalently that 
%\begin{equation}\label{frhorho}
%f_q(\rho^2, \rho^2) \;\leq\; \rho
%\end{equation}
%where $f_q$ is given by \eqref{fxydefinition}. If we use the bound 
%given by \eqref{radiusbound},
%i.e.\ replace $f_q$ by its upper bound $F_q$, then this reduces to
%\begin{equation}
%\frac{\rho^2 + \rho^2 + (\rho^{-1} + 1) \rho^4} {1 - \rho^{-1} \rho^4}
%   \;\leq\; \rho
%\end{equation}
%or equivalently
%\begin{equation}
%\rho - 2\rho^2 - \rho^3 - 2\rho^4  \;\ge\; 0
%\end{equation}
%[which corresponds to the $k=2$ case of \reff{inequality_r}].
%This yields a cubic equation that is satisfied when
%$\rho \ltapprox 0.3760858894$ or equivalently
%\begin{equation}
%|q-1| \;\gtapprox\; 2.6589670819  \;.
%\end{equation}
%This is a circular region that slightly improves the region
%$|q-1| \ge 2/(\log 2) \approx 2.8853900818$ given by Theorem~\ref{thm_main}.

Both of these regions ultimately relied on the quantity $F_q$
given by \eqref{radiusbound} as an upper bound for the true value $f_q$.
We can do better by computing a numerical approximation
to the {\em actual}\/ value $f_q(\rho^2, \rho^2)$,
and then imposing the condition $f_q(\rho^2, \rho^2) \le \rho$
that arises out of Proposition~\ref{prop_radii} with $\Lambda=3$.
Since $f_q(\rho^2, \rho^2)$ depends on $q$ and not just on $|q-1|$,
this procedure will lead to a region with no simple analytic description.
As $t_e \pllq t_f$ is given by a ratio of symmetric multiaffine polynomials
in $t_e$ and $t_f$ [cf.\ \reff{parallel_t2}]
and $D(\rho^2)$ is a circular region,
the Grace--Walsh--Szeg\H{o} coincidence theorem
\cite[Theorem~3.4.1b]{Rahman_02}
implies that
\begin{equation}
\max_{t_e, t_f \in D(\rho^2)} |t_e \pllq t_f| 
  \;=\; \max_{t \in D(\rho^2)} |t \pllq t|
   \;,
\end{equation}
and so we can compute an approximation to $f_q(\rho^2, \rho^2)$
by letting $t$ range over the (discretized) boundary of $D(\rho^2)$
and taking the maximum value of $|t \pllq t|$ thus obtained.
Then for each fixed angle $\theta$ we can set
$1/(1-q) = \rho e^{i \theta}$ and use the bisection method
to determine the maximum possible value of $\rho$.
Figure~\ref{figlambda3} shows how this bound
(shown as a green solid curve)
compares with the circular regions \reff{eq6.1} and \reff{eq6.2}.
This bound is the optimal bound obtainable from Theorem~\ref{regions}
under the assumption that $S_1$ is chosen to be a ``point+disk'' region
$S_1 = \{1/(1-q)\} \cup D(\rho^2)$.

\begin{figure}[t]
\begin{center}
\includegraphics{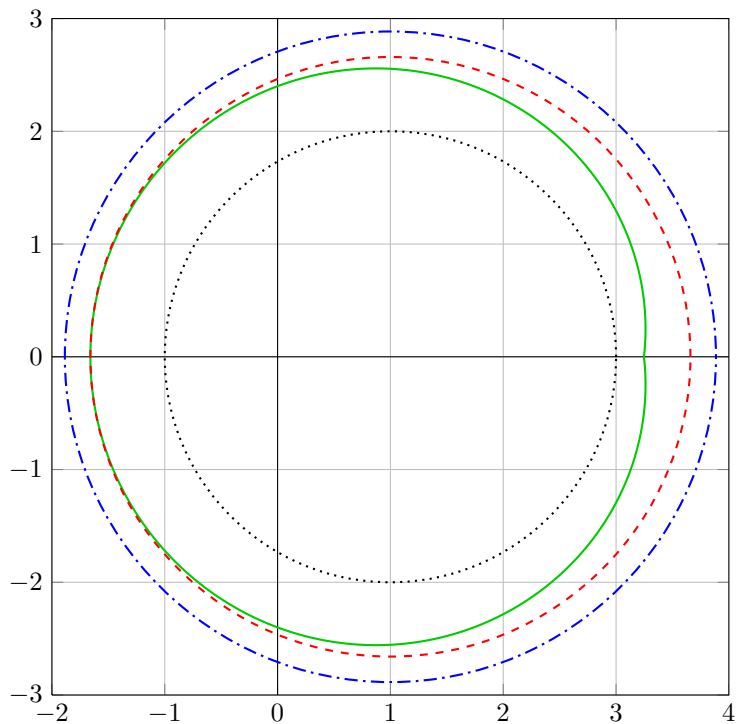}
\end{center}
\caption{
   Different bounds on the chromatic roots for $\Lambda=3$:
   the bound based on $f_q$ (green solid curve),
   the bound $|q-1| = 2.6589670819$ based on $F_q$ (red dashed circle),
   and the bound $|q-1|=2/\log 2 \approx 2.8853900818$
   from Theorem~\ref{thm_main} (blue dot-dashed outer circle).
   The inner circle $|q-1|=2$ is also shown for reference (dotted gray).
 }
\label{figlambda3}
\end{figure}

%% {\bf Don't forget to give 
%% the explicit example of a graph with $\Lambda=3$
%% and a chromatic root with $|q-1| > 2$!!!}

At this point it is natural to inquire:  What is the best possible result?
Otherwise put:
Can we describe exactly the closure of the set of {\em all}\/ chromatic roots
of {\em all}\/ series-parallel graphs of maxmaxflow $\Lambda=3$,
or at least the outer boundary of this set?
In \cite{Royle_BCC2009} one of the authors gave a
computer approximation to this boundary,
but he suspects that this approximation may become poor near the real axis,
in part because this boundary is likely to be fractal-like rather than smooth.

As previously mentioned, we will show in Appendix~\ref{app_tree} that,
for each fixed $r \ge 2$, every point of the circle $|q-1| = r$
is a limit point (as $n\to\infty$) of chromatic roots
of the family $\{G_n^r\}_{n \ge 1}$
of leaf-joined trees of branching factor $r$,
which have maxmaxflow $\Lambda=r+1$.
Moreover, numerical calculations suggest (though we have no proof)
that the chromatic roots of leaf-joined trees
always lie {\em inside}\/ the circle $|q-1| = r$
(see Conjecture~\ref{conj.leaf-joined}).
This led us to conjecture that the exact answer to our question is:
the chromatic roots of series-parallel graphs of maxmaxflow $\Lambda$
always lie inside the disc $|q-1| < \Lambda-1$, and this bound is sharp.

That would be neat, but it is {\em false}\/!
In fact, a counterexample can be found by a simple modification
of a leaf-joined tree.
Let us first recall \cite[Example~2.2]{MR2187739}
the multivariate Tutte polynomial of a cycle $C$:
\begin{equation}
   Z_{C}(q, \bv)   \;=\;
       \prod_{e \in E(C)}  (q+v_e)   \,+\, (q-1) \prod_{e \in E(C)}  v_e
   \;.
 \label{eq.ZCn}
\end{equation}
In particular, if we consider a cycle of $N+1$ edges
where $N$ edges carry weight $v$ and the last edge carries weight $-1$,
we have
\be
   Z_C(q,\bv)  \;=\;  (q-1) [(q+v)^N - v^N]
   \;,
\ee
which vanishes whenever $t = v/(q+v)$ is an $N$th root of unity.
It follows that if we consider {\em any}\/ 2-terminal graph $\sfG = (G,s,t)$
and form the graph $H$ consisting of $N$ copies of $\sfG$
together with one $K_2$ connected in a cycle,
then $H$ has a chromatic root whenever
the ``effective transmissivity'' $t_{\rm eff}(G,s,t)$
is an $N$th root of unity.

It is easy to compute the effective transmissivity
for leaf-joined trees, symbolically as a function of $q$,
by using the recursion \reff{def_Rq_chromatic}/\reff{def_init_Rq_chromatic}
along with $t = (y-1)/(q+y-1)$.
We can then plot the curve in the complex $q$-plane
where $|t_{\rm eff}(G_n^r,s,t)| = 1$.
For $r=2$, we find that this curve stays within the disc $|q-1|<2$
when $n \le 4$, but that it strays slightly outside this disc when $n=5$.
(On the circle $q-1=2e^{i\theta}$, $|t_{\rm eff}(G_5^2,s,t)|$ reaches a
maximum value $\approx 1.08448$ at $\theta \approx \pm 0.679954 \pi$,
 corresponding to $q \approx -0.071413 \pm 1.68881 i$.)
If we now consider the graph $H$ consisting of $N=3$ copies of $G_5^2$
together with one $K_2$ connected in a cycle ---
note that $H$ has maxmaxflow 3 and has 94 vertices ---
we see that $H$ has a chromatic root whenever
$t_{\rm eff}(G_5^2,s,t)$ is a cube root of unity.
Solving $t_{\rm eff}(G_5^2,s,t) = e^{\pm 2\pi i/3}$ for $q$,
we find 31 roots, of which one
($q \approx -0.144883 \mp 1.651418 i$)
has $|q-1| \approx  2.009462 > 2$.

\section{The real antiferromagnetic regime}  \label{sec.antiferro}

%% {\bf Maybe rewrite this introduction to be a bit more general,
%%    as we may be interested in other regions $\scrv$ as well.}

The chromatic polynomial corresponds to the special case
of the multivariate Tutte polynomial in which
all the edge weights $v_e$ take the value $-1$.
However, it is often the case that results valid for this limiting case
also hold throughout the ``real antiferromagnetic regime''
where edge weights $v_e \in [-1,0]$ are chosen independently for each edge.
Expressed in transmissivities, we get $t_e \in \scrc_q$,
where $\scrc_q$ is the curve defined parametrically by
% \begin{equation}\label{cqdef}
% \scrc_q \;=\;
%   \Bigl\{ \frac{\alpha}{\alpha-q} \colon\; \alpha \in [0,1] \Bigr\}
% \end{equation}
% (here $\alpha = -v$).
% In the complex $t$-plane, $\scrc_q$ traces out a circular arc that runs
% from the origin (when $\alpha = 0$) to the point $1/(1-q)$ [when $\alpha=1$].
\begin{equation}\label{cqdef}
\scrc_q \;=\;
  \Bigl\{ \frac{v}{q+v} \colon\; v \in [-1,0] \Bigr\}  \;.
\end{equation}
In the complex $t$-plane, $\scrc_q$ traces out a circular arc that runs
from the origin (when $v = 0$) to the point $1/(1-q)$ [when $v=-1$].
% To adapt Theorem~\ref{regions} to cover this and similar cases,
% we only need alter one of the five conditions.
% More precisely, suppose that $S_0$ denotes the allowed initial edge weights
% (expressed in transmissivities). If we replace condition (1) with 
%\be
%(1)' \qquad S_0 \in S_1
%\ee
%and leave the remaining conditions unchanged, then the condi
%% To adapt Theorem~\ref{regions} to this case,
%% it suffices simply to change the first condition $1/(1-q) \in S_1$
%% to the condition 
%% \be
%% (1') \quad \scrc_q \subseteq S_1
%% \ee with the remaining conditions unchanged.
%%
%% The existence of suitable regions $S_i$ satisfying $(1')$
%% and conditions (1)--(4) of Theorem~\ref{regions} then certifies that
%% $q$ is not a root of the partition function of any antiferromagnetic
%% Potts model on a series-parallel graph of maxmaxflow at most $\Lambda$. 
%%
In this section we will show how we can handle this case by a minor
modification of the argument given in Section~\ref{sec_discs},
thereby proving Theorem~\ref{thm_main_AF}.
In fact, we shall prove the following slight strengthening of
Theorem~\ref{thm_main_AF}, which is identical to Theorem~\ref{thm_main_sec5}
except that $v_e = -1$ is replaced by $-1 \le v_e \le 0$:

\begin{thm}
   \label{thm_main_AF_sec7}
Fix an integer $\Lambda \ge 2$,
and let $G$ be a loopless series-parallel graph of maxmaxflow at most $\Lambda$.
Let $\rho^\star_\Lambda$ be the unique solution of
\be
   (1+\rho)^\Lambda  \;=\;  2 (1+\rho^2)^{\Lambda-1}
\ee
in the interval $(0,1)$ when $\Lambda \ge 3$,
and let $\rho^\star_2 = 1$.
Then the multivariate Tutte polynomial $Z_G(q, \bv)$
is nonvanishing whenever $|q-1| \ge 1/\rho^\star_\Lambda$
(with $\ge$ replaced by $>$ when $\Lambda=2$)
and the edge weights $\bv = \{v_e\}_{e \in E}$
satisfy
\be
   -1 \,\le\, v_e \,\le\, 0   \quad\hbox{or}\quad
   \left| {v_e \over q+v_e} \right| \:\le\: \rho \, {X - 1 \over 1 - \rho X}
 \label{eq.thm_main_AF_sec7}
\ee
(with strict inequality in the second expression when $\Lambda=2$),
where
\be
   \rho \:=\: {1 \over |q-1|}
   \quad\hbox{and}\quad
   X \:=\: \biggl( {2 \over 1+\rho} \biggr)^{1/(\Lambda-1)}
   \;.
\ee
\end{thm}

The first step in the proof of Theorem~\ref{thm_main_AF_sec7}
is the following simple lemma,
which shows how to combine a pair of families
$C_1 \subseteq C_2 \subseteq \cdots \subseteq C_{\Lambda-1}$
and $D_1 \subseteq D_2 \subseteq \cdots \subseteq D_{\Lambda-1}$,
each of which satisfies the ``parallel condition'' (2) of Theorem~\ref{regions},
into a single family that also satisfies the ``parallel condition'':

\begin{lemma}
   \label{lemma_CD}
Let $C_1 \subseteq C_2 \subseteq \cdots \subseteq C_{\Lambda-1}$
and $D_1 \subseteq D_2 \subseteq \cdots \subseteq D_{\Lambda-1}$
be subsets of the complex $t$-plane satisfying
\begin{subeqnarray}
   C_k \pllq C_\ell \,\subseteq\, C_{k+\ell}
      \quad\hbox{\rm whenever } k+\ell \le \Lambda-1    \\[1mm]
   D_k \pllq D_\ell \,\subseteq\, D_{k+\ell}
      \quad\hbox{\rm whenever } k+\ell \le \Lambda-1
\end{subeqnarray}
Now define the sets
$S_1 \subseteq S_2 \subseteq \cdots \subseteq S_{\Lambda-1}$ by
\be
   S_k  \;=\;  \bigcup\limits_{i=0}^k \, (C_i \pllq D_{k-i})
 \label{eq.lemma_CD.defSk}
\ee
with $C_0 = D_0 = \{0\}$.
Then
\be
   S_k \pllq S_\ell \,\subseteq\, S_{k+\ell}
      \quad\hbox{\rm whenever } k+\ell \le \Lambda-1  \;.
\ee
\end{lemma}

\begin{proof}
If $k+l \le \Lambda-1$, we have
\begin{eqnarray}
   S_k \pllq S_\ell
   & = & 
   \bigcup\limits_{i=0}^k \, \bigcup\limits_{j=0}^\ell \,
         (C_i \pllq D_{k-i}) \pllq (C_j \pllq D_{\ell-j})   \nonumber \\[1mm]
   & = & 
   \bigcup\limits_{i=0}^k \, \bigcup\limits_{j=0}^\ell \,
         (C_i \pllq C_j) \pllq (D_{k-i} \pllq D_{\ell-j})   \nonumber \\[1mm]
   & \subseteq & 
   \bigcup\limits_{i=0}^k \, \bigcup\limits_{j=0}^\ell \,
         (C_{i+j} \pllq D_{k+\ell-i-j})   \nonumber \\[1mm]
   & = & 
   \bigcup\limits_{i=0}^{k+\ell} \,
         (C_{i} \pllq D_{k+\ell-i})   \nonumber \\[1mm]
   & = &
   S_{k+\ell}  \;.
\end{eqnarray}
\qed
\end{proof}

In Section~\ref{sec_discs} we treated the chromatic-polynomial case
by taking $C_1 = \ldots = C_{\Lambda-1} = \{t_0\}$ where $t_0 = 1/(1-q)$,
and $D_i = D(r_i)$.
The fact that $t_0 \pllq t = t_0$ for all $t$
--- which is very special to chromatic polynomials ---
then ensures that all the terms $1 \le i \le k$ in \reff{eq.lemma_CD.defSk}
equal $\{t_0\}$, while the term $i=0$ equals $D(r_k)$.
So we indeed have $S_k = \{t_0\} \cup D(r_k)$
as stated in Proposition~\ref{prop_radii},
and the proof of the ``parallel condition'' (2)
given as part of the proof of Proposition~\ref{prop_radii}
is a special case of Lemma~\ref{lemma_CD}.

To treat the real antiferromagnetic regime,
we will take $C_1 = \ldots = C_{\Lambda-1} = \scrc_q$ and $D_i = D(r_i)$
with $r_1 \le r_2 \le \ldots \le r_{\Lambda-1}$.
The invariance of the real antiferromagnetic regime under parallel connection
(which is most easily seen in the $v$-plane or $y$-plane)
then guarantees that $C_k \pllq C_\ell \subseteq C_{k+\ell}$.
It follows that
\be
   S_k  \;=\;  (\scrc_q \pllq D(r_{k-1})) \,\cup\, D(r_k)
 \label{def.Sk.AF}
\ee
where we have set $r_0 = 0$ and hence $D(r_0) = \{0\}$.
The sets $S_k$ are no longer ``point + disc'', but rather ``stalk + disc'':
for $S_1$ the ``stalk'' is precisely the curve $\scrc_q$,
while for higher $S_k$ the ``stalk'' gets increasingly ``fattened out''
by parallel connection with $D(r_{k-1})$.
Figure~\ref{stalkdisc} illustrates this situation
for $\Lambda=3$ and $q=-2+3i$:
the ``stalk'' for $S_2$ is the cone-shaped region
$\scrc_q \pllq D(r_1)$ that runs from $D(r_1)$ to the point $t_0 = 1/(1-q)$.

\begin{figure}[t]
\begin{center}
\includegraphics{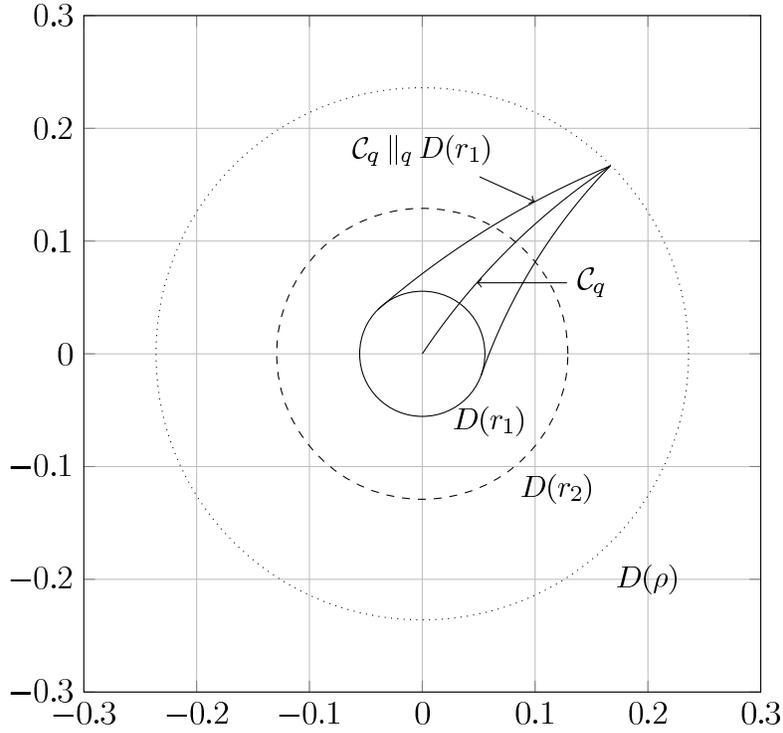}
\end{center}
\caption{The boundaries of the discs and stalks when $q=-2+3i$.}
\label{stalkdisc}
\end{figure}

We can choose the radii
$\rho^2 = r_1 \leq r_2 \leq \ldots \leq r_{\Lambda-1} \leq \rho$
exactly as in Section~\ref{sec_discs},
and this guarantees that $D_k \pllq D_\ell \subseteq D_{k+\ell}$.

To complete the proof of Theorem~\ref{thm_main_AF_sec7},
it therefore suffices to verify the ``series condition'' (1)
of Theorem~\ref{regions}.
%% or equivalently the condition (1${}'$) given immediately afterwards.
We shall do this, once again, by using Lemma~\ref{lemma4.3} with $r=\rho$.
Since $\rho^2 = r_1 \leq r_2 \leq \cdots \leq r_{\Lambda-1} \leq \rho$,
it suffices to verify that
\be
   \scrc_q \pllq D(r_{\Lambda-2})  \;\subseteq\;  D(\rho)
   \;.
 \label{eq.Cq_Drlam-2}
\ee
We shall prove the stronger statement that
\be
   \scrc_q \pllq D(\rho)  \;\subseteq\;  D(\rho)
   \;.
 \label{eq.Cq_Drho}
\ee
Indeed, we shall prove also a strong converse to this statement,
although we shall not make use of this converse.

\begin{lemma}
   \label{lemma_Drho_Cq}
Let $q$ be a fixed complex number such that $|q-1|>1$,
let $\rho = 1/|1-q|$, and let $\scrc_q$ be defined by \eqref{cqdef}.
Then a complex number $t$ satisfies $\{t\} \pllq D(\rho) \subseteq D(\rho)$
if and only if $t \in \scrc_q$.
[Otherwise put, we have $\scrc_q \pllq D(\rho) \subseteq D(\rho)$,
and $\scrc_q$ is the largest set with this property.]
\end{lemma}

\begin{proof}
It is easiest to change variables once again
and consider the situation in the complex $y$-plane,
where parallel connection is simply multiplication
$y_e \tvpll{Y} y_f = y_e y_f$
and the curve $\scrc_q$ corresponds to the segment $[0,1]$.
The relationship between $y$ and $t$ is given by the M\"obius transformation
\begin{equation}
y  \;=\; \frac{ (q-1) t + 1 } {-t + 1}  \;.
\end{equation}
%% so by standard results on M\"obius transformations (see ?? \cite{})
Since $D(\rho)$ is a closed disc in the complex $t$-plane having
the point $t_0 = 1/(1-q)$ on its boundary,
the image of $D(\rho)$ in the complex $y$-plane is a closed disc $D^Y(\rho)$
having the origin $y_0 = 0$ on its boundary.\footnote{
   The center of the disc $D^Y(\rho)$ is the point
   $c = [(q-1) \rho^2 + 1]/(1-\rho^2)$
%%   $$
%%      c \;=\; \frac{(q-1) \rho^2 + 1}{1-\rho^2}
%%   $$
%%   (and its radius is of course $|c|$),
   --- and the radius is of course $|c|$ ---
   but we do not actually need this explicit formula for $c$.
}
(Here we have used $|q-1|>1$.
 If $|q-1|=1$, then $D^Y(\rho)$ is a closed half-plane
 having the origin on its boundary;
 while if $|q-1|<1$, then $D^Y(\rho)$ is the closed exterior of a disc
 having the origin on its boundary.)
Now, for any closed disc $D$ having the origin on its boundary,
it is easy to see that a complex number $y$ satisfies
$yD \subseteq D$ if and only if $y \in [0,1]$.
Back in the $t$-plane, this says that $t \in \scrc_q$.
\qed
\end{proof}

Combining these results, we have:

\proofof{Theorem~\ref{thm_main_AF_sec7}}
%% \comment{{\bf GORDON:  I wrote this proof; please check it!!!} Checked and correct.}
We choose the radii
$\rho^2 = r_1 \leq r_2 \leq \ldots \leq r_{\Lambda-1} \leq \rho$
exactly as in Section~\ref{sec_discs},
which guarantees that $D(r_k) \pllq D(r_\ell) \subseteq D(r_{k+\ell})$.
Then Lemmas~\ref{lemma_CD}, \ref{lemma_Drho_Cq} and \ref{lemma4.3}
guarantee that the sets $S_k$ defined by \reff{def.Sk.AF}
satisfy the conditions (1) and (2) of Theorem~\ref{regions}.
Conditions (3) and (4) of Theorem~\ref{regions}
are verified exactly as in the proof of Proposition~\ref{prop_radii}.
\qed

The fact that Lemma~\ref{lemma_Drho_Cq} gives a
{\em necessary and sufficient}\/ condition
suggests that there is something natural about the
real antiferromagnetic regime $v_e \in [-1,0]$.
On the other hand, our strategy of proof does not really require us
to prove the strong statement \reff{eq.Cq_Drho};
it would suffice to prove the slightly weaker statement \reff{eq.Cq_Drlam-2}
[or perhaps even weaker bounds],
and this might allow a somewhat larger set of weights $v_e$.
In particular, the proof of Theorem~\ref{sokalbound}
given in \cite{MR1827809} works naturally for the
``complex antiferromagnetic regime'' $|1+v_e| \le 1$
(see \cite{Jackson-Procacci-Sokal} for the changes when one goes beyond this),
so it is reasonable to ask whether the bounds in terms of maxmaxflow
can be extended to this case, possibly with a worse constant
(but still growing only linearly in $\Lambda$).
We do not yet know the answer.
As a warm-up, it might be helpful to study the intersection
of the complex antiferromagnetic regime with the real axis,
namely the ``extended real antiferromagnetic regime'' $v_e \in [-2,0]$.

\section{Generalization to non-series-parallel graphs}  \label{sec.Wheatstone}

In this section we show how our constructions can be generalized
to handle graphs that are not series-parallel
but are nevertheless built up by using series and parallel compositions
from a fixed starting set of 2-terminal ``base graphs''.
We begin by stating an abstract theorem on excluding roots,
which generalizes Theorems~\ref{thm_abstract} and \ref{regions}
to the non-series-parallel case (Section~\ref{subsec.genabstract}).
Then we apply this result to prove Theorem~\ref{thm_main_Wheatstone}
(Section~\ref{subsec.Wheatstone}).

\subsection{Generalized abstract theorem on excluding roots}
   \label{subsec.genabstract}

Here is Theorem~\ref{thm_abstract}
generalized to the non-series-parallel case:

\begin{thm}
  \label{thm_genabstract}
Let $q \neq 0$ be a fixed complex number
and let $\Lambda \ge 2$ be a fixed integer.
Let $S_1 \subseteq S_2 \subseteq \cdots \subseteq S_{\Lambda-1}$
be sets in the (finite) complex $v$-plane such that
\begin{itemize}
\item[(1)] $S_k \tvserq{V} S_\ell \subseteq S_{\min(k,\ell)}$
   for all $k,\ell$
\item[(2)] $S_k \tvpll{V} S_\ell\subseteq S_{k+\ell}$
   for $k + \ell \leq \Lambda-1$
\end{itemize}
Now consider any (loopless connected) 2-terminal graph $(G,s,t)$
and any nontrivial decomposition tree for $(G,s,t)$
in which all the proper constituents
have between-terminals flow at most $\Lambda-1$.
Suppose that we equip $G$ with edge weights $\{v_e\}$
such that for every {\em leaf} node $(H,a,b)$ of the decomposition tree,
we have $v_{\rm eff}(H,a,b) \in S_{\lambda_H(a,b)}$.
Then, for {\em every} node $(H,a,b)$ of the decomposition tree
that has between-terminals flow $\lambda_H(a,b) \le \Lambda-1$,
we have $v_{\rm eff}(H,a,b) \in S_{\lambda_H(a,b)}$.

Now assume further that, in addition to (1) and (2),
the following hypotheses hold:
\begin{itemize}
\item[(3)] $-q \notin S_{\Lambda-1}$
\item[(4)] $-q \notin S_k \tvpll{V} S_{\ell}$ for $k+\ell = \Lambda$
\end{itemize}
Then, for any $(G,s,t)$ and $\{v_e\}$ as above,
such that $G$ has maxmaxflow at most $\Lambda$,
we have $Z_G(q,\bv) \not= 0$.
\end{thm}

The proof of Theorem~\ref{thm_genabstract}
is a minor modification of that of Theorem~\ref{thm_abstract}
and is left to the reader.
There is also an obvious translation of Theorem~\ref{thm_genabstract}
to the $t$-plane along the lines of Theorem~\ref{regions},
of which the statement and proof are again left to the reader.

\subsection{Wheatstone bridge}  \label{subsec.Wheatstone}

The {\em Wheatstone bridge}\/ is the 2-terminal graph $\sfW = (W,s,t)$
obtained from $W = K_4- e$ by taking the two vertices of
degree $2$ to be the terminals $s$ and $t$.
Note that although $K_4 -e$ is a series-parallel graph,
$\sfW$ is not a {\em 2-terminal}\/ series-parallel graph
(by Corollary~\ref{cor_serpar_2},
because $W+st = K_4$ is not a series-parallel graph).

Now define the class $\scrw$ of 2-terminal graphs
to be the smallest class that contains both
$K_2$ (with the two vertices as terminals) and $\sfW$
and is closed under series and parallel composition.
%{\bf Can we characterize the class of \emph{graphs}
 % (forgetting the terminals) that arises by this construction???
 % Is it minor-closed?  And if so, what are the excluded minors?}
Figure~\ref{graphsinc} shows some graphs in $\scrw$:
the first is a 2-terminal series-parallel graph,
while the second has used $\sfW=(W,s,t)$ in place of
the ``diamond''  $\sfD = (K_2 \ser K_2) \pll (K_2 \ser K_2)$.

\begin{figure}[t]
\begin{center}
\beginpgfgraphicnamed{wheatstonepic}
\begin{tikzpicture}
\tikzstyle{vertex}=[circle,draw=black,inner sep = 0.75mm]
\tikzstyle{terminal}=[circle,draw=black,fill=gray,inner sep = 0.75mm]

\node[terminal] (v0) at (0,0) {};
\node[vertex] (v1) at (1,-1) {};
\node[vertex] (v2) at (1,1) {};

\node[vertex] (v3) at (2,-1.5) {};
\node[vertex] (v4) at (2,-0.5) {};
\node[vertex] (v5) at (2,0.5) {};
\node[vertex] (v6) at (2,1.5) {};

\node[terminal] (v7) at (3,0) {};
\draw (v0)--(v1);
\draw (v0)--(v2);
\draw (v1)--(v3);
\draw (v1)--(v4);
\draw (v2)--(v5);
\draw (v2)--(v6);

\draw (v3)--(v7);
\draw (v4)--(v7);
\draw (v5)--(v7);
\draw (v6)--(v7);

\pgftransformxshift{5cm}

\node[terminal] (v0) at (0,0) {};
\node[vertex] (v1) at (1,-1) {};
\node[vertex] (v2) at (1,1) {};

\node[vertex] (v3) at (2,-1.5) {};
\node[vertex] (v4) at (2,-0.5) {};
\node[vertex] (v5) at (2,0.5) {};
\node[vertex] (v6) at (2,1.5) {};

\node[terminal] (v7) at (3,0) {};
\draw (v0)--(v1);
\draw (v0)--(v2);
\draw (v1)--(v3);
\draw (v1)--(v4);
\draw (v2)--(v5);
\draw (v2)--(v6);

\draw (v3)--(v7);
\draw (v4)--(v7);
\draw (v5)--(v7);
\draw (v6)--(v7);

\draw (v3)--(v4);
\draw (v5)--(v6);

\end{tikzpicture}
\endpgfgraphicnamed
\end{center}
\caption{
   The graphs $(K_2 \ser \sfD) \pll (K_2 \ser \sfD)$
   and $(K_2 \ser \sfW) \pll (K_2 \ser \sfW)$.
}
\label{graphsinc}
\end{figure}

%% In \cite{MR2047238}, Sokal shows that the multivariate Tutte polynomial of a graph $H$ with an edge $e$ replaced by a 2-terminal graph $(G,s,t)$ is the product of a prefactor and the multivariate Tutte polynomial of $H$ where $e$ is given a suitable effective weight. 
%% Let $\zgst$ denote the sum given by \eqref{multivariatetutte} restricted
%% to the subsets of edges that connect $s$ to $t$ and $\zgsnt$ denote this sum restricted to the subsets of edges that {\em do not} connect
%% $s$ to $t$. If we define
%% \begin{align}
%% A & = q^{-2} \zgsnt \\
%% B & = q^{-1} \zgst 
%% \end{align}
%% then the prefactor is given by $A$ and the effective weight of $e$
%% is given by $B/A$ (Sokal \cite{MR2047238}).

For the Wheatstone bridge, simple calculations
[e.g.\ using \reff{eq2.G}/\reff{eq2.Gst}]
show that if $v_f = -1$ for every edge,
then the partial Tutte polynomials \reff{def.AGst}/\reff{def.BGst}
are given by
\begin{subeqnarray}
   A_{W,s,t}  & = &  (q-2)(q-3)  \\[1mm]
   B_{W,s,t}  & = &  2(q-2)
\end{subeqnarray}
and hence
\begin{equation}
v_{\rm eff}(W,s,t)  \;=\;  \frac{2}{q-3}  \;.
\end{equation}
Expressed in terms of transmissivities, this yields
\begin{equation}
  \label{wheat-teff}
t_{\rm eff}(W,s,t)  \;\equiv\; {v_{{\rm eff}} \over q + v_{{\rm eff}}}
   \;=\; \frac{2}{(q-1)(q-2)}  \;.
\end{equation}
The maximum flow between the terminals of the Wheatstone bridge is equal to $2$.
Therefore, as far as the chromatic roots of graphs in $\scrw$ are concerned,
the Wheatstone bridge just behaves as a sort of ``super-edge''
with capacity (in the flow-carrying sense) equal to 2
and effective transmissivity given by \eqref{wheat-teff}:
that is the upshot of Theorem~\ref{thm_genabstract}.

Now suppose that $S_1 \subseteq S_2 \subseteq \cdots \subseteq S_{\Lambda-1}$
is a set of regions in the complex $t$-plane
certifying (via Theorem~\ref{regions}) that a particular value of $q$
is not the chromatic root of any series-parallel graph
of maxmaxflow at most $\Lambda$.
Then, by Theorem~\ref{thm_genabstract},
the {\em same}\/ set of regions will suffice for graphs in $\scrw$
of maxmaxflow at most $\Lambda$, provided only that $2/[(q-1)(q-2)] \in S_2$.

So, under the same hypothesis $|q-1| \ge 1/\rho^\star_\Lambda$
as in Theorem~\ref{thm_main_sec5},
let us again choose ``point+disk'' regions \reff{regionsequation}
with radii $r_k$ given by \reff{eq.choice2}.
We then have
\be
   r_2 \;=\;  \rho \, {X^2 - 1 \over 1 - \rho X^2}
   \qquad\hbox{where}\quad
   \rho \,=\, {1 \over |q-1|} \;\;\hbox{and}\;\;
   X \,=\, \biggl( {2 \over 1+\rho} \biggr)^{1/(\Lambda-1)}
   \;.
 \label{eq.maxs2}
\ee
Therefore, these regions suffice to show that $q$ is not a chromatic root
of any graph in $\scrw$ whenever we have
\be
\frac{2}{|q-2|} \;\le\; {X^2 - 1 \over 1 - \rho X^2}
\label{eq.bounds2}
\ee
in addition to the hypothesis $|q-1| \ge 1/\rho^\star_\Lambda$.
We have therefore proven:

\begin{thm}
    \label{thm_main_Wheatstone_sec8}
Fix an integer $\Lambda \ge 3$,
and let $\sfG = (G,s,t)$ be a 2-terminal graph in the class $\scrw$
such that $G$ has maxmaxflow at most $\Lambda$.
Then the chromatic polynomial $P_G(q)$ is nonvanishing
whenever
\be
   |q-1| \:\ge\: 1/\rho^\star_\Lambda
   \quad\hbox{and}\quad
   |q-2| \:\ge\: {2 (1 - \rho X^2) \over X^2 - 1}
   \;,
\ee
where $\rho_\Lambda^\star$ is defined by \reff{eq.rhostarlambda}
and $\rho$ and $X$ are defined by \reff{eq.maxs2}.
% \be
%    \rho \:=\: {1 \over |q-1|}
%    \quad\hbox{and}\quad
%    X \:=\: \biggl( {2 \over 1+\rho} \biggr)^{1/(\Lambda-1)}
%    \;.
% \ee
\end{thm}

When $\Lambda=3$, the condition \reff{eq.bounds2}
becomes particularly simple as it reduces to 
\be
   |q-2| \;\ge\; 2  \;.
\ee
The disk $|q-2| < 2$ extends only slightly beyond
the disk $|q-1| < 1/\rho_3^\star$,
with the greatest protrusion $4-(1+1/\rho_3^\star) \approx 0.34103\ldots$
occurring on the positive real axis.
Thus, the region guaranteed to contain the chromatic roots of the graphs
in $\scrw$ of maxmaxflow 3,
given by the union of the discs $|q-1| < 1/\rho_3^\star$ and $|q-2| < 2$
(the left-hand picture of Figure~\ref{fig_wheatstone}),
is only slightly larger than the region $|q-1| < 1/\rho_3^\star$
guaranteed to contain the chromatic roots of series-parallel graphs
of maxmaxflow 3.

For $\Lambda>3$, the right-hand side of \eqref{eq.bounds2}
is not independent of $\rho$,
and so the corresponding region is not quite a circular disk
(as it depends on the phase of $q$), but rather a slightly squashed disk.
Nevertheless, the corresponding region always extends slightly past
the region $|q-1| < 1/\rho_\Lambda^\star$,
with maximum protrusion again on the positive real axis
(see the right-hand picture of Figure~\ref{fig_wheatstone} for $\Lambda=4$).

In order to obtain a simple {\em sufficient}\/ condition
depending only on $|q-1|$, we can use the trivial bound
$|q-2| \ge |q-1| - 1 = \rho^{-1} - 1$.
After some simple algebra
we find that a {\em sufficient}\/ condition on $\rho$
for \eqref{eq.bounds2} to be satisfied is that
\be
\left( \frac{2}{1+\rho}\right) ^{2/{(\Lambda-1)}}
  \;\ge\;  \frac{1+\rho}{1-\rho+2\rho^2}
   \;,
   \label{rholambdaforwheat}
\ee
or equivalently that
\be
   (1+\rho)^{\Lambda+1}  \;\le\;  4 (1-\rho+2\rho^2)^{\Lambda-1}
   \;.
\ee
This condition is handled by
the following analogue of Lemma~\ref{lemma_sharp_rkbound_NEW}:

\begin{lem}
   \label{lemma_rhodoublestar}
For $\rho \in (0,1)$ and real $\Lambda > 2$, the following are equivalent:
\begin{itemize}
   \item[(a)] $(1+\rho)^{\Lambda+1}  \:\le\:  4 (1-\rho+2\rho^2)^{\Lambda-1}$.
   \item[(b)] $\rho \le \rho^{\star\star}_\Lambda$,
where $\rho^{\star\star}_\Lambda$ is the unique solution of
\be
   (1+\rho)^{\Lambda+1}  \;=\;  4 (1-\rho+2\rho^2)^{\Lambda-1}
 \label{eq.rhodoublestarlambda}
\ee
in the interval $(0,1)$.
\end{itemize}
\end{lem}

\noindent
Deferring temporarily the proof of this Lemma,
let us observe that
from \reff{eq.rhostarlambda} it follows easily that
\be
\left( \frac{2}{1+\rho_\Lambda^\star} \right)^{1/{(\Lambda-1)}} 
   \;=\;
\frac{1+\rho_\Lambda^\star}{1+{\rho_\Lambda^\star}^2}
   \;<\;
\frac{1+\rho_\Lambda^\star}{1-\rho_\Lambda^\star + 2{\rho_\Lambda^\star}^2}
  \;,
\ee
so that the condition \reff{rholambdaforwheat}
is {\em false}\/ when $\rho = \rho_\Lambda^\star$,
or in other words we have $\rho_\Lambda^{\star\star} < \rho_\Lambda^\star$
whenever $\Lambda > 2$ (see Table~\ref{table_rhostarlambda}).
We have therefore proven
(subject to the proof of Lemma~\ref{lemma_rhodoublestar}):

\begin{cor}
    \label{cor_main_Wheatstone_sec8}
Fix an integer $\Lambda \ge 3$,
and let $\sfG = (G,s,t)$ be a 2-terminal graph in the class $\scrw$
such that $G$ has maxmaxflow at most $\Lambda$.
Then the chromatic polynomial $P_G(q)$ is nonvanishing
whenever
\be
   |q-1| \;\ge\; 1/\rho^{\star\star}_\Lambda
   \;,
\ee
where $\rho_\Lambda^{\star\star}$ is the unique solution of
\be
   (1+\rho)^{\Lambda+1}  \;=\;  4 (1-\rho+2\rho^2)^{\Lambda-1}
\ee
in the interval $(0,1)$.
\end{cor}

Since the bound $|q-2| \ge |q-1| - 1$ holds as equality
when $q$ is real and $q > 2$, we see that the condition \reff{rholambdaforwheat}
is also {\em necessary}\/ for \eqref{eq.bounds2} when $q$ is real and positive;
or in other words,
the circle $|q-1| = 1/\rho_\Lambda^{\star\star}$
coincides with the boundary of the region \reff{eq.bounds2}
on the positive real axis, but lies outside of it elsewhere. Figure~\ref{fig_wheatstone} shows how
these regions compare for $\Lambda=3, 4$.

\begin{figure}[t]
\begin{center}
\includegraphics[width=0.9\textwidth]{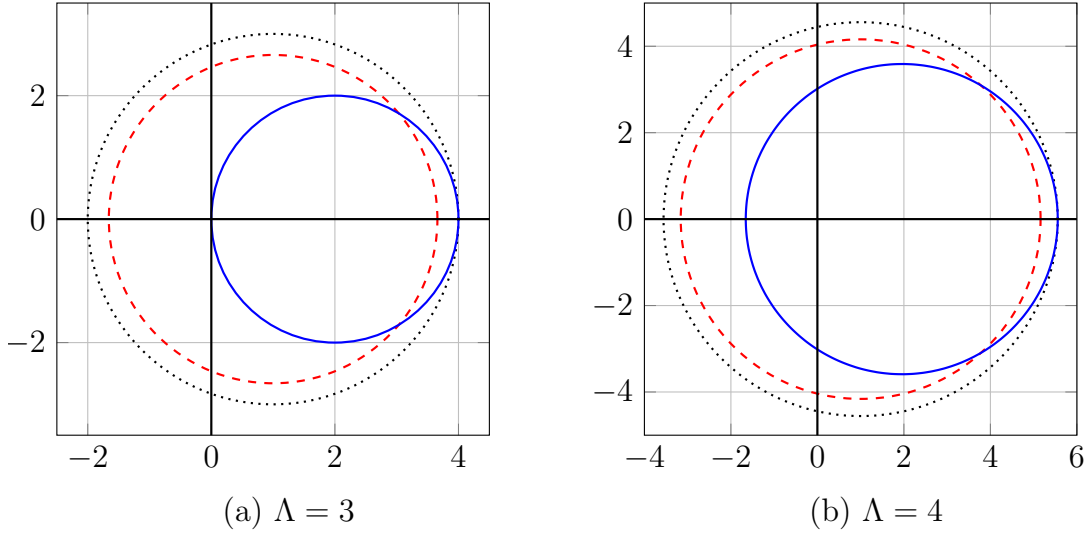}
\end{center}
\caption{
   %% Circles $|q-1|=1/\rho_\Lambda^*, 1/\rho_\Lambda^{**}$
   %% and the region defined by \eqref{eq.bounds2} for $\Lambda=3,4$.
   %% {\bf This caption needs to be clarified!!!}
   Circles $|q-1| = 1/\rho_\Lambda^*$ (red dashed)
   and $|q-1| = 1/\rho_\Lambda^{**}$ (black dotted)
   and the boundary of the region \eqref{eq.bounds2} [blue solid curve],
   for (a) $\Lambda=3$ and (b) $\Lambda=4$.
}
\label{fig_wheatstone}
\end{figure}

\begin{table}[t]
\begin{center}
\begin{tabular}{|r||l|r||l|r|}
\hline
   \multicolumn{1}{|c||}{$\Lambda$} &
   \multicolumn{1}{|c|}{$\rho^\star_\Lambda$} &
   \multicolumn{1}{|c||}{$\rho^{\star\star}_\Lambda$} &
      \multicolumn{1}{|c|}{$1/\rho^\star_\Lambda$} &
   \multicolumn{1}{|c|}{$1/\rho^{\star\star}_\Lambda$} \\
\hline
2  &    1          &   1\phantom{.000000}        &1 & 1\phantom{.000000} \\
3  &    0.376086 &    0.333333    & 2.658967   & 3\phantom{.000000} \\
4  &    0.240380 &   0.219471 &4.160076    & 4.556417 \\
5  &    0.177591 &   0.165204   &5.630929    & 6.053134\\
6  &    0.141038   &  0.132841   &7.090297   & 7.527812\\
7  &     0.117041  &  0.111213   &8.544040  & 8.991750   \\
8  &     0.100054  &  0.095697 &9.994599    & 10.449611  \\
9   &    0.087388   &  0.084008  &11.443181 &11.903688  \\
10   &   0.077577  & 0.074877&12.890449     &13.355246  \\
\hline
\end{tabular}
\end{center}
\vspace*{-3mm}
\caption{
Values of 
   $\rho^\star_\Lambda$, $\rho^{\star\star}_\Lambda$,
   $1/\rho^\star_\Lambda$, $1/\rho^{\star\star}_\Lambda$
   for  $2 \le \Lambda \le 10$.
%   {\bf Can we also show
%     $1/\rho^\star_\Lambda$ and $1/\rho^{\star\star}_\Lambda$
%     in two additional columns????}
}
\label{table_rhostarlambda}
\end{table}

\proofof{Lemma~\ref{lemma_rhodoublestar}}
%{\bf GORDON:  Please check this proof!!!!!}
Consider the function
\be
   f_\Lambda(\rho)  \;=\;
   (\Lambda+1) \log(1+\rho) \,-\, (\Lambda-1) \log(1-\rho+2\rho^2)
   \;.
\ee
Its first two derivatives are
\begin{subeqnarray}
   f'_\Lambda(\rho)
   & = &
   {\Lambda+1 \over 1+\rho} \,-\,
       {(\Lambda-1)(4\rho-1) \over 1-\rho+2\rho^2}  \\[3mm]
   f''_\Lambda(\rho)
   & = &
   - \, {(6+4\rho+18\rho^2-24\rho^3+4\rho^4) \,+\,
            (\Lambda-2) (4+8\rho+8\rho^2-16\rho^3-4\rho^4)
         \over (1+\rho)^2 (1-\rho+2\rho^2)^2
        }
     \nonumber \\
\end{subeqnarray}
For $0 \le \rho \le 1$ we manifestly have
$6+4\rho+18\rho^2-24\rho^3+4\rho^4 \ge 4\rho-4\rho^3 \ge 0$
and
$4+8\rho+8\rho^2-16\rho^3-4\rho^4 \ge 4 - 4\rho^4 \ge 0$,
with strict inequality when $0 < \rho < 1$;
hence $f_\Lambda$ is strictly concave on $[0,1]$ whenever $\Lambda \ge 2$.
We have $f_\Lambda(0) = 0$, $f'_\Lambda(0) = 2\Lambda > 0$,
$f_\Lambda(1) = \log 4$ and $f'_\Lambda(1) = -(\Lambda-2)$.
Therefore, for $\Lambda > 2$,
there is a unique $\rho^{\star\star}_\Lambda \in (0,1)$
satisfying $f_\Lambda(\rho^{\star\star}_\Lambda) = \log 4$;
and for $\rho \in [0,1)$ we have $f_\Lambda(\rho) \le \log 4$
if and only if $\rho \le \rho^{\star\star}_\Lambda$.
This proves the equivalence of (a) and (b) for all real $\Lambda > 2$.
\qed

Finally, let us deduce Theorem~\ref{thm_main_Wheatstone}
as an immediate consequence of Corollary~\ref{cor_main_Wheatstone_sec8},
by proving that 
$\rho^{\star\star}_\Lambda > (\log 2)/(\Lambda- \log 2)$
for all integers $\Lambda \ge 3$,
or equivalently (in view of Lemma~\ref{lemma_rhodoublestar}) that
$(1+\rho)^{\Lambda+1} < 4(1-\rho+2\rho^2)^{\Lambda-1}$
when $\rho = (\log 2)/(\Lambda-\log 2)$.
We shall actually prove this for all {\em real}\/
$\Lambda > 2 \log 2 \approx 1.386294$.
Taking logarithms and parametrizing by $\rho \in (0,1)$,
we see that this is equivalent to the following claim:

\begin{lem}
   \label{lemma_rkbound_log2_bis}
The function
\begin{subeqnarray}
   g(\rho)
   & = &
   \rho \left[ 2 \log 2 \:+\:
               \biggl( {\log 2 \over \rho} \,+\, \log 2 \,-\, 1
               \biggr) \, \log(1-\rho+2 \rho^2)
           \:-\: \biggl( {\log 2 \over \rho} +  \log 2 + 1\biggr)
               \, \log(1+\rho)
        \right]
      \nonumber \\ \\
   & = &
  2 \rho \log \biggl( {2 \over 1+\rho} \biggr) \,-\, (\log 2) \log \biggl( {1+\rho \over 1-\rho+2\rho^2} \biggr) - (\log 2 - 1) \rho \log \biggl( {1+\rho \over 1-\rho+2\rho^2} \biggr)
\end{subeqnarray}
is strictly positive for $0 < \rho < 1$.
\end{lem}

\begin{proof}
The second derivative of $g$ is given by
$g''(\rho) = h(\rho) \times \rho / (1+\rho^2 +2\rho^3)^2$ where
\be
   h(\rho) \;=\;
   - (12 - 4 \log 2) \rho^4  \,-\, 12 \rho^3
     \,-\, (2 + 8 \log 2) \rho^2  \,-\, (24 - 16 \log 2)\rho
     \,+\, (20 \log 2 - 6)
   \;.
\ee
All the coefficients of $h(\rho)$ are strictly negative
except for the last (constant) term,
so we have $h'(\rho) < 0$ for all $\rho \geq 0$.
Since $h(0) = 20 \log 2 - 6 > 0$ and $h(1) = 32\log 2 - 56 < 0$
and $h$ is strictly decreasing for $\rho \ge 0$,
it follows that $h(\rho)$ has exactly one positive real root $\rho^*$
and that it lies between 0 and 1
(by computer $\rho^* \approx 0.417655$).
Therefore $g$ is strictly convex on $[0,\rho^*]$
and strictly concave on $[\rho^*,\infty)$.
Since $g(0) = g'(0) = 0$, we have $g(\rho) > 0$ for $\rho \in (0,\rho^*]$.
Moreover, since $g(\rho^*) > 0$ and $g(1) = 0$
and $g$ is strictly concave on $[\rho^*,1]$,
we have $g(\rho) > 0$ for $\rho \in [\rho^*,1)$.
Hence $g(\rho) > 0$ for all $\rho \in (0,1)$, as claimed.
\qed
\end{proof}

\medskip

{\bf Remark.}
A straightforward calculation shows that the
large-$\Lambda$ asymptotic behavior of $\rho^{\star\star}_\Lambda$ is given by
\be
   \rho^{\star\star}_\Lambda
   \;=\;
   (\log 2)
   \Biggl[ {1 \over \Lambda-1} \,+\, {\log 2 - 1 \over (\Lambda-1)^2}
              \,+\, {16 \log^2 2 - 15\log 2 + 6 \over 6 (\Lambda-1)^3}
              \,+\, \ldots
   \Biggr]
\ee
and hence
\be
   {1 \over \rho^{\star\star}_\Lambda}
   \;=\;
   {\Lambda-1 \over \log 2} \,-\, {\log 2 - 1 \over  \log 2}
              \,-\, {10 \log 2 - 3 \over 6 (\Lambda-1)}
              \,+\, \ldots
   \;\;.
\ee
So the inequality
$\rho^{\star\star}_\Lambda > (\log 2)/(\Lambda - \log 2)$
captures the first two terms of the large-$\Lambda$ asymptotic behavior.
\qed

%\bigskip
%\bigskip
%
%{\bf To prove the analogous theorem for the real antiferromagnetic regime,
%   we would need to handle $t_{\rm eff}(W,s,t)$ when all five edge weights
%   are in $[-1,0]$, and show that it lies in $S_2$ whenever the case
%   $v_e = -1$ does.  But we should first check numerically whether
%   this is true!!!!!}

%When $\Lambda = 3$ and $|q-1| \geq 2 \log 2$, we can take
%\begin{equation}
%S_1 = \{1/(1-q)\} \cup D(\rho^2) \qquad S_2 = D(\rho)
%\end{equation}

%\begin{equation}
%v_{\rm eff}  = \frac{q Z_{G\bullet xy} (q,v) - Z_G (q,v)} {Z_G(q,v)- Z_{G\bullet xy}(q,v)}
%\end{equation}

%\section{Extras}

%\begin{equation}
%(t_1+d_1) \pllq (t_2+d_2)   -   t_1 \pllq t_2
%\end{equation}
%\begin{equation}
%\frac
%{(1-t_2)(1+(q-1)t_2) d_1 + (1-t_1)(1+(q-1)t_1) d_2 + ((q-2) - (q-1)(t_1+t_2)) d_1 d_2} 
%{
%[1 + (q-1)t_1 t_2] [1 + (q-1)(t_1+d_1)(t_2+d_2)]
%}
%\end{equation}

%

%\bigskip
%\bigskip
%
%{\bf GORDON:  Consider the class of graphs obtained from $K_n$ ($n$ fixed)
%   by 2-sums and taking of minors.  This is a minor-closed class,
%   and I believe the excluded minors are all minimally 3-connected graphs
%   on $n+1$ vertices (I believe that's what Alex Scott told me and
%   James Oxley confirmed).  For $n=4$ is the excluded minor just $W_4$
%   (the wheel with four spokes)?
%   And is the same as the class obtained by decomposition trees from
%   leaves that are $K_n - e$ (as a 2-terminal graph with the terminals
%   across the deleted edge) plus suitable ``minors'' thereof?????}

% \section{Discussion}
% 
% {\bf Discuss possible extensions, open problems, etc. What to say?????}

\appendix
\section{Parallel and series connection on the Riemann sphere}
    \label{app_riemann}

In this appendix we shall define parallel and series connection
for edge weights $\{v_e\}$ (or $\{y_e\}$ or $\{t_e\}$)
lying in the Riemann sphere $\Cbar = \C \cup \{\infty\}$;
in our opinion this is the most natural way to view these maps.
We shall always treat $\Cbar$ as a one-dimensional complex manifold
equipped with its usual holomorphic structure.

We begin with some general remarks concerning rational functions
of several complex variables.

At an {\em algebraic}\/ level, there is no difficulty
in defining the field $\K(z_1,\ldots,z_n)$ of rational functions
in an arbitrary number of indeterminates $z_1,\ldots,z_n$
over an arbitrary field $\K$.
%% (see e.g.\ \cite{?????}).
The elements of this field are simply equivalence classes
of ratios $P/Q$ where $P,Q \in \K[z_1,\ldots,z_n]$
are polynomials with $Q \not\equiv 0$,
and the field operations are defined in the obvious way.
%% There is no essential difference between the cases $n=1$ and $n>1$.
%% {\bf Is this so????  Are there \emph{deeper} algebraic differences?????}

However, when $\K=\C$ and we wish to consider rational functions
from an {\em analytic}\/ point of view, there is a fundamental difference
between the cases $n=1$ and $n>1$.
% As is well known, a rational function in one complex variable
% defines unambiguously a holomorphic map of the Riemann sphere $\Cbar$
% into itself.
% On the other hand, a nonconstant rational function in $n>1$ complex variables
% can {\em never}\/ define a holomorphic (or even continuous)
% map of $\Cbar^n$ into $\Cbar$
% {\bf Not quite true --- there are trivial cases that have to be excluded,
%    like $R$ being a function of only one variable --- can we make this
%    precise????}
This is a consequence of the following theorem
\cite[Theorem 1.3.2]{Rudin_69}:
Let $P,Q$ be relatively prime polynomials in $n$ complex variables,
and let $z_0 \in \C^n$ be such that $P(z_0) = Q(z_0) = 0$.
(Of course, this can happen only when $n > 1$.)
Then in every neighborhood $U \ni z_0$,
the function $P/Q$ takes unambiguously all possible values in $\Cbar$
(in addition to also taking the undefined value $0/0$).
Therefore, if the zero sets $\scrz(P)$ and $\scrz(Q)$
have a nonempty intersection,
the rational function $P/Q$ cannot even be defined as a {\em continuous}\/
function (much less a holomorphic one) in a neighborhood of the
intersection point.
On the other hand, if we cut out the ``bad points'' $\scrz(P) \cap \scrz(Q)$,
then $P/Q$ is a well-defined holomorphic function from
$\C^n \setminus (\scrz(P) \cap \scrz(Q))$ into $\Cbar$.
And with a little more work we can also allow $\infty$
as a possible value for $z_1,\ldots,z_n$
(expanding the ``bad set'' as needed).

In our case we want to make sense of the parallel and series maps
(in $n=2$ variables) defined in \reff{eq_vpar}--\reff{eq_yser}.
Note first that the maps \reff{def_te}/\reff{def_ye}
between the $v$-, $t$- and $y$-variables are
invertible M\"obius transformations,
hence biholomorphic maps of the Riemann sphere $\Cbar$ onto itself,
whenever $q \neq 0,\infty$ (as we shall assume henceforth).
We can therefore define the parallel connection $\pll$ as follows:
Map first into the $y$ variables;  use the definition
$y_e \pll^Y y_f = y_e y_f$;  then map back.
Now, the operation of multiplication, $(y_1,y_2) \mapsto y_1 y_2$,
is unambiguously defined for $y_1$ and $y_2$ in the Riemann sphere,
{\em with two exceptions}\/:  $0 \cdot \infty$ and $\infty \cdot 0$
are ill-defined.  Deleting these two ``bad points'',
we see that multiplication is a well-defined holomorphic map from
$\Cbar^2 \setminus \{(0,\infty),\, (\infty,0)\}$ into $\Cbar$.
It follows that
\begin{itemize}
   \item the parallel connection $\pll^Y$ is well-defined for
$(y_e,y_f) \in \Cbar^2 \setminus \{(0,\infty),\, (\infty,0)\}$,
   \item the parallel connection $\pll^V$ is well-defined for
$(v_e,v_f) \in \Cbar^2 \setminus \{(-1,\infty),\, (\infty,-1)\}$, and
   \item the parallel connection $\tvpllq{T}$ is well-defined for
$(t_e,t_f) \in \Cbar^2 \setminus \{(1/(1-q),1),\, (1,1/(1-q))\}$.
\end{itemize}
Only at the two ``bad points'' do we declare that the
parallel connection takes the value ``undefined''.

Likewise, we can define the series connection $\ser$ as follows:
Map first into the $t$ variables;  use the definition
$t_e \ser^Y t_f = t_e t_f$;  then map back.
It follows that
\begin{itemize}
   \item the series connection $\ser^T$ is well-defined for
$(t_e,t_f) \in \Cbar^2 \setminus \{(0,\infty),\, (\infty,0)\}$,
   \item the series connection $\tvserq{V}$ is well-defined for
$(v_e,v_f) \in \Cbar^2 \setminus \{(0,-q),\, (-q,0)\}$, and
   \item the series connection $\tvserq{Y}$ is well-defined for
$(y_e,y_f) \in \Cbar^2 \setminus \{(1,1-q),\, (1-q,1)\}$.
\end{itemize}
At the two ``bad points'' we assign once again the value ``undefined''.

Finally, we declare that if one or both of the inputs to a
$\pll^\sharp$ or $\ser^\sharp$ operation ($\sharp = V$, $T$ or $Y$)
is ``undefined'', then the output is also ``undefined''.
The operations $\pll^\sharp$ and $\ser^\sharp$
then become well-defined maps $\Ctilde \times \Ctilde \to \Ctilde$,
where $\Ctilde = \Cbar \cup \{\hbox{undefined}\}
               = \C \cup \{\infty,\hbox{undefined}\}$,
and moreover these operations for different $\sharp$
intertwine correctly with the M\"obius transformations
that map one set of variables ($v$, $t$ or $y$) to another.

\section{Chromatic roots of leaf-joined trees}   \label{app_tree}

For integers $r \ge 2$ and $n \ge 1$,
the {\em leaf-joined tree of branching factor $r$ and depth $n$}\/
is defined to be the graph $G_n^r$
obtained by taking a complete $r$-ary rooted tree of height $n$
and then identifying all the leaves into a single vertex.
Our goal in this appendix is to study, for fixed $r$,
the accumulation points as $n \to\infty$
of the chromatic roots of the family $\{G_n^r\}_{n \ge 1}$.
In particular, we shall prove the following:

\begin{thm}[= Theorem~\ref{thm.leaf-joined}]
  \label{thm.leaf-joined.APP}
For fixed $r \ge 2$, every point of the circle $|q-1| = r$
is a limit point of chromatic roots
for the family $\{G_n^r\}_{n \ge 1}$
of leaf-joined trees of branching factor $r$.
[More precisely, for every $q_0$ satisfying $|q_0-1| = r$
 and every $\epsilon > 0$, there exists $n_0 = n_0(q_0,\epsilon)$
 such that for all $n \ge n_0$ the graph $G_n^r$ has a chromatic root $q$
 lying in the disc $|q-q_0| < \epsilon$.]
%% {\bf Can we prove the more general theorem for $Z_G(q,v)$
%%    with \emph{any} fixed $v$, not just $v=-1$????}
\end{thm}

To lighten the notation,
we shall henceforth fix $r \ge 2$ and write $G_n$ in place of $G_n^r$.
We shall regard $G_n$ as a 2-terminal graph in which
the terminals are the root and the identified-leaves vertex.
It is easy to see that $G_n$ is in fact 2-terminal series-parallel,
as it can be defined recursively as follows:
\begin{subeqnarray}
%% G_0  & = &  K_1  \\[1mm]
G_1  & = &  K_2^{(r)} \\[1mm]
G_{n+1} & = & (K_2 \ser G_n)^{\pll r}
  \slabel{eq.app_tree.def_Gn.b}
  \label{eq.app_tree.def_Gn}
\end{subeqnarray}
where $K_2^{(r)}$ is the graph with two vertices
connected by $r$ parallel edges,
and $G^{\pll r}$ denotes the parallel connection of $r$ copies of $G$.

To every edge of $G_n$ we assign the {\em same}\/ weight $v_\sharp$
and we study the (bivariate) Tutte polynomial $Z_{G_n}(q,v_\sharp)$.
In particular, by taking $v_\sharp = -1$ we can study the chromatic polynomial
$P_{G_n}(q)$.
We shall prove Theorem~\ref{thm.leaf-joined.APP}
by tracking the evolution of the ``effective coupling'' $v_{\rm eff}(G_n)$
according to \reff{eq.app_tree.def_Gn.b}.
In doing this we shall adopt the approach set forth in
Appendix~\ref{app_riemann}, where $v_{\rm eff}$ is considered to lie
in the Riemann sphere $\Cbar = \C \cup \{\infty\}$.
Please note that the parallel and series connections arising in
\reff{eq.app_tree.def_Gn.b}
are always well-defined in the sense of Appendix~\ref{app_riemann}:
the series connections avoid the ``bad points''
$(v_1,v_2) = (0,-q)$ and $(-q,0)$ because
we always have $v_1 = v_\sharp$ (corresponding to the graph $K_2$)
{\em and we shall assume that $q \notin \{0,-v_\sharp\}$}\/;
while the parallel connections avoid the ``bad points''
$(v_1,v_2) = (-1,\infty)$ and $(\infty,-1)$ because
we perform repeated parallel connections
of the {\em same}\/ graph $K_2 \ser G_n$
and hence at every stage $v_1$ and $v_2$
are either both finite or both infinite.

The basic idea is now that $Z_{G_n}(q,v_\sharp)$ equals
$q [q + v_{\rm eff}(G_n)]$ times some prefactors;
so if we temporarily put aside the problem of the prefactors,
we can conclude that
$Z_{G_n}(q,v_\sharp) = 0$ if and only if $v_{\rm eff}(G_n) = -q$
(provided that $q \neq 0$, as we shall assume henceforth).
To be more precise (and to handle the problem of the prefactors),
let us use Proposition~\ref{prop.serpar.AB} to compute the polynomials
$A_n(q,v_\sharp) \equiv A_{G_n,s,t}$ and $B_n(q,v_\sharp) \equiv B_{G_n,s,t}$
that were defined in \reff{def.AGst}/\reff{def.BGst};
from these we can obtain
$Z_{G_n}(q,v_\sharp) = q^2 A_n(q,v_\sharp) + q B_n(q,v_\sharp)$.
Defining the intermediate graphs $G'_n = K_2 \ser G_n$
and the corresponding quantities
$A'_n \equiv A_{G'_n,s,t}$ and $B'_n \equiv B_{G'_n,s,t}$,
we have from Proposition~\ref{prop.serpar.AB}(b)
\begin{subeqnarray}
   A'_n & = & (q+v_\sharp) A_n + B_n  \\[1mm]
   B'_n & = & v_\sharp B_n
\end{subeqnarray}
and thence from Proposition~\ref{prop.serpar.AB}(a)
\begin{subeqnarray}
   A_{n+1} & = & (A'_n)^r    \nonumber \\
           & = & [(q+v_\sharp) A_n + B_n]^r        \\[3mm]
   B_{n+1} & = & (A'_n + B'_n)^r \,-\, (A'_n)^r   \nonumber \\
           & = & [(q+v_\sharp) A_n + (1+v_\sharp) B_n]^r \,-\,
                       [(q+v_\sharp) A_n + B_n]^r
 \label{eq.recursion.AnBn}
\end{subeqnarray}
%% where $A_n$ and $B_n$ are polynomials
%% in the indeterminates $q$ and $v_\sharp$.
The initial condition for this recursion is
\be
   A_1 \:=\: 1, \qquad  B_1 \:=\: (1+v_\sharp)^r \,-\, 1
 \label{eq.recursion.AnBn.initial}
\ee
or equivalently
\be
   A_0 \:=\: 0, \qquad  B_0 \:=\: 1
   \;.
\ee
Note now that the polynomials $A_n$ and $B_n$ have no common zeros in
$\C^2 \setminus \{(q,v_\sharp) \colon\: q+v_\sharp = 0 \}$:
this follows by induction from \reff{eq.recursion.AnBn},
since $A_{n+1} = B_{n+1} = 0$ and $q+v_\sharp \neq 0$ imply $A_n = B_n = 0$.
[This observation is just a rephrasing of the previously-observed fact that
the series and parallel connections are here well-defined
whenever $q \notin \{0,-v_\sharp\}$.]
Therefore, provided that $q \notin \{0,-v_\sharp\}$,
we have $Z_{G_n}(q,v_\sharp) = 0$
if and only if $A_n(q,v_\sharp) \neq 0$
and $B_n(q,v_\sharp) / A_n(q,v_\sharp) = -q$.

But $B_n/A_n$ is precisely what we have called $v_{\rm eff}(G_n)$
[cf. \reff{def.veff}].
Writing $v_n = v_{\rm eff}(G_n)$ to lighten the notation,
we obtain from \reff{eq.recursion.AnBn}
[or equivalently from \reff{eq.par.veff}/\reff{eq.ser.veff}]
the recursion
\be
   v_{n+1}
   \;=\;
   (v_\sharp \,\tvserq{V} v_n)^{\pll^V r}
   \;=\;
   \left( 1 \,+\, {v_\sharp v_n  \over  q + v_\sharp + v_n} \right) ^{\! r}
   \:-\: 1
\ee
with initial condition
\be
   v_1  \;=\;  (1 + v_\sharp)^r \,-\, 1
\ee
or
\be
   v_0  \;=\;  \infty
   \;.
\ee
And provided that $q \notin \{0,-v_\sharp\}$,
we have proven that $Z_{G_n}(q,v_\sharp) = 0$ if and only if $v_n = -q$.

% A straightforward computation using Proposition~\ref{prop.serpar.AB}
% shows that
% %% [using the shorthand $v_n = v_{\rm eff}(G_n)$ to lighten the notation]
% \be
%    Z_{G_n}(q,v_\sharp)
%    \;=\;
%    q [q + v_{\rm eff}(G_n)]
%    \prod_{j=1}^{n-1} [q + v_\sharp + v_{\rm eff}(G_j)]^{r^{n-j}}
% \ee
% whenever all the quantities $v_{\rm eff}(G_j)$ are finite.
% {\bf OOPS:  We \emph{do} have to be careful about the possibility
%   of hitting $v_{\rm eff}(G_j) = \infty$ or $-(q+v_\sharp)$
%   --- i.e.\ $y_n = \infty$ or $2-q-y_\sharp$ ---
%   in the course of the iteration.
%   Indeed, I think this \emph{will} happen for values of $q$
%   very near those where our roots are found.
%   So how do we handle this?????}
% %% {\bf  Is this true???
% %%   How to handle the subtleties having to do with prefactors?????
% %%    Maybe we get ``if'' but not ``only if''?????}

%% {}From the recursion \reff{eq.app_tree.def_Gn} defining the graphs $\{G_n\}$,
%% we can immediately read off the recursion for $v_{\rm eff}(G_n)$
%% by using the series and parallel laws discussed in
%% Sections~\ref{subsec.serpar} and \ref{subsec.tutte.serpar.2term}.
Since the final operation in \reff{eq.app_tree.def_Gn.b}
is parallel composition,
it is actually more convenient to use the $y$-variables,
i.e.\ $y_\sharp = 1 + v_\sharp$ and
$y_n = y_{\rm eff}(G_n) = 1 + v_{\rm eff}(G_n)$.
We then have
\begin{subeqnarray}
   y_0  & = &  \infty  \\[1mm]
   y_1  & = &  y_\sharp^r  \\
   y_{n+1}  & = &  (y_\sharp \,\ser_q^Y y_n)^{\pll^Y r}  \;=\;
      \left( {q-1 + y_\sharp y_n  \over  q-2 + y_\sharp + y_n} \right) ^{\! r}
\end{subeqnarray}
or in other words $y_{n+1} = R_q(y_n)$ where
\be
   R_q(y)  \;=\;
   \left( {q-1 + y_\sharp y  \over  q-2 + y_\sharp + y} \right) ^{\! r}
   \;.
 \label{def.Rqy}
\ee
We have $Z_{G_n}(q, v_\sharp) = 0$ if and only if $y_n = 1-q$.
We summarize the foregoing discussion as follows:

\begin{lemma}
   \label{lemma.iteration}
For $v_\sharp \in \C$, $q \in \C \setminus \{0,-v_\sharp\}$
and $n \ge 1$, we have
$$
   Z_{G_n}(q, v_\sharp) \,=\, 0
     \quad\text{if and only if} \quad
   R_q^n(\infty) \,=\, 1-q
   \;,
$$
where $R_q$ is the map \reff{def.Rqy} with $y_\sharp = 1 + v_\sharp$.
\end{lemma}

We are thus interested in the iteration of the rational function $R_q$,
considered as a map from the Riemann sphere
$\Cbar = \C \cup \{\infty\}$ to itself,
and so we are naturally led to the theory of holomorphic dynamics
\cite{Beardon_91,Carleson_93,Steinmetz_93,Milnor_06}.

In what follows we shall restrict attention to the chromatic-polynomial case
$y_\sharp = 0$, which turns out to behave in a simpler way than
$y_\sharp \neq 0$.  We must therefore study the iteration of the map
\be
   R_q(y)  \;=\;
   \left( {q-1 \over  q-2 + y} \right) ^{\! r}
 \label{def_Rq_chromatic}
\ee
with initial condition
\begin{subeqnarray}
   y_0  & = &  \infty  \\[1mm]
   y_1  & = &  0
 \label{def_init_Rq_chromatic}
\end{subeqnarray}
Provided that $q \neq 1$
(which is simply our assumption $q \neq -v_\sharp$
specialized to the chromatic-polynomial case $v_\sharp = -1$),
the map $R_q$ is a rational function of degree $r$.

We now proceed to analyze the properties of the map $R_q$
defined by \reff{def_Rq_chromatic}.
Elementary calculus proves the following two lemmas:

\begin{lemma}
   \label{lemma_B.1}
Fix $y_\sharp = 0$ and $q \neq 1$.
Then the critical points of the map $R_q$ are $2-q$ and $\infty$,
each of multiplicity $r-1$.
Moreover, there is only one critical orbit, because one critical point
maps onto the other:
\be
   2-q \;\mapsto\; \infty \;\mapsto\; 0 \;\mapsto\;
    \left( {q-1 \over  q-2} \right) ^{\! r}  \;\mapsto\;  \ldots
   \;.
\ee
\end{lemma}

The fact that there is only one critical orbit
is what makes the case $y_\sharp = 0$ simpler than $y_\sharp \neq 0$.

\begin{lemma}
   \label{lemma_B.2}
Fix $y_\sharp = 0$ and $q \neq 1$.
Then the map $R_q$ has a fixed point at $y=1$,
with multiplier $\lambda = -r/(q-1)$.
In particular, this fixed point is attractive (but not superattractive)
if $|q-1| > r$, marginal if $|q-1| = r$, and repulsive if $|q-1| < r$.
\end{lemma}

{}From Lemmas~\ref{lemma_B.1} and \ref{lemma_B.2}
we can infer the following:

\begin{corollary}
   \label{cor_B.3}
Fix $y_\sharp = 0$ and $|q-1| > r$.
Then:
\begin{itemize}
   \item[(a)]  The initial condition $y_0=\infty$ is attracted to the
      attractive fixed point at $y=1$, but without falling onto it:
      that is, $\lim\limits_{n \to\infty} R_q^n(\infty) = 1$
      but $R_q^n(\infty) \neq 1$ for all $n \ge 0$.
   \item[(b)]  The map $R_q$ has no attractive or parabolic cycles
      other than the attractive fixed point at $y=1$.
\end{itemize}
\end{corollary}

\begin{proof}
It is known \cite[Theorems~9.3.1 and 9.3.2]{Beardon_91}
that every attractive or parabolic cycle
contains a critical point within its immediate basin of attraction,
and that this critical point has an infinite forward orbit
that lies entirely within the immediate basin of attraction
and converges to the cycle without falling onto it.\footnote{
   The {\em immediate basin of attraction}\/
   of an attractive cycle $\{z_1,\ldots,z_p\}$
   is the union of the Fatou components $F_1,\ldots,F_p$
   containing the points $z_1,\ldots,z_p$ \cite[p.~104]{Beardon_91}.
   It is easy to show \cite[Theorem~6.3.1]{Beardon_91}
   that the iterates $R^{np}$ converge to $z_i$ uniformly on compact subsets
   of $F_i$;
   moreover, it follows from the linearization theorem
   for attractive (but not superattractive) cycles
   \cite[Theorem~6.3.3]{Beardon_91}
   that every point in $\bigcup\limits_{i=1}^p (F_i \setminus \{z_i\})$
   has an infinite forward orbit that lies entirely within
   $\bigcup\limits_{i=1}^p (F_i \setminus \{z_i\})$.

   Similarly, the immediate basin of attraction
   of a parabolic (= rationally indifferent) cycle
   is the union of the Fatou components that contain a petal
   at some point of the cycle \cite[p.~194]{Beardon_91}.
   Once again, $R$ maps the immediate basin of attraction into itself
   \cite[p.~124]{Beardon_91};
   the iterates $R^{np}$ converge to the cycle,
   uniformly on compact subsets of the Fatou component
   \cite[Theorems~6.5.8 and 6.5.10]{Beardon_91};
   and the iterates cannot fall onto the cycle,
   because the iterates belong to the Fatou set
   while the parabolic cycle belongs to the Julia set
   \cite[Theorem~6.5.1]{Beardon_91}.
}
Since in the present case
there are only two critical points ($2-q$ and $\infty$)
and only one critical orbit ($2-q \mapsto \infty \mapsto \ldots$),
it follows that $\infty$ and its iterates
(that is, the entire critical orbit except perhaps $2-q$)
lie within the immediate basin of attraction
of the attractive fixed point at $y=1$
and converge to it without falling onto it.
It also follows that there cannot exist
any other attractive or parabolic cycles.
\qed
\end{proof}

Part (a) of Corollary~\ref{cor_B.3} will play a central role in our argument.
Part (b) is very interesting to know, but we will not need to use it.

Let us now recall \cite[Definition~4.1.1 and Theorem~4.1.4]{Beardon_91}
that a point $z$ is called {\em exceptional}\/ for
a rational map $R$ if its backward orbit is finite;
we denote by $E(R)$ the set of all the exceptional points.
The following characterization is well known:

\begin{proposition}
{$\!\!$ \bf \protect\cite[Theorems~4.1.2 and 4.1.4]{Beardon_91} \ }
   \label{prop_B.4}
A rational map $R$ of degree $d \ge 2$ has at most two exceptional points.
Exactly one of the following possibilities holds:
\begin{itemize}
   \item[(a)]   $E(R) = \emptyset$.
   \item[(b)]   $E(R) = \{z\}$, and $z$ is a superattractive fixed point
       satisfying $R^{-1}(\{z\}) = \{z\}$.
       In this case $R$ is conjugate to a polynomial of degree $d$.
   \item[(c)]   $E(R) = \{z_1,z_2\}$, and $z_1,z_2$ are
       superattractive fixed points
       satisfying $R^{-1}(\{z_i\}) = \{z_i\}$ $[i=1,2]$.
       In this case $R$ is conjugate to the map $z \mapsto z^d$.
   \item[(d)]   $E(R) = \{z_1,z_2\}$, and $\{z_1,z_2\}$ form
       a superattractive cycle of period 2
       satisfying $R^{-1}(\{z_1\}) = \{z_2\}$ and $R^{-1}(\{z_2\}) = \{z_1\}$.
       In this case $R$ is conjugate to the map $z \mapsto z^{-d}$.
\end{itemize}
In particular, an exceptional point is always a critical point
and a critical value.
%% and the exceptional set is contained in the Fatou set.
\end{proposition}

Using this characterization of exceptional points,
we can determine the exceptional set for our maps $R_q$
as an immediate consequence of Lemma~\ref{lemma_B.1}:

\begin{lemma}
   \label{lemma_B.5}
Fix $y_\sharp = 0$ and $q \neq 1$.  Then
\be
   E(R_q)  \;=\;  \begin{cases}
                      \{0,\infty\}  & \text{\rm if $q=2$}  \\[1mm]
                      \emptyset     & \text{\rm if $q \neq 2$}
                  \end{cases}
\ee
\end{lemma}

In our case we are dealing, not with a single rational map $R$,
but with a {\em family}\/ of rational maps
$\{ R_q \}_{q \in \C \setminus \{1\}}$
that depend analytically (= holomorphically) on the complex parameter $q$.
So let us recall some of the general theory \cite{lyubichregular,avilasokal}
concerning the iteration of holomorphic families of rational maps.

The basic setup is as follows:
We are given a family $\{ R_\lambda \}_{\lambda \in \Lambda}$
of rational maps (of degree $d \ge 2$)
parametrized holomorphically by $\lambda \in \Lambda$,
where $\Lambda$ is a connected finite-dimensional complex manifold.
We are also given a holomorphically varying initial point
$Z_\lambda \in \Cbar$.
Our goal is to understand the joint dynamics of the pair
$(R_\lambda,Z_\lambda)$,
i.e.\ the behavior of the family of maps
$\lambda \mapsto R_\lambda^n(Z_\lambda)$ ($n \ge 0$).
We say that a point $\lambda_0 \in \Lambda$
is a {\em regular point}\/ for the pair $(R_\lambda,Z_\lambda)$
if the family $\{R_\lambda^n(Z_\lambda)\}$
is normal in some neighborhood of $\lambda_0$,
and an {\em irregular point}\/ otherwise.\footnote{
   We recall that if $U$ is a connected open subset of $\Lambda$,
   a family ${\cal F}$ of functions from $U$ to $\Cbar$
   is called {\em normal}\/ if every sequence of functions from ${\cal F}$
   admits a subsequence that either converges uniformly on compacts
   or else escapes to infinity uniformly on compacts.
}
We denote by $\Reg$ (resp. $\Irr$) the set of
regular (resp.\ irregular) points;
these sets are open and closed, respectively.
A {\em domain of regularity}\/ for the pair $(R_\lambda,Z_\lambda)$ is a
connected open subset of $\Reg$.

One rather trivial way for the family $\{R_\lambda^n(Z_\lambda)\}$ to be
normal is for it to reduce to a finite set of maps.
This case corresponds to $Z_\lambda$ being {\em persistently preperiodic}\/,
i.e.\ there exist $m>n \geq 0$ such that
$R_\lambda^m(Z_\lambda)=R_\lambda^n(Z_\lambda)$ for
all $\lambda \in \Lambda$.
The papers \cite{lyubichregular,avilasokal}
studied the regular and irregular sets
for the pair $(R_\lambda,Z_\lambda)$
under the assumption that $Z_\lambda$ is {\em not}\/
persistently preperiodic.
One of the simpler results from these papers ---
which will be our main tool in what follows --- is the following:

\begin{proposition}
{$\!\!$ \bf \protect\cite[Proposition~3.1]{lyubichregular}
            \protect\cite[Proposition~4.1]{avilasokal} \ }
   \label{prop.attractingcase}
Let $\lambda_0 \in \Lambda$ and $z_0 \in \Cbar$ be such that
$\lim_{n \to \infty} R^{np}_{\lambda_0}(z_0) = \zeta_{\lambda_0}$,
where $\zeta_{\lambda_0}$ is an attractive periodic point of period $p$
for $R_{\lambda_0}$.
Then:
\begin{itemize}
   \item[(a)] There exist open sets $V \ni \lambda_0$ and $W \ni z_0$
       and a holomorphic function $\zeta_\lambda$ defined for $\lambda \in V$
       and taking the given value at $\lambda_0$,
       such that $\zeta_\lambda$ is an attractive periodic point of period $p$
       for $R_\lambda$ for all $\lambda \in V$,
       and $\lim_{n \to \infty} R^{np}_{\lambda}(z) = \zeta_{\lambda}$
       uniformly on compact subsets of $V \times W$.
   \item[(b)] %% $\jointR$ is jointly regular at $(\lambda_0,z_0)$.
       %% In particular,
       $\lambda_0$ is a regular point for every pair
       $(R_\lambda,Z_\lambda)$ satisfying $Z_{\lambda_0} = z_0$.
   \item[(c)]  Suppose that $Z_\lambda$ is a holomorphic path satisfying
       $Z_{\lambda_0} = z_0$, and that $U \ni \lambda_0$ is a
       domain of regularity for $(R_\lambda,Z_\lambda)$.
       Then $\lim_{n \to \infty} R^{np}_\lambda(Z_\lambda) \equiv \zeta_\lambda$
       exists for all $\lambda \in U$ [uniformly on compact subsets of $U$]
       and satisfies $R_\lambda^p(\zeta_\lambda) = \zeta_\lambda$.
       Moreover, if $Z_\lambda$ is not persistently preperiodic,
       then $\zeta_\lambda$ remains attractive
       [i.e.\ satisfies $|(DR_\lambda^p)(\zeta_\lambda)| < 1$]
       and of period $p$, for all $\lambda \in U$.
\end{itemize}
\end{proposition}

Let us now apply Proposition~\ref{prop.attractingcase} to our family
$R_\lambda = R_q$ with initial condition $Z_\lambda = \infty$.
We first need the following lemma:

\begin{lemma}
   \label{lem_persistent}
For the family $R_\lambda = R_q$, the initial condition $Z_\lambda = \infty$
is not persistently preperiodic.
\end{lemma}

\begin{proof}
By Corollary~\ref{cor_B.3}(a) we know that for $|q-1|>r$
the initial condition $Z_\lambda = \infty$ is attracted to the
attractive fixed point at $y=1$ but without falling onto it.
Therefore $Z_\lambda = \infty$ is not persistently preperiodic.
\qed
\end{proof}

With this lemma in hand, we can apply Proposition~\ref{prop.attractingcase}
to conclude the following:

\begin{corollary}
   \label{cor.irregular}
Fix $y_\sharp = 0$.  Then every point of the circle $|q-1| = r$
is an irregular point for the family $\{R_q\}$ with initial condition $\infty$.
\end{corollary}

\begin{proof}
Consider any $q_1$ satisfying $|q_1-1| = r$.
Suppose that $q_1$ is a regular point,
and let $U \ni q_1$ be a domain of regularity.
Then $U$ contains a point $q_0$ satisfying $|q_0 - 1| > r$,
and Corollary~\ref{cor_B.3}(a) guarantees that
$\lim_{n \to\infty} R_{q_0}^n(\infty) = 1$.
But then Proposition~\ref{prop.attractingcase}(c)
and Lemma~\ref{lem_persistent} imply that
the fixed point at $1$ remains attractive whenever $q \in U$,
which contradicts the fact (Lemma~\ref{lemma_B.2})
that it is repulsive whenever $|q-1| < r$.
It follows that $q_1$ must be an irregular point.
\qed
\end{proof}

Next we use a result guaranteeing that the joint dynamics is ``wild''
in the neighborhood of every irregular point.
First, a definition:
If $U$ is a connected open subset of $\Lambda$,
we call a function $f \colon\, U \to \Cbar$
{\em persistently exceptional}\/ in case $f(\lambda)$
is an exceptional point for $R_\lambda$ for all $\lambda \in U$.
We then have:

\begin{prop}
{$\!\!$ \bf \protect\cite[Proposition~3.5]{lyubichregular}
            \protect\cite[Proposition~3.9]{avilasokal} \ }
 \label{prop.nonexceptional}
Let $U$ be a connected open subset of $\Lambda$
having a nonempty intersection with the irregular set~$\Irr$,
and let $f \colon\, U \to \Cbar$ be a holomorphic function
that is not persistently exceptional.
Then the analytic varieties
\begin{equation}
   {\cal S}^f_n  \;=\;
   \{ \lambda \in U \colon\, R_\lambda^n(Z_\lambda) = f(\lambda) \}
\end{equation}
accumulate everywhere on $\Irr \cap U$
[that is, $\liminf\limits_{n\to\infty} {\cal S}^f_n \supseteq \Irr \cap U$].
\end{prop}

We are now ready to prove Theorem~\ref{thm.leaf-joined.APP}:

\proofof{Theorem~\ref{thm.leaf-joined.APP}}
By Lemma~\ref{lemma.iteration} with $v_\sharp = -1$,
we have $P_{G_n}(q) = 0$ if and only if $R_q^n(\infty) = 1-q$.
We therefore apply Proposition~\ref{prop.nonexceptional}
to the ``target function'' $f(q) = 1-q$
with the initial condition $Z_q = \infty$.
By Lemma~\ref{lemma_B.5} there are {\em no}\/ persistently exceptional
functions for our family;  in particular, $f(q) = 1-q$
is not persistently exceptional.
Combining Corollary~\ref{cor.irregular}
and Proposition~\ref{prop.nonexceptional},
we complete the proof of Theorem~\ref{thm.leaf-joined.APP}.
\qed

\bigskip

Let us conclude by making a few further remarks concerning the map $R_q$
and the chromatic roots of the graphs $G_n$.
Note first that the map $R_q(y) = [(q-1)/(q-2+y)]^r$ is conjugate,
under the M\"obius transformation $z = 1 + y/(q-2)$, to the map
\be
   \widetilde{R}_w(z)  \;=\;  1 \,+\, {w \over z^r}
 \label{def.Rtildew}
\ee
where
\be
   w  \;=\;  {(q-1)^r \over (q-2)^{r+1}}
   \;.
 \label{def.w_q}
\ee
The family of maps $\widetilde{R}_w$,
parametrized by $w \in \C \setminus \{0\}$,
has been studied by several authors
\cite{Milnor_quadratic,Milnor_bicritical,Bamon_99,Bobenrieth_tesis,%
Bobenrieth_1,Bobenrieth_2,Bobenrieth_3};
it is the unique (modulo conjugation) one-parameter family of
degree-$r$ rational maps with two $(r-1)$-fold critical points, 
one of which maps onto the other
(this follows from \cite[Lemma~1.1]{Milnor_bicritical}).
Curiously, the recursion \reff{def.Rtildew}
arises also in the study of the hard-core lattice-gas partition function
(= independence polynomial) for a rooted tree of branching factor $r$
\cite[Example~3.6]{Scott-Sokal_lovasz}.

An easy calculation shows that the map $\widetilde{R}_w$
possesses a fixed point of multiplier $\lambda$ if and only if
\be
   w  \;=\; - \, {\lambda \, r^r \over (\lambda+r)^{r+1}}
   \;,
 \label{eq.Rw.fixed.w}
\ee
and in this case the fixed point lies at
\be
   z  \;=\;  {r \over \lambda+r}
   \;.
 \label{eq.Rw.fixed.z}
\ee
Combining \reff{eq.Rw.fixed.w} with \reff{def.w_q}
and solving for $q$, we find $r+1$ solutions:
one of them, $q = 1-r/\lambda$,
corresponds in the map $R_q$ to the fixed point at $y=1$
of multiplier $\lambda = -r/(q-1)$;
but the others are new.  For instance, for $r=2$ we have
\be
   q  \;=\; 
   {8 - 6\lambda - \lambda^2 \pm (2+\lambda) \sqrt{\lambda(8+\lambda)}
    \over
    8
   }
 \label{eq.Rq.fixed.q}
\ee
with fixed points at
\be
   y
   \;=\;
   {\lambda \, (2-q) \over 2+\lambda}
   \;=\;
   {\lambda \,
    [4+\lambda \mp \sqrt{\lambda(8+\lambda)}]
    \over
    8
   }
 \;.
 \label{eq.Rq.fixed.y}
\ee
For $r \ge 3$ the formulae become much more complicated.

Similarly one can search for periodic orbits of higher period $p$
with a given multiplier $\lambda$.
For at least one case the formulae are simple:
for $r=2$ the map $\widetilde{R}_w$ has an orbit of period $p=2$
with multiplier $\lambda$ if and only if
\be
   w  \;=\;  {4 \over \lambda}
   \;,
 \label{eq.Rw.period2.w}
\ee
and in this case the orbit lies at
\be
   z  \;=\;  {2 \pm 2 \sqrt{1-\lambda} \over \lambda}
   \;.
 \label{eq.Rw.period2.z}
\ee
[Here the case $\lambda=1$, $w=4$, $z=2$
 is actually a fixed point of multiplier $-1$:
 cf.\ \reff{eq.Rw.fixed.w}/\reff{eq.Rw.fixed.z} with $\lambda=-1$.]
The corresponding values of $q$ and $y$ can then be obtained,
but the formulae are messy.

\begin{figure}[t]
\begin{center}
\includegraphics[width=0.6\textwidth]{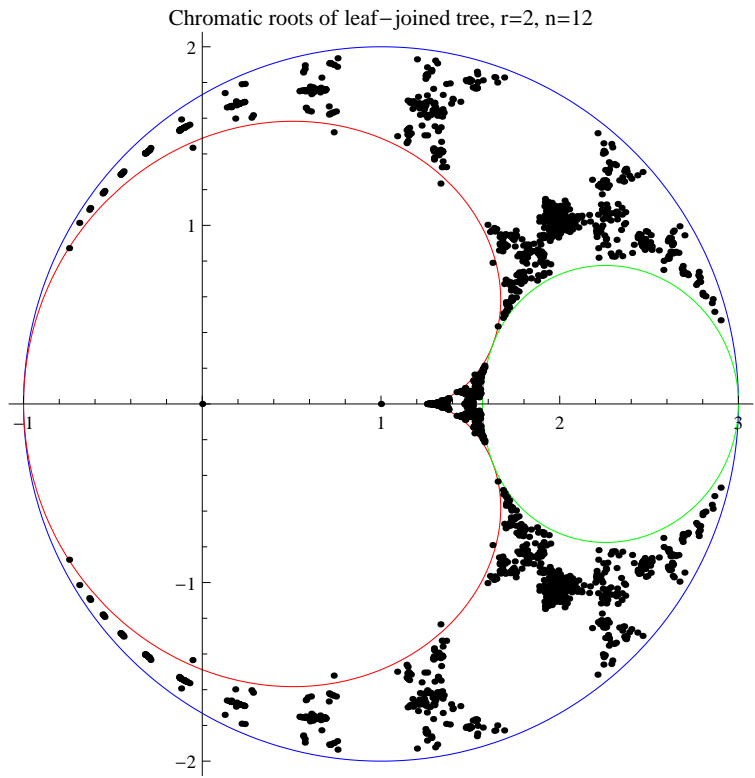}
\end{center}
\caption{
   Chromatic roots of the leaf-joined tree $G_n^r$ with $r=2$ and $n=12$.
   The blue circle represents the locus $|q-1|=2$
   where the fixed point at $y=1$ becomes marginal.
   The red cardioid represents the locus \reff{eq.Rq.fixed.q}
   with $|\lambda|=1$,
   where the fixed point \reff{eq.Rq.fixed.y} becomes marginal.
   The green egg-shaped curve represents the $q$-plane locus corresponding
   to \reff{eq.Rw.period2.w} with $|\lambda|=1$,
   where the period-2 orbit becomes marginal.
}
\label{fig.treeroots}
\end{figure}

In Figure~\ref{fig.treeroots} we plot the chromatic roots
of the graph $G_n^r$ with $r=2$ and $n=12$.\footnote{
   We first used {\sc Mathematica} to compute the polynomials $P_n(q)$,
   with exact integer coefficients,
   using the recursion
   \reff{eq.recursion.AnBn}/\reff{eq.recursion.AnBn.initial}
   with $v_\sharp = -1$.
   We then used the program {\sc MPSolve} \cite{Bini_MPSolve,Bini_00}
   to compute the zeros of $P_n$ to 30-digit accuracy.
   We were able to do this for $n \le 12$.
   The computation of the polynomials is extremely quick ---
   about two minutes for $n = 12$,
   on an Intel Core i7-2600 CPU processor running at 3.4~GHz ---
   and could easily have been pushed to larger $n$.
   The computation of the zeros is, however, much slower:
   approximately 0.8~hour for $n=10$, 3~hours for $n=11$,
   and 67~hours for $n=12$.
   This computation could be speeded significantly by coding
   the recursion \reff{eq.recursion.AnBn}/\reff{eq.recursion.AnBn.initial}
   directly as a user-defined C~program as explained
   in \cite[Section~6]{Bini_MPSolve};
   but we did not attempt to do this.
}
The blue circle represents the locus $|q-1|=2$
where the fixed point at $y=1$ becomes marginal.
The red cardioid represents the locus \reff{eq.Rq.fixed.q} with $|\lambda|=1$,
where the fixed point \reff{eq.Rq.fixed.y} becomes marginal;
the cusp of this cardioid lies at $q = 5/4$.
The green egg-shaped curve represents the $q$-plane locus corresponding
to \reff{eq.Rw.period2.w} with $|\lambda|=1$,
where the period-2 orbit becomes marginal.
The convergence of the chromatic roots to the circle $|q-1|=2$,
as asserted in Theorem~\ref{thm.leaf-joined.APP}, seems quite slow
(perhaps like $1/n$).
We expect that by a similar argument one can prove convergence
of chromatic roots to the red and green curves,
but again this convergence seems quite slow.

We see from Figure~\ref{fig.treeroots} that
all the chromatic roots lie in the region $|q-1| < 2$;
and we have confirmed this for $n \le 12$.
Let us formulate this as an explicit conjecture for general $r$:

\begin{conjecture}
  \label{conj.leaf-joined}
For every $r \ge 2$ and $n \ge 1$,
all the chromatic roots of the graph $G_n^r$
(the leaf-joined tree of branching factor $r$ and height $n$)
lie in the disc $|q-1| < r$.
\end{conjecture}

\noindent
This conjecture can be rephrased as saying that the region $|q-1| \ge r$
where the fixed point at $y=1$ is attractive or marginal is
free of chromatic roots.

For $r=2$ it also appears that no chromatic roots lie on or inside
the green egg-shaped curve, i.e.\ in the region where the period-2 orbit
is attractive or marginal.
We have confirmed this also for $n \le 12$.

% {\bf Later we can remark about how the changes of variable and parameter
% $z = 1 + y/(q-2)$ and $w = (q-1)^r/(q-2)^{r+1}$ map $R_q(y)$ into
% $\widetilde{R}_w(z) = 1 + w/z^r$, which is exactly the recursion for
%   the hard-core lattice gas on a tree of branching factor $r$ ---
%    cf.\ Scott--Sokal, Example~3.6 for this.
% 
%  Also, can something similar be done for soft-core?????
% 
%    I think the general fact here is that there is a unique (modulo
%    conjugation) one-parameter family of degree-$r$ rational maps
%    with two $(r-1)$-fold critical points, one of which maps onto
%    the other --- and it is given by $\widetilde{R}_w(z) = 1 + w/z^r$
%    for $w \in \C \setminus \{0\}$.  What about for degree-$r$ rational maps
%    with two $(r-1)$-fold critical points, irrespective of whether they
%    map into each other?  Is this a 2-parameter family, and can it be
%    parametrized so that $w_2 = 0$ corresponds to the first family?
%    I think this is studied for $r=2$ in Milnor's paper on quadratic
%    rational maps, and for general $r$ in Milnor's paper on
%    rational maps with two critical points.
% }

\section*{Acknowledgments}

%% We wish to thank ???????????????????
%% for valuable conversations and correspondence.

We are extremely grateful to
the Isaac Newton Institute for Mathematical Sciences,
University of Cambridge, for support during the programme on
Combinatorics and Statistical Mechanics (January--June 2008),
where this work was mostly carried out.
We also thank the Laboratoire de Physique Th\'eorique
at the \'Ecole Normale Sup\'erieure for hospitality in April--June 2011
and April--July 2013.

This research was supported in part by
U.S.\ National Science Foundation grant PHY--0424082,
and by Australian Research Council Discovery Project DP110101596.

%%% THIS IS THE BIBLIOGRAPHY

%%\bibliographystyle{acm2url}
%%\bibliography{serpar}

\end{document}